\documentclass{amsart}
\usepackage{graphicx}
\usepackage{amsfonts}
\usepackage{amscd}
\usepackage{amssymb}
\usepackage{alltt}
\usepackage{multicol}
\usepackage{amsmath}
\usepackage{amsthm}
\usepackage{amscd}
\usepackage[numbers]{natbib}
\usepackage{rotating}
\usepackage{floatpag}
 \rotfloatpagestyle{empty}
\usepackage{amsthm}
\usepackage{graphicx}
\usepackage{multind}
\usepackage{times}

\usepackage{verbatim}
\usepackage{latexsym}
\usepackage{amsfonts}
\usepackage{amsmath}
\usepackage{crop}
\usepackage{fancyhdr}
\usepackage{txfonts}
\usepackage{setspace}
\usepackage{ellipsis} % 
\usepackage{wrapfig}
\usepackage{listings}

\usepackage{tikz} % Needs pgf version >= 2.10.
\usetikzlibrary{chains,shapes,arrows,trees,matrix,positioning,backgrounds,fit,calc,fadings}

\usepackage[mathscr]{euscript} % powerset.
\usepackage{pifont} %ding
\usepackage[displaymath]{lineno}

\def\displayallproof{t} 
\def\displayrating{f}
\def\verbose{t}
\def\tikzfig#1#2#3{%
\begin{figure}[htb]%
  \centering
\begin{tikzpicture}#3
\end{tikzpicture}
  \caption{#2}
  \label{fig:#1}%
\end{figure}%
}

  {~\par\phantom{!}\endgroup\bigskip}
\def\indy#1#2{\index{index/#1}{#2}\relax}
\def\hide#1{}
\def\swallowed{\relax}
\def\swallow#1\swallowed{}
\newenvironment{iproved}{}{}

\def\hideproof{\renewenvironment{iproved}{%
   \centerline{\it -- Proof Proofed --}
  
  \renewenvironment{enumerate}{}{}
  \def\item{\relax}
  \catcode13=12
  \swallow
}{}}
\def\showproof{\renewenvironment{iproved}{\begin{proof}}{\end{proof}}}
\def\resetproved{\if\displayallproof t\showproof\else\hideproof\fi}

\def\rating#1{\if\displayrating t%
  {{\textsc {[rating={\ensuremath {#1}}].\ }}}\else{}\fi}

\def\oldrating#1{\if\displayrating t%
  {{\textsc {[former rating={\ensuremath {#1}}].\ }}}\else{}\fi}

\def\ifcverbose#1#2{\if\verbose t{{#1}}\else{#2}\fi}
\def\dcg#1#2{{\if\verbose t%
  {{\tt{[DCG-#1]}}\indy{References}{ZC{#2 #1}@{DCG-#1}|page{#2}}}\else{}\fi}}
\def\tlabel#1{\label{#1}\if\verbose t{{\tt [#1].\ }%
   \indy{References}{#1|itt}}\else{}\fi}
\def\formal#1{\relax }

\def\mar#1{}
\def\emptyset{\varnothing}
\def\|{\hbox{\ensuremath{\hspace{0.1em}|\hspace{-0.1em}|\hspace{0.1em}}}}
\def\mid{\ :\ }

\def\norm#1#2{\hbox{\ensuremath{\|#1\unskip-\unskip{#2}\|}}}
\def\normo#1{{\|#1\|}}

\def\Ra {\Rightarrow}
\def\=#1{\accent"16 #1}

\newcommand{\ring}[1]{\mathbb{#1}}
\def\v{{\mathbf v}}
\def\u{{\mathbf u}}
\def\w{{\mathbf w}}

\def\p{{\mathbf p}}
\def\q{{\mathbf q}}

\def\op#1{{\operatorname{#1}}}

\def\arc{\operatorname{arc}}

\def\card{\op{card}}

\crop
\theoremstyle{plain}
\newtheorem{theorem}[equation]{Theorem}
\newtheorem{theorem*}[equation]{Theorem$^*$}
\newtheorem{lemma}[equation]{Lemma}
\newtheorem{lemma*}[equation]{Lemma$^*$}

\newtheorem{corollary}[equation]{Corollary}

\theoremstyle{definition}
\newtheorem{definition}[equation]{Definition}
\theoremstyle{remark}
\newtheorem{remark}[equation]{Remark}

\renewcommand\footnotemark{}
\begin{document}
\title
    {Packings of Regular Pentagons in the Plane}
\author{Thomas Hales and W\"oden Kusner}
%\date{September 11, 2016}
\thanks{Research supported by NSF grant 1104102}

\begin{abstract}  
We show that every packing of congruent regular pentagons in the Euclidean plane has
density at most $(5-\sqrt5)/3\approx 0.92$.     More specifically,
this article proves the pentagonal ice-ray
conjecture of Henley (1986), and Kuperberg and Kuperberg (1990), 
which asserts that an optimal packing of congruent regular pentagons in the plane is a double lattice,
formed by aligned vertical columns of upward pointing pentagons alternating
with aligned vertical columns of downward pointing pentagons.  The strategy is based on estimates
of the areas of Delaunay triangles.  Our strategy reduces the pentagonal ice-ray conjecture to
area minimization problems that involve at most four Delaunay triangles.  These minimization problems
are solved by computer.  The computer-assisted portions of the proof use techniques such as
interval arithmetic, automatic differentiation, and a meet-in-the-middle algorithm.
\end{abstract}

 \lhead{Hales and Kusner}
\rhead{Pentagon Packings}

\parskip=\baselineskip

 \maketitle

%%%

%%%%%%%%%%%%%%%%%%%%%%%%%%%%%%%%%%%%%%%%%
%%% FRONT

% The Optimal Packing of Regular Pentagons in the Plane, I
% Tex file started June 23, 2015.

% 

%Title: Packings of Regular Pentagons in the Plane,
% draft Nov, 2015.
% version 2, June 2016
%Authors: Thomas Hales, Woden Kusner.

% global spellcheck (done Aug 27 2016).
% Notation: kappa = c = rho = cos pi /5.
% global check of calculations in calcs.ml completed 2016/9/11.

\def\area{\op{area}}
\def\areta{\op{area}_\eta}
\def\ao{a_\text{min}}
\def\acrit{a_\text{crit}}
\def\loc{\op{loc}}

\def\Ra{\Rightarrow}
\def\nRa{\nRightarrow}
\def\rab{\Ra_{b}}

\def\C{\mathcal C}
\def\S{\mathcal S}
\def\N{\mathcal N}
\def\M{\mathcal M}
\def\T{\mathcal T}
\def\PD{\Psi D}.
\def\cong{\equiv}
\def\r{{\mathbf r}}
\def\s{{\mathbf s}}
\def\bl{\ell}
\def\epsM{\epsilon_{\M}}
\def\epsN{\epsilon_{\N}}
\def\c{{\mathbf c}}
\def\K{\op{K}}

\def\dx#1{d_{#1}} 
\def\dd#1#2{\norm{\c_{#1}}{\c_{#2}}}
\def\tb{\underline{b}}
\def\tn{\bar n}

%for debugging:
%\def\libel#1{{\text{\sc [#1]~}}\label{#1}}
%\def\rif#1{(\ref{#1}-{\text{\sc #1})}}
\def\libel#1{\label{#1}}
\def\rif#1{\ref{#1}}

% tikz.
\def\smalldot#1{\draw[fill=black] (#1) node [inner sep=0.8pt,shape=circle,fill=black] {}}
\def\graydot#1{\draw[fill=gray] (#1) node [inner sep=1.3pt,shape=circle,fill=gray] {}}
\def\whitedot#1{\draw[fill=gray] (#1) node [inner sep=1.3pt,shape=circle,fill=white,draw=black] {}}
\tikzset{dartstyle/.style={fill=black,rotate=-90,inner sep=0.7pt,dart,shape border uses incircle}}
\tikzset{grayfatpath/.style={line width=1ex,line cap=round,line join=round,draw=gray}}

% parameters (x,y) center coords, theta vertex angle, rho radius in cm.
\def\pent#1#2#3#4{%
\draw[red] (#1,#2) + (#3:#4cm) -- + (#3+72:#4cm) -- 
+(#3+144:#4cm) -- +(#3+216:#4cm) -- + (#3+288:#4cm) -- cycle
}

\def\pentpink#1#2#3#4#5{%
\draw[#5] (#1,#2) + (#3:#4cm) -- + (#3+72:#4cm) -- 
+(#3+144:#4cm) -- +(#3+216:#4cm) -- + (#3+288:#4cm) -- cycle
}

% parameters (x,y) center coords, theta vertex angle.
\def\pen#1#2#3{%
\draw[red] (#1,#2) + (#3:1cm) -- + (#3+72:1cm) -- 
  +(#3+144:1cm) -- +(#3+216:1cm) -- + (#3+288:1cm) -- cycle
}

% parameters (x,y) center coords, theta vertex angle.
\def\penp#1#2{%
\draw[red] (#1) + (#2:1cm) -- + (#2+72:1cm) -- 
  +(#2+144:1cm) -- +(#2+216:1cm) -- + (#2+288:1cm) -- cycle
}

\def\threepent#1#2#3#4#5#6#7#8#9{%
\pen{#1}{#2}{#3};
\pen{#4}{#5}{#6};
\pen{#7}{#8}{#9};
\draw[blue] (#1,#2) -- (#4,#5) -- (#7,#8) -- cycle
}

\def\threepentnoD#1#2#3#4#5#6#7#8#9{%
\pen{#1}{#2}{#3};
\pen{#4}{#5}{#6};
\pen{#7}{#8}{#9}
}

\def\blankfig#1{\tikzfig{#1}{Insert graphic}
{
[scale=1]
\draw(0,) circle(1cm);
}}

\centerline{\it We dedicate this article to W. Kuperberg.}

\section{Introduction}\label{sec:intro} 

A fundamental problem in discrete geometry is to determine the highest
density of a packing in Euclidean space by congruent copies of a
convex body $C$.  For example, when $C$ is a ball of given radius,
this problem reduces to the sphere-packing problem in Euclidean space.
Besides a sphere, the simplest shape to consider is a regular polygon
$C$ in the plane.  If the regular polygon is an equilateral
triangle, square, or hexagon, then copies of $C$ tile the plane.  In
these cases, the packing problem is trivial.  The first nontrivial
case is the packing problem for congruent regular pentagons.  This
article solves that problem.

Henley and Kuperberg and Kuperberg have conjectured that the densest
packing of congruent regular pentagons in the plane is achieved by a
double-lattice arrangement: verticals column of aligned pentagons
pointing upward, alternating with vertical columns of aligned
pentagons pointing downward (Figure~\rif{fig:double-lattice})
\cite{Kup} \cite[p.801]{henley}.  This packing of pentagons is called
the {\it pentagonal ice-ray} in Dye's book on Chinese lattice designs
\cite{dye}.  Two plates (Y3b and Y3c) in Dye's book depict the
pentagonal ice-ray, originating from Chengdu, China around 1900 CE.
% also Y, symmetry Ice-ray.

 We call this the {\it pentagonal ice-ray  conjecture}.
This packing has density
\[
\frac{5 - \sqrt{5}}3 \approx 0.921311.
\] % Checked 2016/2/18. In Vallentin, p.3.  calcs.ml.
Before our work, the best known bound on the density of packings of
regular pentagons was $0.98103$, obtained through the representation
theory of the group of isometries of the plane \cite{Val}.  Our
research is a continuation of W. Kusner's thesis \cite{Kus}, which
proves the local optimality of the pentagonal ice-ray.
% 0.98103 page 27 of Vallentin. Checked 2016/2/18.

As far as we know, our methods are adequate for the solution of other
related problems in geometric optimization.  The limiting factor seems
to be the availability of sufficient computer resources.

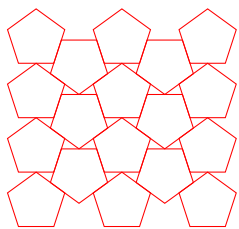
\tikzfig{double-lattice}{The pentagonal ice-ray.  All figures show
pentagons in red and Delaunay triangles in blue.}
{
\pent{0.0}{0.0}{-90}{0.4};  % s = 0.4
\pent{0.0}{0.724}{-90}{0.4};  % (1+kappa)*s
\pent{0.0}{2*0.724}{-90}{0.4};  
\pent{0.57}{-0.3854}{90}{0.4};  % 3kappa*sigma*s,(3sigma^2-2)*s
\pent{0.57}{-0.3854+0.724}{90}{0.4}; 
\pent{0.57}{-0.3854+2*0.724}{90}{0.4};  
\pent{0.57}{-0.3854+3*0.724}{90}{0.4};  
\pent{-0.57}{-0.3854}{90}{0.4};  % 3kappa*sigma*s,(3sigma^2-2)*s
\pent{-0.57}{-0.3854+0.724}{90}{0.4}; 
\pent{-0.57}{-0.3854+2*0.724}{90}{0.4};  
\pent{-0.57}{-0.3854+3*0.724}{90}{0.4};  
\pent{2*0.57}{0.0}{-90}{0.4};  % s = 0.4
\pent{2*0.57}{0.724}{-90}{0.4};  % (1+kappa)*s
\pent{2*0.57}{2*0.724}{-90}{0.4}; 
\pent{3*0.57}{-0.3854}{90}{0.4};  % 3kappa*sigma*s,(3sigma^2-2)*s
\pent{3*0.57}{-0.3854+0.724}{90}{0.4}; 
\pent{3*0.57}{-0.3854+2*0.724}{90}{0.4};  
\pent{3*0.57}{-0.3854+3*0.724}{90}{0.4};   
}

%\section{Introduction}

This article gives a computer-assisted proof of the pentagonal ice-ray
conjecture.  The proof appears at the end of Section~\rif{sec:nonobtuse}.

\begin{theorem}\libel{thm:main}  
  No packing of congruent regular pentagons in the Euclidean plane has
  density greater than that of the pentagonal ice-ray.  The pentagonal
  ice-ray is the unique periodic packing of congruent regular
  pentagons that attains optimal density.
\end{theorem}

In this article, we consider packings of congruent regular pentagons
in the Euclidean plane.  Density is scale-invariant.  Without loss of
generality, we may assume that all pentagons are regular pentagons of
fixed circumradius $1$.  The inradius of each pentagon is $\kappa:=
\cos (\pi/5) = (1+\sqrt{5})/4 \approx 0.809$. We set $\sigma :=
\sin(\pi/5) \approx 0.5878$.  The length of each pentagon edge is
$2\sigma$.  (All of the numerical calculations in this article have
been checked in a file {\tt calcs.ml}, which is available for download
from the project code repository \cite{Git}.)
% checked 2016/2/18 in calcs.ml.

All pentagon packings will be assumed to be saturated; that is, no
further regular pentagons can be added to the packing without overlap.
The assumption of saturation can be made without loss of generality,
because our ultimate aim is to give upper bounds on the density of
pentagon packings, and the saturation of a packing cannot decrease its
density.

We form the Delaunay triangulations of the pentagon packings.  The
vertices of the triangles are alway taken to be the centers of the
pentagons.  The saturation hypothesis implies that no Delaunay
triangle has circumradius greater than $2$.  This property of Delaunay
triangles is crucial.   Every edge of a Delaunay
triangle has length at most $4$.

A Delaunay triangle in a pentagon packing has edge lengths at least
$2\kappa$.  This minimum Delaunay edge length is attained exactly when
the the two pentagons have a full edge in common.

For most of the article, we consider a fixed saturated packing and its
Delaunay triangulation.  Generally, unless otherwise stated,
every triangle is a  Delaunay triangle.
Statements of lemmas and theorems implicitly
assume this fixed context.

Initially, pentagons in a packing play two roles: they constrain the
shapes of the Delaunay triangles and they carry mass for the density.
We prefer to we change our model slightly so pentagons are only used
to constrain the shapes of triangles.  We replace the mass of each
pentagon by a small, massive, circular disk (each of the same small
radius) centered at the center of the pentagon, and of uniform density
and the same total mass as the pentagon.  In this model, each Delaunay
triangle contains exactly one-half the mass of a pentagon.  By
distributing the pentagon mass uniformly among the Delaunay triangles,
we may replace density maximization with Delaunay triangle area
minimization.

We write 
\[
\acrit := \frac{3}{2}{ \sigma \kappa(1+\kappa)} \approx 1.29036
\] % checked 2016/2/18 in calcs.ml
for the common (critical) area of every Delaunay triangle in the 
pentagonal ice-ray.

\section{Clusters of Delaunay Triangles}

The area of a Delaunay triangle in a saturated pentagon packing can be
smaller than $\acrit$.  Our strategy for proving the pentagonal
ice-ray conjecture is to collect triangles into finite clusters such
that the average area over each cluster is at least $\acrit$.

We say that a Delaunay triangle (in any pentagon packing) is {\it
  subcritical} if its area is at most $\acrit$.  We will obtain a
lower bound $\ao := 1.237$ on the area of a nonobtuse Delaunay
triangle (Lemma~\rif{lemma:a0}).  This is a very good bound.  It is
very close to the numerically smallest achievable area, which is
approximately $1.23719$.\footnote{In the notation of the appendix, the
  numerical minimum is achieved by a pinwheel with parameters
  $\alpha=\beta=0$ and $x_\gamma\approx 0.16246$.}  We write $\epsN :=
\acrit - \ao \approx 0.05336$, for the difference between the desired
bound $\acrit$ on averages of triangles and the bound $\ao$ for a
single nonobtuse triangle.  Let $\epsM = 0.008$.
% Numerical page 8. of nonobtuse notes.
% numerical min checked 2016/2/18 in Mathematica.

\subsection{examples}

While reading this article, it is useful to carry along a series of examples,
illustrated in Figures~\rif{fig:KKtriangle} -- \rif{fig:obtuse}.
Otherwise, later definitions such as the modified area function $b(T)$
and the construction of clusters might appear to be unmotivated.  Some
of these examples serve as counterexamples to naive approaches to this
problem.  These different examples can interact with one another in
potentially complicated ways.  The proof of the main theorem must sort
through these interactions.

(Another essential ingredient in the understanding of the proof is a
rather large body of computer code that is used to carry out the
computer-assisted portions of the proof.  We will have more to say
about this later.)

% adapted from lattice figure.
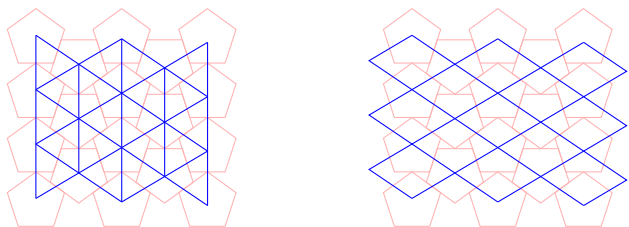
\tikzfig{KKtriangle}{Ice-ray triangles and ice-ray dimers.
  All Delaunay triangles in the pentagonal ice-ray are congruent and
  have area $\acrit$.  We call them ice-ray triangles. 
  The pentagonal ice-ray can be partitioned into pairs of Delaunay
  triangles (called ice-ray dimers) in which the two triangles in each
  dimer share their common longest edge.  
}
{
\begin{scope}
\pentpink{0.0}{0.0}{-90}{0.4}{red!30};  % s = 0.4
\pentpink{0.0}{0.724}{-90}{0.4}{red!30};  % (1+kappa)*s
\pentpink{0.0}{2*0.724}{-90}{0.4}{red!30};  
\pentpink{0.57}{-0.3854}{90}{0.4}{red!30};  % 3kappa*sigma*s,(3sigma^2-2)*s
\pentpink{0.57}{-0.3854+0.724}{90}{0.4}{red!30}; 
\pentpink{0.57}{-0.3854+2*0.724}{90}{0.4}{red!30};  
\pentpink{0.57}{-0.3854+3*0.724}{90}{0.4}{red!30};  
\pentpink{-0.57}{-0.3854}{90}{0.4}{red!30};  % 3kappa*sigma*s,(3sigma^2-2)*s
\pentpink{-0.57}{-0.3854+0.724}{90}{0.4}{red!30}; 
\pentpink{-0.57}{-0.3854+2*0.724}{90}{0.4}{red!30};  
\pentpink{-0.57}{-0.3854+3*0.724}{90}{0.4}{red!30};  
\pentpink{2*0.57}{0.0}{-90}{0.4}{red!30};  % s = 0.4
\pentpink{2*0.57}{0.724}{-90}{0.4}{red!30};  % (1+kappa)*s
\pentpink{2*0.57}{2*0.724}{-90}{0.4}{red!30}; 
\pentpink{3*0.57}{-0.3854}{90}{0.4}{red!30};  % 3kappa*sigma*s,(3sigma^2-2)*s
\pentpink{3*0.57}{-0.3854+0.724}{90}{0.4}{red!30}; 
\pentpink{3*0.57}{-0.3854+2*0.724}{90}{0.4}{red!30};  
\pentpink{3*0.57}{-0.3854+3*0.724}{90}{0.4}{red!30};   
\coordinate (T) at (0.5706,-0.38541);
\coordinate (U) at (0,0.7236);
\coordinate (T') at ($(T)+(U)$);
\draw[blue] (0,0) -- ++ ($2*(U)$);
\draw[blue] ($-1.0*(T')$) -- ++ ($3*(U)$);
\draw[blue] (T) -- ++ ($3*(U)$);
\draw[blue] ($(T)+(T')$) -- ++ ($2*(U)$);
\draw[blue] ($(T)+(T')+(T)$) -- ++ ($3*(U)$);
\draw[blue] ($-1.0*(T)$) -- ++ ($2.0*(T)$);
\draw[blue] ($-1.0*(T) + (U)$) -- ++ ($4.0*(T)$);
\draw[blue] ($-1.0*(T) + 2.0*(U)$) -- ++ ($4.0*(T)$);
\draw[blue] ($1.0*(T) + 3.0*(U)$) -- ++ ($2.0*(T)$);
\draw[blue] ($(T)$) -- ++ ($2.0*(T')$);
\draw[blue] ($1.0*(T) +1.0*(U) -2.0 *(T')$) -- ++ ($4.0*(T')$);
\draw[blue] ($1.0*(T) +2.0*(U) -2.0 *(T')$) -- ++ ($4.0*(T')$);
\draw[blue] ($1.0*(T) +3.0*(U) -2.0 *(T')$) -- ++ ($2.0*(T')$);
\end{scope}
\begin{scope}[xshift=5cm]
\pentpink{0.0}{0.0}{-90}{0.4}{red!30};  % s = 0.4
\pentpink{0.0}{0.724}{-90}{0.4}{red!30};  % (1+kappa)*s
\pentpink{0.0}{2*0.724}{-90}{0.4}{red!30};  
\pentpink{0.57}{-0.3854}{90}{0.4}{red!30};  % 3kappa*sigma*s,(3sigma^2-2)*s
\pentpink{0.57}{-0.3854+0.724}{90}{0.4}{red!30}; 
\pentpink{0.57}{-0.3854+2*0.724}{90}{0.4}{red!30};  
\pentpink{0.57}{-0.3854+3*0.724}{90}{0.4}{red!30};  
\pentpink{-0.57}{-0.3854}{90}{0.4}{red!30};  % 3kappa*sigma*s,(3sigma^2-2)*s
\pentpink{-0.57}{-0.3854+0.724}{90}{0.4}{red!30}; 
\pentpink{-0.57}{-0.3854+2*0.724}{90}{0.4}{red!30};  
\pentpink{-0.57}{-0.3854+3*0.724}{90}{0.4}{red!30};  
\pentpink{2*0.57}{0.0}{-90}{0.4}{red!30};  % s = 0.4
\pentpink{2*0.57}{0.724}{-90}{0.4}{red!30};  % (1+kappa)*s
\pentpink{2*0.57}{2*0.724}{-90}{0.4}{red!30}; 
\pentpink{3*0.57}{-0.3854}{90}{0.4}{red!30};  % 3kappa*sigma*s,(3sigma^2-2)*s
\pentpink{3*0.57}{-0.3854+0.724}{90}{0.4}{red!30}; 
\pentpink{3*0.57}{-0.3854+2*0.724}{90}{0.4}{red!30};  
\pentpink{3*0.57}{-0.3854+3*0.724}{90}{0.4}{red!30};   
\coordinate (T) at (0.5706,-0.38541);
\coordinate (U) at (0,0.7236);
\coordinate (T') at ($(T)+(U)$);
%\draw[blue] (0,0) -- ++ ($2*(U)$);
%\draw[blue] ($-1.0*(T')$) -- ++ ($3*(U)$);
%\draw[blue] (T) -- ++ ($3*(U)$);
%\draw[blue] ($(T)+(T')$) -- ++ ($2*(U)$);
%\draw[blue] ($(T)+(T')+(T)$) -- ++ ($3*(U)$);
\draw[blue] ($-1.0*(T)$) -- ++ ($2.0*(T)$);
\draw[blue] ($-1.0*(T) + (U)$) -- ++ ($4.0*(T)$);
\draw[blue] ($-1.0*(T) + 2.0*(U)$) -- ++ ($4.0*(T)$);
\draw[blue] ($1.0*(T) + 3.0*(U)$) -- ++ ($2.0*(T)$);
\draw[blue] ($(T)$) -- ++ ($2.0*(T')$);
\draw[blue] ($1.0*(T) +1.0*(U) -2.0 *(T')$) -- ++ ($4.0*(T')$);
\draw[blue] ($1.0*(T) +2.0*(U) -2.0 *(T')$) -- ++ ($4.0*(T')$);
\draw[blue] ($1.0*(T) +3.0*(U) -2.0 *(T')$) -- ++ ($2.0*(T')$);
\draw[blue] ($-1.0*(T) + 2.0*(U)$) -- ++ ($-1.0*(T')$) -- ++ (T)
  -- ++ ($-1.0*(T')$) -- ++ (T)
  -- ++ ($-1.0*(T')$) -- ++ (T);
\draw[blue] ($3.0*(T) + 4.0*(U)$) -- ++ (T) -- ++ ($-1.0*(T')$) 
  -- ++ (T)
  -- ++ ($-1.0*(T')$) 
  -- ++ (T)
  -- ++ ($-1.0*(T')$);
\end{scope}
}

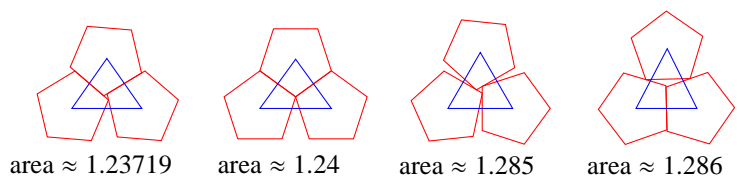
\tikzfig{subcritical}{Subcritical triangles.
Experimentally, the first triangle has the smallest area among all
acute Delaunay triangles.  The second triangle is a cloverleaf that
has two edges of minimal length $2\kappa$.  Experimentally, the
third triangle minimizes the longest edge among subcritical acute Delaunay triangles.
It is equalateral with edge length about $1.72256$.}
{
\begin{scope}[scale=0.5]
\threepent{0.00}{0.00}{85.23}{0.93}{1.33}{265.23}{1.86}{0.00}{229.23};
\draw (0.5,-1.5) node {$\area\approx 1.23719$};
\end{scope}
\begin{scope}[scale=0.5,xshift=5cm]
\threepent{0.00}{0.00}{90.00}{0.95}{1.31}{270.00}{1.90}{0.00}{234.00};
\draw (0.5,-1.5) node {$\area\approx 1.24$};
\end{scope}
\begin{scope}[scale=0.5,xshift=10cm]
\threepent{0.00}{0.00}{96.03}{0.86}{1.49}{264.03}{1.72}{0.00}{216.03};
\draw (0.5,-1.5) node {$\area\approx 1.285$};
\end{scope}
\begin{scope}[scale=0.5,xshift=15cm] %example4
\threepent{0.00}{0.00}{109.36}{0.83}{1.59}{235.36}{1.62}{0.00}{145.36};
\draw (0.5,-1.5) node {$\area\approx 1.286$};
\end{scope}
}
% format_pinwheelAB zero zero (m 0.16246);;
% example4: format_ljedgeAC (ratpi 3 10) zero (m 0.68);;

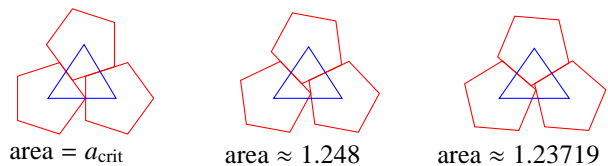
\tikzfig{KK2min}{Ice-ray triangle deformation.
The ice-ray triangle is not a local minimum of the area function.
The  ice-ray triangle (left) can be continuously deformed
along an area decreasing curve to the subcritical acute triangle of 
numerically minimum area (right). }
{
\begin{scope}[scale=0.5]
\threepent{0.00}{0.00}{72.00}{0.96}{1.43}{252.00}{1.81}{0.00}{216.00};
\draw (0.5,-1.5) node {$\area=\acrit$};
\end{scope}
\begin{scope}[scale=0.5,xshift=6cm]
\threepent{0.00}{0.00}{79.49}{0.93}{1.37}{259.49}{1.82}{0.00}{223.49};
\draw (0.5,-1.5) node {$\area\approx 1.248$};
\end{scope}
\begin{scope}[scale=0.5,xshift=12cm]
\threepent{0.00}{0.00}{85.23}{0.93}{1.33}{265.23}{1.86}{0.00}{229.23};
\draw (0.5,-1.5) node {$\area\approx 1.23719$};
\end{scope}
}
%format_pinwheelAB zero zero sigma;;
%format_pinwheelAB zero zero (m 0.35);;
%format_pinwheelAB zero zero (zero);;

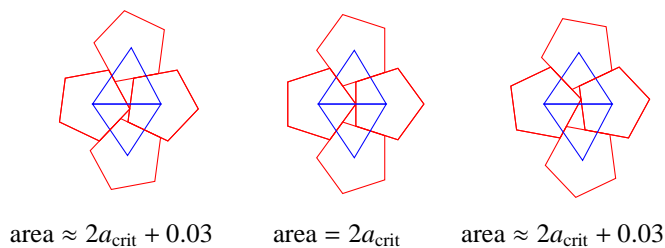
\tikzfig{Gamma}{Ice-ray dimer deformation.  The ice-ray
  dimer (center) admits a shear motion that preserves all edges of
  contact.  This deformation increases the area of the dimer.  It is
  obvious by the symmetry of the figures on the left and right that
  the area function along this deformation has a critical point at the
  ice-ray dimer.  It is known that the ice-ray dimer is a local
  minimimum of the area function \cite{Kus}.  }
{
\begin{scope}[scale=0.5]
\threepent{0.00}{0.00}{72.00}{0.96}{1.43}{252.00}{1.81}{0.00}{216.00};
\threepent{0.00}{0.00}{-72.00}{0.96}{-1.43}{-252.00}{1.81}{0.00}{-216.00};
\draw (0.5,-3.5) node {$\area=2\acrit$};
\end{scope}
\begin{scope}[scale=0.5,xshift=6cm]
\threepent{0.00}{0.00}{79.49}{0.93}{1.37}{259.49}{1.82}{0.00}{223.49};
\threepent{0.00}{0.00}{-64.51}{1.03}{-1.50}{-244.51}{1.82}{0.00}{-208.51};
\draw (0.5,-3.5) node {$\area\approx 2\acrit+0.03$};
\end{scope}
\begin{scope}[scale=0.5,xshift=-6cm]
\threepent{0.00}{0.00}{-79.49}{0.93}{-1.37}{-259.49}{1.82}{0.00}{-223.49};
\threepent{0.00}{0.00}{64.51}{1.03}{1.50}{244.51}{1.82}{0.00}{208.51};
\draw (0.5,-3.5) node {$\area\approx 2\acrit+0.03$};
\end{scope}
}
% format_pinwheelAB zero zero (m 0.35);;
% {0.00}{0.00}{79.49}{0.93}{1.37}{259.49}{1.82}{0.00}{223.49};

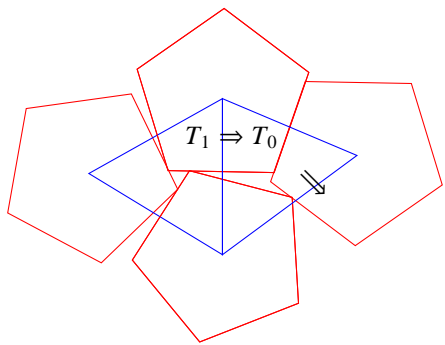
\tikzfig{pseudo-dimer}{Pseudo-dimer.  There exist pairs of acute
  triangles $(T_1,T_0)$ such that the sum of their two areas is at
  most $2\acrit$ and such that (1) $T_1$ is subcritical and is
  adjacent to $T_0$ along the longest edge of $T_1$, but such that (2)
  the longest edge of $T_0$ is not the edge shared with $T_1$.  We
  call such pairs {\it pseudo-dimers}.    The illustrated
  pseudo-dimer has area about $2\acrit - 10^{-5}$, and the triangle $T_0$
  has longest edge length about $1.84$.  The correction
  term $\epsM = 0.008$ is based on this and closely related examples
  of pseudo-dimers.  Pseudo-dimers add
  significant complications to the proof. }
{
\begin{scope}[scale=1.2,rotate=90]
\threepent{0.00}{0.00}{93.42}{0.90}{1.48}{259.93}{1.73}{0.00}{214.93};
\threepent{0.00}{0.00}{-50.58}{1.10}{-1.49}{-181.07}{1.73}{0.00}{-217.07};
\draw (1.3,-0.1) node {$T_1\Rightarrow T_0$};
\draw (0.8,-1) node {$\Searrow$};
\end{scope}
}
%format_shared_pinwheelAC (m 0.157) (m 0.2359) (m 0.2376);;
%format_shared_lj2edgeAC zero (pi25-(m 0.2359)) (two*sigma  - (m 0.2376));;

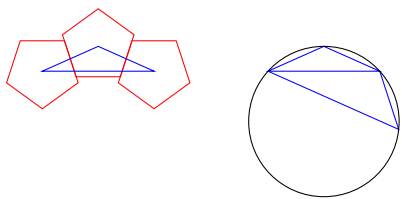
\tikzfig{obtuse}{Obtuse Delaunay triangles can have small area.  A
  triangle with circumradius about $2$ and area about $0.98$
  is shown. The adjacent Delaunay triangle cannot have a vertex
  inside the circumcircle of the first.  The neighbor of a subcritical
  obtuse Delaunay triangle tends to have large area.}
{
\begin{scope}[scale=0.5]
\threepent{0}{0}{90}
{1.486}{-0.661}{-90}
{-1.486}{-0.661}{-90};
\end{scope}
\begin{scope}[scale=0.5,xshift=6cm]
\draw[blue] (0,0)-- (1.486,-0.661) -- (-1.486,-0.661) -- cycle;
\draw[blue] (1.486,-0.661) -- ($(0,-2) + (-6:2)$) -- (-1.486,-0.661);
\draw (0,-2) circle (2cm);
\end{scope}
}
% let double_slider x = 
%  let theta = pi45 - two * sin_I (x / (two * kappa)) in
%  let ell = sqrt_I (four*kappa*kappa + x*x) in
%  let ell' = iloc ell ell theta in
%  area_I ell ell ell';;
%
% double_slider (m 0.17);;

\subsection{attachment, modified area, and clusters}

As mentioned above, our strategy for proving the pentagonal ice-ray
conjecture is to collect triangles into finite clusters such that the
average area over each cluster is at least $\acrit$.  The clusters are
defined by an equivalence relation.  The equivalence relation, in
turn, is defined as the reflexive, symmetric, transitive closure of a
further relation on the set of Delaunay triangles in a pentagon
packing.

A second strategy is to replace the area function $\area(T)$ on
Delaunay triangles with a modified area function $b(T)$.  The
modification steals area from nearby triangles that have area to spare
and gives to triangles in need.  It will be sufficient to prove that
the average of $b(T)$ over each cluster is at least $\acrit$.

In more detail, below, we define a particular relation $(\rab)$,
viewing a relation in the usual way as a set of ordered pairs.  For
this relation $(\rab)$, we write ${(\equiv_{b})}$ for the equivalence
relation obtained as the reflexive, symmetric, transitive closure of
$(\rab)$.  We call a corresponding equivalence classes $\C$ a {\it
  cluster}.  In other words, the relation $(\rab)$ defines a directed
graph whose nodes are the Delaunay triangles of a pentagon packing,
with directed edges given as arrows $T_1\rab T_0$.  A cluster is the
set of nodes in a connected component of the underlying undirected
graph.

We say that Delaunay triangle $T_1$ {\it attaches to Delaunay
  triangle} $T_0$ when the following condition holds: $T_0$ is the
adjacent triangle to $T_1$ along the longest edge of $T_1$.  (If the
triangle $T_1$ has more than one equally longest edge, fix once and
for all a choice among them, and use this choice to determine the
triangle that $T_1$ attaches to.  Thus, $T_1$ always attaches to
exactly one triangle $T_0$.  We can assume that the tie-breaking
choices are made according to a translation invariant rule.)  We write
$T_1 \Ra T_0$ for the attachment relation {\it $T_1$ attaches to
  $T_0$}.  Although it is not always possible to adhere to the
convention, note our general convention to use descending subscripts
$i > j$ for attachment $T_i\Ra T_j$ of triangles.  We also follow a
general convention of letting $T_0$ denote a triangle that is the
target of other triangles in the same cluster.

If $\T$ is a finite set of Delaunay triangles, we set
$\area(\T):=\sum_{T\in\T} \area(T)$.  The following are key
definitions of this article: dimer pair, pseudo-dimer, $\N$, $\M$,
$n_\pm$, $m_\pm$, $b(T)$, $(\rab)$, and cluster.

\begin{definition}[dimer pair]
  We define a dimer pair to be an ordered pair $(T_1,T_0)$ of Delaunay
  triangles such that
\begin{enumerate}
\item $T_0$ and $T_1$ are both nonobtuse.
\item $T_1\Ra T_0$, and $T_1$ is subcritical.
\item $T_0 \Ra T_1$.
\item $\area\{T_1,T_0\} \le 2\acrit$.
\end{enumerate}
We write $DP$ for the set of dimer pairs.
\end{definition}

\begin{definition}[pseudo-dimer]
  We define a pseudo-dimer to be an ordered pair $(T_1,T_0)$ of
  Delaunay triangles such that
\begin{enumerate}
\item $T_0$ and $T_1$ are both nonobtuse;
\item $T_1\Ra T_0$, and  $T_1$ is subcritical.
\item $T_0 \nRightarrow T_1$;
\item $\area\{T_1,T_0\} \le 2\acrit$;
\end{enumerate}
We write $\PD$ for the set of pseudo-dimers.
\end{definition}

We observe that dimer pairs differ from pseudo-dimers in the third
defining condition, which is a condition on the location of the
longest edge of $T_0$.  Every pseudo-dimer determines a third triangle
$T_-$ by the condition $T_0\Ra T_-$.  The shared edge of $T_0$ and
$T_-$, leading out of the pseudo-dimer is called the {\it egressive} edge of
the pseudo-dimer or of the triangle $T_0$.

For any set $\S$ of ordered pairs, and any Delaunay triangles $T_+$
and $T_-$, let
\[
n_+(T_+,\S) = \card \{T_-\mid (T_+,T_-)\in \S\}
\quad\text{ and }\quad
n_-(T_-,\S)
= \card \{T_+\mid (T_+,T_-)\in \S\}.
\]

We define $\N$ as the disjoint union of two sets of ordered pairs: $\N
= \N_{\text{obtuse}} \sqcup \N_{\text{nonobtuse}}$.  The set
$\N_{\text{obtuse}}$ consists of those pairs with obtuse target and
$\N_{\text{nonobtuse}}$ are those pairs with nonobtuse target $T_-$,
as follows.

Define $\N_{\text{obtuse}}$ to be the set of pairs $(T_+,T_-)$ of
Delaunay triangles such that
\begin{enumerate}
\item $T_-$ is obtuse;
\item $T_+\Ra T_-$.
\item $T_+$ is nonobtuse and the longest edge of $T_+$ has length at
  least $1.72$;
\end{enumerate}

Define  $\N_{\text{nonobtuse}}$ to 
be the set of pairs $(T_+,T_-)$ of Delaunay triangles such that
\begin{enumerate}
\item $T_-$ is nonobtuse;
\item $T_+\Ra T_-$;
\item $T_+$ is nonobtuse and the longest edge of $T_+$ has length at
  least $1.72$;
\item there exists an obtuse triangle $T$ such that $T\Ra T_-$ and 
\[
\area(T) -   n_-(T,\N_{\text{obtuse}})\epsN\le \acrit.
\]
\end{enumerate}

Let $\M$ be the set of pairs $(T_+,T_-)$ of Delaunay triangles such that
\begin{enumerate}
\item $(T_+,T_-)\not\in \N$;
\item $T_+\Ra T_-$;
\item The longest edge of $T_+$ has length  at least $1.72$.
\item There exists a unique $T_1$ such that $(T_1,T_+)\in \PD$.
\end{enumerate}

We abbreviate 
\[
m_+(T) = n_+(T,\M), 
\quad m_-(T) = n_-(T,\M), 
\quad n_+(T) =n_+(T,\N), 
\ \text{ and } \ 
n_-(T) = n_-(T,\N).
\]

\begin{remark}
  Note  that $n_+(T)\le 1$, because each triangle attaches to
  exactly one other triangle.  Also, $n_-(T)\le 3$, because each
  attachment forms along an edge of the triangle $T$.  Similarly,
  $m_+(T)\le 1$ and $m_-(T)\le 3$.
\end{remark}

We define the modified area function
\begin{equation}
b(T) := \op{area}(T) + \epsN (n_+(T) - n_-(T)) + \epsM (m_+(T) - m_-(T)).
\end{equation}
We say that $T$ is {\it $b$-subcritical} if $b(T) \le \acrit$.  We write
$T_1\rab T_2$, if $T_1\Ra T_2$ and $T_1$ is $b$-subcritical.  An
equivalence classes of triangles under the corresponding equivalence
relation $(\equiv_b)$ is called a {\it cluster}.

We note that $(\rab)$ is given by a translation-invariant rule.  The
function $b(T)$ is also translation invariant and depends only on
local information in the pentagon packing near the triangle $T$.

The intuitive basis of using $b(T)\le \acrit$ as the condition for
cluster formation with $(\rab)$ is the following.  Eventually, we wish
to show that the average of the modified areas $b(T)$ over each
cluster is greater than $\acrit$.  (See Lemma~\rif{lemma:main}.)  If
$b(T)\le \acrit$, then this means that its modified area $b(T)$ is
less than our desired goal for the average over the cluster, so that
$T$ needs to be part of a larger cluster.  This suggests we should
define $(\rab)$ in such a way that a further triangle is added
whenever $b(T)\le \acrit$. That is what our definition of $(\rab)$
does.

We remark that the modification $b(T)$ of the area function has two
correction terms.  The first term $\epsN (n_+(T) - n_-(T))$ takes away
from obtuse triangles (and their neighbors) and gives to nonobtuse
triangles.  The intuition behind this correction term is that the
triangle adjacent to an obtuse triangle along its long edge has a very
large surplus area that can be beneficially redistributed
(Figure~\rif{fig:obtuse}).  It allows us to make a clean separation of
the proof of the main inequality into two cases: clusters that contain
an obtuse triangle and clusters that do not.

The second term $\epsM (m_+(T) - m_-(T))$ augments the area of a
pseudo-dimer, by taking from a neighboring triangle.  The intuition
behind this correction term is that a pseudo-dimer can have area
strictly less than the ice-ray dimer, and we need to boost its area
with a correction term to make it satisfy the main inequality
(Figure~\rif{fig:pseudo-dimer}).  A calculation given below shows that
the neighbor of the pseudo-dimer has area to spare
(Corollary~\rif{lemma:egress'}).

The correction terms allow us to keep the size of each cluster small.
Eventually, we show that each cluster contains at most four triangles
(Lemma~\rif{lemma:card4}).  This small size will be helpful when we
turn to the computer calculations.  If $(T_+,T_-)$ is a member of $\N$
or $\M$, the rough expectation is that there should not be an arrow
$T_+\rab T_-$ and that $T_+$ and $T_-$ should belong to different
clusters.  That is, $\N$ and $\M$ are designed to mark cluster
boundaries.  Some lemmas in this article make this expectation more
precise (for example, Lemma~\rif{lemma:Nb}).

We give some simple consequences of our definitions.

\begin{lemma} 
  If $(T_+,T_-)\in \M$, then both $T_+$ and $T_-$ are nonobtuse.
\end{lemma}

\begin{proof} 
  By the definition of pseudo-dimer, $T_+$ is nonobtuse, because
  $(T_1,T_+)\in\PD$ for some $T_1$.  If $T_-$ were obtuse, then we
  would satisfy all the membership conditions for $(T_+,T_-)\in \N$
  (obtuse target), which is impossible because $\N\cap \M =
  \emptyset$.
\end{proof}

\begin{lemma}[obtuse $b$]\libel{lemma:obtuse-b} 
  If $T$ is an obtuse Delaunay triangle, then
  $m_+(T)=m_-(T)=n_+(T)=0$.  Thus, $b(T) = \area(T) - \epsN n_-(T)$.
\end{lemma}

\begin{proof} 
  The previous lemma gives $m_+(T)=m_-(T)=0$.  The first component of
  $\N$ is nonobtuse by definition, so $n_+(T)=0$.
\end{proof}

\begin{corollary}\libel{lemma:obtuse-b-arrow} 
  If $(T_+,T_-)\in\N$ with nonobtuse target $T_-$, then there exists
  an obtuse triangle $T$ such that $T\rab T_-$.
\end{corollary}

\begin{proof} 
  Let $T$ be the obtuse triangle such that $T\Ra T_-$ given by $\N$
  (nonobtuse target), condition 4.  Then $n_-(T,\N_{\text{obtuse}}) =
  n_-(T)$, and by Lemma~\rif{lemma:obtuse-b}, we have $b(T) =
  \area(T)-\epsN n_-(T)\le\acrit$.  The result follows from the
  definition of $(\rab)$.
\end{proof}

\begin{lemma}\libel{lemma:n1m1}  
  Suppose that $n_+(T)>0$. Then $m_+(T)=0$.
\end{lemma}

\begin{proof} 
  This follows directly from the disjointness of $\M$ and $\N$.
\end{proof}

\begin{lemma}\libel{lemma:n2m2}  
  Suppose $n_-(T_-)>0$.  Then $m_-(T_-)=0$.
\end{lemma}

\begin{proof}
  If $T_-$ is obtuse, then $m_-(T_-)=0$ by Lemma~\rif{lemma:obtuse-b}.
  We may assume that $T_-$ is nonobtuse.  By the definition of $\N$
  (nonobtuse target), the inequality $n_-(T_-) >0$ implies the
  existence of an obtuse $T$ with $T\Ra T_-$ (by $\N$ condition 4).
  If (for a contradiction) $m_-(T_-)>0$, then there exists
  $(T_+,T_-)\in\M$.  We complete the proof by checking that
  $(T_+,T_-)$ satisfies each membership condition of $\N$ (nonobtuse
  target), so that $(T_+,T_-)\in\N$.  This contradicts disjointness:
  $(T_+,T_-)\in \N\cap\M = \emptyset$.
\end{proof}

\subsection{the main inequality}

\begin{lemma}\libel{lemma:main}  
If  every cluster $\C$ in every saturated packing of regular
  pentagons is finite, and if for some $a$ every cluster average
  satisfies
\begin{equation}\libel{eqn:main}
\frac{\sum_{T\in \C} b(T)}{\card(\C)} \ge a,
\end{equation}
then the density of a packing of regular pentagons never exceeds 
\[
\frac{\area_P} {2 a},
\]
where $\area_P = 5\kappa\sigma$ is the area of a regular pentagon of
circumradius $1$.  In particular, if the inequality holds for
$a=\acrit$, then the density never exceeds
\[
\frac{\area_P}{2 \acrit} = \frac{5 - \sqrt{5}}{3},
\] % checked 2016/2/18 in calcs.ml
the density of the pentagonal ice-ray.
\end{lemma}

For any finite set $\C$ of triangles, we will call the inequality
(\rif{eqn:main}) with the constant $a=\acrit$ the {\it main inequality}
(for $\C$).  We call the strict inequality,
\begin{equation}\libel{eqn:strict-main}
\frac{\sum_{T\in \C} b(T)}{\card(\C)} > \acrit,
\end{equation}
the {\it strict main inequality} (for $\C$).

\begin{proof} The maximum density can be obtained as the limit of the
  densities of a sequence of saturated periodic packings.  Thus, it is
  enough to consider the case when the packing is periodic.  A
  periodic packing descends to a packing on a flat torus
  $\ring{R}^2/\Lambda$, for some lattice $\Lambda$.  The rule defining
  $\N$ is translation invariant, and $\N$ descends to the torus.  On
  the torus, the set of pentagons, the set of triangles, and the set
  $\N$ are finite.  The equivalence relation $(\equiv_b)$ defining
  clusters is translation invariant, and each cluster is finite, so
  that no cluster contains both a triangle and a translate of the
  triangle under a nonzero element of $\Lambda$.  Thus, each cluster
  $\C$ in $\ring{R}^2$ maps bijectively to a cluster $\C$ in the flat
  torus.  The functions $b$, $n_\pm$, $m_\pm$ are the same whether
  computed on $\ring{R}^2$ or $\ring{R}^2/\Lambda$.  Let $p$ be the
  number of pentagons in the torus.  By the Euler formula for a torus
  triangulation, the number of Delaunay triangles is $2p$.  We have
\[
\sum_{T} n_-(T) =  \sum_{T} n_+(T) = \card(\N);\qquad
\sum_{T} m_-(T) =  \sum_{T} m_+(T) = \card(\M).
\]
Thus, the terms in $b(T)$ involving $n_+(T)$, $n_-(T)$, $m_+(T)$,
and $m_-(T)$ cancel:
\[
\area(\ring{R}^2/\Lambda) = \sum_T \area(T) = \sum_T b(T).
\]    
Let $\area_P$ be the area of a regular pentagon.  Making use of the
hypothesis of the lemma, we see that the density is
\[
\frac{p\, \area_P}{\sum_T \area(T)} 
=\frac{p\, \area_P}{\sum_T b(T)} \le\frac{p\, \area_P}{2 p\, a} 
= \frac{\area_P}{2\, a}.
\]
When $a=\acrit$, the term on the right is the density of the
pentagonal ice-ray, as desired.
\end{proof}

This article gives a proof of the following theorem. In view of
Lemma~\rif{lemma:main}, it implies the main result,
Theorem~\rif{thm:main}.  The proof of this result appears at the end
of Section~\rif{sec:nonobtuse}.

\begin{theorem}\libel{thm2:main}
  Let $\C$ be a cluster of Delaunay triangles in a saturated packing
  of regular pentagons.  Then $\C$ is finite and the average of $b(T)$
  over the cluster is at least $\acrit$.  That is, $\C$ satisfies the
  main inequality.  Equality holds exactly when $\C$ consists of two
  adjacent Delaunay triangles from the pentagonal ice-ray, attached
  along their common longest edge, forming an ice-ray dimer pair.
\end{theorem}

\begin{remark}\libel{rem:equal}
Analyzing the proof of Lemma~\rif{lemma:main}, we see that for a
periodic packing, the maximum density is achieved exactly when each
 cluster in the packing gives exact equality in the main inequality.
Thus, the theorem implies that the pentagonal ice-ray  is the
unique periodic packing that achieves maximal density.
\end{remark}

\section{Pentagons in Contact}\libel{sec:contact}

\subsection{notation}

By way of general notation, we use uppercase $A,B,C,\ldots$ for
pentagons; $\v_A,\v_B,\ldots$ for the vertices of pentagons;
$\c_A,\c_B,\ldots$ for centers of pentagons; $\p,\q,\ldots$ for
general points in the plane; $\normo{\p}$ for the Euclidean norm;
$\dx{AB}=\dd{A}{B}$ for center-to-center distances;
$\alpha,\beta,\gamma,\phi,\psi,\ldots$ for angles; and
$T,T',T_+,T_-,T_0,T_1,T_2,\ldots$ for Delaunay triangles.

We let $\eta(T) = \eta(d_1,d_2,d_3)$ be the circumradius of a triangle
$T$ with edge lengths $d_1,d_2$, and $d_3$.

Let $\angle(\p,\q,\r)$ be the angle at $\p$ of the triangle with
vertices $\p$, $\q$, and $\r$.  Let $\arc(d_1,d_2,d_3)$ be the angle
of a triangle (when it exists) with edge lengths $d_1$, $d_2$, and
$d_3$, where $d_3$ is the edge length of the edge opposite the
calculated angle.  We write $\area(T)= \area(d_1,d_2,d_3)$ for the
area of triangle $T$ with edge lengths $d_1,d_2,d_3$.

\subsection{triple contact}

In this subsection, we describe possible contacts between pentagons.

% nonobtuse.
We consider a single Delaunay triangle and the three nonoverlapping
pentagons centered at the triangle's vertices
(Figure~\rif{fig:3C-type}).  We call such a configuration a {\it
  $P$-triangle}.  A $P$-triangle is determined up to congruence by six
parameters: the lengths of the edges of the Delaunay triangle and the
rotation angles of the regular pentagons.  When we refer to the area
or edges of a $P$-triangle, we mean the area or edges of the
underlying Delaunay triangle.  More generally, we allow $P$-triangles
to inherit properties from Delaunay triangles, such as obtuseness or
nonobtuseness, the relation $(\rab)$, clusters, and so forth. In
clusters of $P$-triangles it is to be understood that the pentagons
agree at coincident vertices of the triangles.  When there is a
fixed backdrop of a Delaunay triangulation of a pentagon packing, it
is not necessary to make a careful distinction between a $P$-triangle
and its underlying Delaunay triangle.

When two pentagons touch each other, some vertex of one meets an edge
of the other.  We call the pentagon with the vertex contact the {\it
  pointer} pentagon, and the pentagon with the edge contact the {\it
  receptor} pentagon (Figure~\rif{fig:receptor}).  We also call the
vertex in contact the {\it pointer vertex} of the pointer
pentagon. There are degenerate cases, when the contact set between two
pentagons contains of a vertex of both pentagons.  In these degenerate
cases, the designation of one pentagon as a pointer and the other as a
receptor is ambiguous.

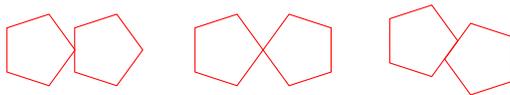
\tikzfig{receptor}{Pointer and receptor pairs of pentagons.  
In each pair, the pentagon on the
left can be considered a pointer pentagon, 
with receptor on the right.  The first pair is nondegenerate,
and the other two pairs are degenerate.}{
[scale=0.5]
\pent{0}{0}{0}{1};
\pent{1.8}{0}{0}{1};
\pent{5}{0}{0}{1};
\pent{7}{0}{180}{1};
\pent{10+0.176}{0.243}{0}{1};  % 0.3 {Cos [54^o],Sin[54^o]}
\pent{12-0.172}{-0.243}{180}{1};
}

We say that a $P$-triangle is $3C$ (triple contact),
if each of the three pentagons contacts the other two.

We may direct the edges of a $3C$ triangle by drawing an arrow from
the pointer pentagon to the receptor pentagon.  We may classify $3C$
triangles according to the types of triangles with directed edges.
There are two possibilities for the directed graph.
\begin{enumerate}
\item Some vertex of the triangle is a source of two directed edges
  and another vertex is the target of two directed edges
  ($LJ$-junction, $TJ$-junction or $\Delta$-junction).
\item Every vertex of the triangle is both a source and a target
  (pinwheel, pin-$T$).
\end{enumerate}

As indicated in parentheses, we have named each of the various contact
types.  An example of each of the contact types is shown in
Figure~\rif{fig:3C-type}.  An exact description of these contact types
appears later.  The name $LJ$-junction is suggested by the $L$-shaped
region bounded by the three pentagons.  Similarly, the name
$TJ$-junction is suggested by the $T$-shaped region bounded by the
three pentagons.  Similarly, for $\Delta$-junctions.  This section
shows that the types in the figure exhaust the geometric types of
$3C$ contact.

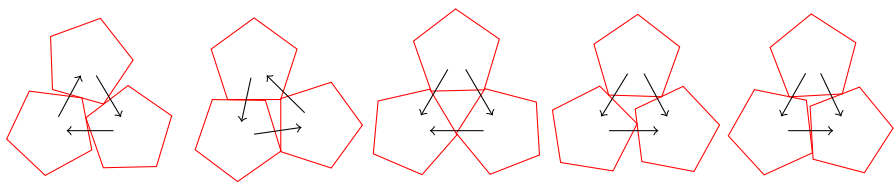
\tikzfig{3C-type}{Types of $3C$-contact from left-to-right:
a pinwheel, a pin-$T$ junction, 
a $\Delta$-junction, an $LJ$-junction, and a $TJ$-junction.}
{
[scale=0.6]
\threepentnoD{0.00}{0.00}{46.69}{0.82}{1.53}{218.09}{1.73}{0.00}{163.28};
\coordinate (A) at (0,0);
\coordinate (B) at (0.82,1.53);
\coordinate (C) at (1.73,0);
\draw[->] ($0.8*(A) + 0.2*(B)$) -- ($0.2*(A) + 0.8*(B)$);
\draw[->] ($0.8*(B) + 0.2*(C)$) -- ($0.2*(B) + 0.8*(C)$);
\draw[->] ($0.8*(C) + 0.2*(A)$) -- ($0.2*(C) + 0.8*(A)$);
\begin{scope}[xshift=4.5cm,yshift=1.5cm]
\threepentnoD{0.00}{0.00}{90}{1.40}{-1.377}{0}{-0.35}{-1.628}{-18.89};
\coordinate (A) at (0,0);
\coordinate (B) at (1.40,-1.377);
\coordinate (C) at (-0.35,-1.628);
\draw[<-] ($0.8*(A) + 0.2*(B)$) -- ($0.2*(A) + 0.8*(B)$);
\draw[<-] ($0.8*(B) + 0.2*(C)$) -- ($0.2*(B) + 0.8*(C)$);
\draw[<-] ($0.8*(C) + 0.2*(A)$) -- ($0.2*(C) + 0.8*(A)$);
\end{scope}
\begin{scope}[xshift=8cm]
\threepentnoD{0.00}{0.00}{-5.16}{0.99}{1.70}{235.38}{1.98}{0.00}{183.43};
\coordinate (A) at (0,0);
\coordinate (B) at (0.99,1.7);
\coordinate (C) at (1.98,0);
\draw[<-] ($0.8*(A) + 0.2*(B)$) -- ($0.2*(A) + 0.8*(B)$);
\draw[->] ($0.8*(B) + 0.2*(C)$) -- ($0.2*(B) + 0.8*(C)$);
\draw[->] ($0.8*(C) + 0.2*(A)$) -- ($0.2*(C) + 0.8*(A)$);
\end{scope}
\begin{scope}[xshift=12cm]
\threepentnoD{0.00}{0.00}{80.86}{0.97}{1.58}{232.21}{1.82}{0.00}{223.24};
\coordinate (A) at (0,0);
\coordinate (B) at (0.97,1.58);
\coordinate (C) at (1.82,0);
\draw[<-] ($0.8*(A) + 0.2*(B)$) -- ($0.2*(A) + 0.8*(B)$);
\draw[->] ($0.8*(B) + 0.2*(C)$) -- ($0.2*(B) + 0.8*(C)$);
\draw[<-] ($0.8*(C) + 0.2*(A)$) -- ($0.2*(C) + 0.8*(A)$);
\end{scope}
\begin{scope}[xshift=16cm]
\threepentnoD{0.00}{0.00}{114.48}{0.90}{1.59}{237.18}{1.66}{0.00}{219.24};
\coordinate (A) at (0,0);
\coordinate (B) at (0.9,1.59);
\coordinate (C) at (1.66,0);
\draw[<-] ($0.8*(A) + 0.2*(B)$) -- ($0.2*(A) + 0.8*(B)$);
\draw[->] ($0.8*(B) + 0.2*(C)$) -- ($0.2*(B) + 0.8*(C)$);
\draw[<-] ($0.8*(C) + 0.2*(A)$) -- ($0.2*(C) + 0.8*(A)$);
\end{scope}
}
%
% formatpent3deg (delToPent3 (delAf pinwheeldelA (0.15) (0.15) (0.5)));;  
% formatpent3deg (delToPent3 (delAf pinwheeldelA (0.15) (0.15) (0.0)));;  

% formatpent3deg (delToPent3 (delAf deltajdelA (0.2) (0.15) (0.05)));;  
% formatpent3deg (delToPent3 (delAf ljdelA (0.5) (0.6) (0.90)));;  
% formatpent3deg (delToPent3 (delAf tjdelA (1.0) (1.2) (1.3*. ee)));;  

% graphics for pin-T brute forced in Mathematica.

%Note that the figures give two different geometric types
%associated to the combinatorial structure of a vertex that
%is the target of two edges.  

A {\it cloverleaf} arrangement is a $3C$ triangle that has a point at
which vertices from all three pentagons meet
(Figure~\rif{fig:clover}).  This is degenerate because this shared
vertex can be considered as a pointer or receptor.

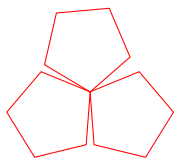
\tikzfig{clover}{A cloverleaf (degenerate pinwheel).  The Delaunay triangle
in this particular example is not subcritical.}
{
[scale=0.6]
\begin{scope}[xshift=4cm]
\threepentnoD{0.00}{0.00}{31.70}{0.76}{1.52}{203.11}{1.70}{0.00}{148.30};
\end{scope}
}

In general, in this article, a {\it non-anomaly} lemma refers to a
geometrical lemma that shows that certain geometric configurations are
impossible.  Generally, it is obvious from the informal pictures that
various configurations cannot exist.  The non-anomaly lemmas then
translate the intuitive impossibilities into mathematically precise
statements.  We give a few non-anomaly lemmas as follows.
They are expressed as separation results, asserting that
two pentagons $A$ and $C$ do not touch.

\begin{lemma}\libel{lemma:sep1} 
Let $T$ be a $3C$-triangle with pentagons $A$, $B$, and $C$ such that
$B$ is a pointer to both of the other pentagons $A$ and $C$.  
Assume that $T$ is not a cloverleaf.
Then the two pointer vertices $\v_B$ and $\v_B'$ are adjacent
vertices of $B$.
\end{lemma}

\begin{lemma}\libel{lemma:sep2}  
Let $T$ be a $3C$-triangle with pentagons $A$, $B$, and $C$.
Suppose that pentagon
$A$ is a pointer to $B$ at $\v_A$ and that $B$ is a pointer to $C$
at $\v_B$.  Then on $B$, the vertex
$\v_B$ is not opposite to the edge of $B$ containing
$\v_A$.
\end{lemma}

\begin{lemma}\libel{lemma:sep3} 
  Let $T$ be a $3C$-triangle with pentagons $A$, $B$, and $C$ such
  that $B$ is a receptor of both of the other pentagons.  Then the two
  pointer vertices $\v_A$ and $\v_C$ lie on the same edge or adjacent
  pentagon edges of $B$.
\end{lemma}

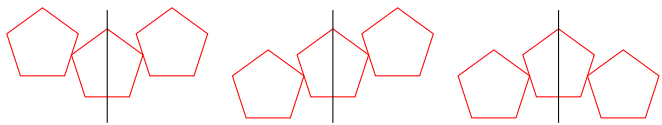
\tikzfig{sep}{A line through the center 
of the middle pentagon $B$
through one of its vertices 
separates the two extremal 
pentagons $A$ and $C$.}{
[scale=0.5]
\pen{0}{0}{90};
\pen{1.72}{0.56}{90};  % (1+kappa) {Cos pi/10 , sin pi/10}
\pen{-1.72}{0.56}{90};  
\draw (0,-1.5) -- (0,1.5);
\begin{scope}[xshift=6cm]
\pen{0}{0}{90};
\pen{1.72}{0.56}{90};  
\pen{-1.72}{-0.56}{90};  
\draw (0,-1.5) -- (0,1.5);
\end{scope}
\begin{scope}[xshift=12cm]
\pen{0}{0}{90};
\pen{1.72}{-0.56}{90};  
\pen{-1.72}{-0.56}{90};  
\draw (0,-1.5) -- (0,1.5);
\end{scope}
}

\begin{proof} The Lemmas~\rif{lemma:sep1}, \rif{lemma:sep2},
  and~\rif{lemma:sep3} can be proved in the same way.  In each case,
  we prove the contrapositive, assuming the negation of the geometric
  conclusion, and proving that the configuration is not $3C$.  We show
  that the configuration is not $3C$ by constructing a separating
  hyperplane between the pentagons $A$ and $C$.  In each case, the
  separating hyperplane is a line through the center of the middle
  pentagon $B$ and passing through a vertex $\v$ of that pentagon.
  See Figure~\rif{fig:sep}.  In the case of Lemma~\rif{lemma:sep1},
  there is a degenerate case of a cloverleaf, where all three
  pentagons meet at the vertex $\v$ on the separating line.
\end{proof}

\begin{definition}[$\Delta$]
  We say that a $3C$-triangle has type $\Delta$ if we are in the first
  case of Lemma~\rif{lemma:sep3} (the two pointer vertices $\v_A$ and
  $\v_C$ of $A$ and $C$ lie on the same edge of $B$) provided the line
  $\lambda$ through that edge of $B$ separates $B$ from $A$ and
  $C$. (See Figure~\rif{fig:delta}.)
\end{definition}

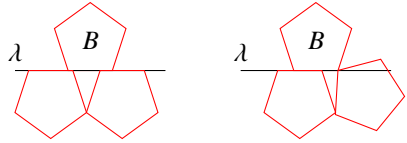
\tikzfig{delta}{In type $\Delta$, a line separates pentagon 
$B$ from
the other two pentagons. 
The second figure (which is degenerate of type $LJ$) does not have
type $\Delta$.}{
[scale=0.5]
\pen{0}{0}{90};
\draw (0,0) node {$B$};
\draw (-2.0,-0.809) node[above] {$\lambda$} --  (2.0,-0.809);
\pen{-1.05}{-2*0.809}{-90};
\pen{1.9 - 1.05}{-2*0.809}{-90};
\begin{scope}[xshift=6cm]
\pen{0}{0}{90};
\draw (0,0) node {$B$};
\draw (-2.0,-0.809) node[above] {$\lambda$} --  (2.0,-0.809);
\pen{-0.428}{-2*0.809}{-90};
\pen{1.36}{ -1.44}{ 212.704};
\end{scope}
}

In type $\Delta$, say $A$ is a pointer into $C$ at $\v$.  Then $\v_A$
and $\v$ are the two endpoints of some edge of $A$.  Also, $\v_C$ and
$\v$ lie on the same edge of $C$.  If the line $\lambda$ does not
separate $B$ from $A$ and $C$, then $\v_C$ is a shared vertex of $B$
and $C$, and we have a degeneracy that can also be viewed as $\v_A$
and $\v_C$ on adjacent pentagon edges of $B$.  This case will be
classified as a degenerate $LJ$-junction below.

\begin{definition} Let $T$ be a $P$-triangle with pentagons $A$, $B$,
  and $C$.  Assume that $A$ points to $B$ at $\v_{AB}$, and $B$ points
  to $C$ at $\v_{BC}$.  An {\it inner vertex} $\v$ of $B$ is a vertex
  $\v\ne \v_{BC}$ of $B$ such that $\v$ lies between $\v_{AB}$ and
  $\v_{BC}$ (along the short run of the perimeter of $B$ from
  $\v_{AB}$ to $\v_{BC}$).  We allow the degeneracy $\v=\v_{AB}$.
  See Figure~\rif{fig:pin1}.
\end{definition}

\begin{definition} Let $T$ be a $3C$-triangle.  FIx the pointer
  directions on $T$, if ambiguous. We say $T$ has type {\it pin}-$k$,
  for $k\in \{0,1,2,3\}$ if $A$ points into $B$, $B$ points into $C$,
  $C$ points into $A$, and if there are exactly $k$ pentagons among
  $A,B,C$ that have an inner vertex.  We use {\it pinwheel} as a
  synonym for pin-$0$ and {\it pin-T} as a synonym for pin-$2$.
\end{definition}

% page 23-nonobtuse.
\begin{lemma} There does not exist a $3C$-triangle $T$ of type pin-$1$
  with pentagons $A$, $B$, and $C$.
\end{lemma}

\tikzfig{pin1}{A distorted pin-$1$ configuration.}{
\begin{scope}[scale=1.0]
\pen{0}{0}{-90};
\pen{-0.974}{ 1.543}{ 32.70422};
\draw[red] ++ (-0.974,1.543) 
  ++ (0.842,0.532) node[black,anchor=west] {$\v_{BC}$} -- 
  ++ (115:0.2) -- ++ (-65:1.6) node[black,anchor=south west] {$\v_{CA}$}
   -- ++ (7:1.17);
\draw (0,0) node {$A$};
\draw (0,0) + (3*72-90:1) node[anchor=north west] {$\v_{AB}$};
\draw (-0.974,1.543) node {$B$};
\draw (1.1,1.543) node {$C$};
\smalldot{3*72-90:1}; %vAB
\draw (-0.974,1.543)+(32.7-72:1) node[anchor=south west]{$\v$};
\smalldot{$(-0.974,1.543)+(32.7-72:1)$}; %v
\smalldot{$(-0.974,1.543)+(32.7:1)$}; %vBC
\smalldot{$(-0.974,1.543)+(0.842,0.532)+(115:0.2)+(-65:1.6)$}; %vCA
%\draw (0.57,0.809) -- (2.1,0.809) node[anchor=south] {$\gamma$};
\end{scope}
}

\begin{proof} 
  For a contradiction, we draw a (distorted) picture of a pin-$1$
  configuration (Figure~\rif{fig:pin1}).  We let $\v$ be the inner
  vertex of $B$; that is, the vertex that is interior to the triangle
  $(\v_{CA},\v_{AB},\v_{BC})$ with vertices at the pointers $X$ into
  $Y$.  It is an endpoint of the edge of the pentagon $B$ containing
  $\v_{AB}$.  We have
\[
\angle (\v,\v_{AB},\v_{CA}) \le \pi,\quad 
\angle(\v,\v_{BC},\v_{AB}) = 3\pi/5,\quad
\angle(\v_{CA},\v_{BC},\v_{AB})\le 2\pi/5.
\]
(The last inequality uses the fact that $T$ is not a degenerate
pin-$2$, so that $\v_{CA}$ is not a vertex of $A$.)  We also have
\[
\angle(\v,\v_{BC},\v_{CA}) \ge 2\pi /5 
\ge \angle(\v_{CA},\v_{BC},\v_{AB}) \ge \angle(\v_{CA},\v_{BC},\v).
\]
The law of sines applied to the triangle $(\v,\v_{BC},\v_{CA})$ then
gives
\[
2\sigma =\norm{\v_{BC}}{\v}\le\norm{\v_{BC}}{\v_{CA}}\le 2\sigma.
\]
Thus, we have equality everywhere.  In particular,
$\angle(\v,\v_{AB},\v_{CA})=\pi$, and $\v_{BC}$ is a vertex of $C$.
Hence $\v_{BC}$ is a degenerate inner vertex of $C$, and $T$ has type
pin-$k$, for some $k\ge 2$.
%.  This has type $\Delta$.
\end{proof}

\begin{lemma}  The type pin-$3$ does not exist.  
\end{lemma}

\begin{proof} 
  Suppose for a contradiction that a $P$-triangle $T$ of type pin-$3$
  exists.  The region $X$ bounded by the three pentagons is a
  nonconvex star-shaped hexagon, with interior angles $\alpha'$,
  $7\pi/5$, $\beta'$, $7\pi/5$, $\gamma'$, and $7\pi/5$.  The vertices
  of $X$ with angles $7\pi/5$ are the inner vertices of the three
  pentagons of $T$.  The sum of the interior angles in a hexagon is
  $4\pi$:
\[
4\pi = \alpha'+\beta'+\gamma' + 3 (7\pi/5),
\]
which implies that $\alpha'+\beta'+\gamma' = -\pi/5$, which is
impossible.
\end{proof}

\begin{definition}[$TJ$ and $LJ$-junction]
  We say that a $3C$-triangle is a type $TJ$- or $LJ$-junction if it
  is not type $\Delta$ and if we are in the second case of
  Lemma~\rif{lemma:sep3} (both $A$ and $C$ point into $B$, and the two
  pointer vertices $\v_A$ and $\v_C$ lie on adjacent edges of $B$).
  Say $A$ is a pointer into $C$ at $\v$.  We say that it has type {\it
    $LJ$-junction} if $\v_C$ and $\v$ lie on the same pentagon edge of
  $C$, and otherwise we say it has type {\it $TJ$-junction}.
\end{definition}

We can be more precise about the structure of a $TJ$-junction.  In the
context of the definition, Lemma~\rif{lemma:sep2} implies that $\v$
and $\v_C$ lie on adjacent pentagon edges of $C$.

This completes the classification of $3C$-triangles: $\Delta$,
pinwheel, pin-$T$, $LJ$, and $TJ$.  Useful coordinate systems for the
various types can be found in the appendix
(Section~\rif{sec:appendix}).

\section{Delaunay Triangle Areas}

As an application of the classification from the previous section,
this section makes a computer calculation of a lower bound on the
longest edge length of a subcritical triangle.  We also obtain a lower
bound on the area of a nonobtuse Delaunay triangle.

\begin{lemma}\libel{lemma:21} 
  A nonobtuse subcritical Delaunay triangle has edge lengths at most $2.1$.
\end{lemma}

\begin{proof} By the monotonicity of area as a function of edge length
  for nonobtuse triangles, a triangle with an edge length at least
  $2.1$ has area at least
\[
\area(2.1,2\kappa,2\kappa) > \acrit,
\] % checked 2016/2/18 in Mathematica.
which is not subcritical.
\end{proof}

\begin{lemma}\libel{lemma:right} 
  A nonobtuse subcritical triangle has edge lengths less than
  $\kappa\sqrt8$.  In particular, a right-angled Delaunay triangle is
  not subcritical.
\end{lemma}

\begin{proof}  
  This is a corollary of the previous lemma, because $2.1 <
  \kappa\sqrt8 \approx 2.288$.
\end{proof}

\begin{remark}
  A motion of a pentagon in the plane can be described by an element
  of the isometry group of the plane, which is a semidirect product of
  a translation group and an orthogonal group.  Because of the
  dihedral symmetries of the regular pentagon, each motion can be
  realized as a translation followed by a rotation by angle between
  $0$ and $2\pi/5$.  In particular, a translation of a pentagon is a
  motion of a pentagon such that the rotational part is the identity.
\end{remark}

\begin{definition}
In a $P$-triangle, we say that a pentagon $A$ has {\it primary
  contact} if one or more of the following three conditions hold:
\begin{enumerate}
\item (slider contact) The pentagon $A$ and one $B$ of the other two
   share a positive length edge segment;
\item (midpointer contact) A vertex of one of the other two pentagons
  is the midpoint of one of the edges of the pentagon $A$; or
\item (double contact) The pentagon $A$ is in contact with both of the
  other pentagons.
\end{enumerate}
\end{definition}

The next lemma is used to give area estimates when an edge has length
at most $\kappa\sqrt{8}$.  We give two forms of the lemma.  We prove
them together.

\begin{lemma}\libel{lemma:primary} 
  Let $A$ be a pentagon in a nonobtuse $P$-triangle.  Assume that the
  triangle edge opposite $\c_A$ has length at most $\kappa\sqrt{8}$.
  Then the $P$-triangle can be continuously deformed until $A$ is in
  primary contact, while preserving the following constraints: the
  deformation (1) maintains nonobtuseness, (2) is non-increasing in
  the edge lengths, and (3) keeps fixed the other two pentagons $B$
  and $C$.
\end{lemma}

\begin{lemma}\libel{lemma:primary2} 
  Let $A,B,C$ be pentagons in a nonobtuse $P$-triangle.  Assume that
  the triangle edges opposite $\c_A$ and $\c_C$ have length at most
  $\kappa\sqrt{8}$.  Then the $P$-triangle can be continuously
  deformed until $A$ is in primary contact, while preserving the
  following constraints: the deformation (1) maintains nonobtuseness,
  (2) fixes the edge length $\dx{AB}$ and does not increasing the area
  of the triangle, and (3) keeps fixed the other two pentagons $B$ and
  $C$.
\end{lemma}

\begin{proof} Fixing $B$ and $C$, we translate $A$ to contract the two
  edges of the triangle at $\c_A$, where we keep $d_{AB}$ fixed in
  Lemma~\rif{lemma:primary2}.  For a contradiction, assume that none
  of the primary contact conditions occur throughout the deformation.
  Continue the contractions, until $A$ contacts another pentagon, then
  continue by rotating $A$ about its center $\c_A$ to break the
  contact and continue.  Eventually, the assumption of nonobtuseness
  must be violated.  However, this triangle cannot be obtuse at
  $\c_A$, because the triangle edge lengths are at least $2 \kappa$,
  $2 \kappa$ with opposite edge at most $\kappa\sqrt{8}$. This is a
  contradiction.

  In the second lemma, after $A$ is rotated to break the contact, we
  translate $A$ along the circle such that $\c_A$ stays at fixed
  distance from $\c_B$.
\end{proof}

% page 30.
\begin{lemma}\libel{lemma:delta} 
  A subcritical $3C$-triangle does not have type $\Delta$.  In fact,
  such a $P$-triangle $T$ has area greater than $1.5$.
\end{lemma}

\begin{proof} The proof is computer-assisted.  The $3C$-triangles of
  type $\Delta$ form a three-dimensional configuration space.  The
  appendix (Section~\rif{sec:appendix}) introduces good coordinate
  systems for each of the various $3C$-triangle types.  We make a
  computer calculation of the area of the Delaunay triangle as a
  function of these coordinates.  We use interval arithmetic to
  control the computer error.  The lemma follows from these computer
  calculations.
\end{proof}

\begin{lemma}\libel{lemma:mid-172}  
  If a pentagon $A$ has midpointer contact with a pentagon $B$, then
  $\dx{AB} > 1.72$.
\end{lemma}

\begin{proof} Suppose a pointer vertex $\v_A$ of $A$ is the midpoint of an
  edge of pentagon $B$.  Rotating $A$ about the vertex $\v_A$, keeping
  $B$ fixed, we may decrease $\dx{AB}$ until $A$ and $B$
  have slider contact.  This determines the configuration of $A$ and
  $B$ up to rigid motion.  By the Pythagorean theorem, the
  distance between pentagon centers is
\[
\dx{AB} = \sqrt{(2\kappa)^2 + \sigma^2} \approx 1.72149 > 1.72.
\] % checked 2016/2/18 in calcs.ml
\end{proof}

\begin{lemma}\libel{lemma:172}
  If every edge of a $P$-triangle $T$ is at most $1.72$, then the
  triangle is not subcritical.
\end{lemma}

\begin{proof} This is a computer-assisted proof.\footnote{The constant
    $1.72$ is nearly optimal.  For example, in the notation of the
    appendix, the pinwheel with parameters $\alpha=\beta=\pi/15$,
    $x_\gamma = 0.18$ is subcritical equilateral with edge lengths
    approximately $1.72256$.} Such a triangle is nonobtuse.  By
  Lemma~\rif{lemma:primary}, we may deform $T$, decreasing its
  edge lengths and area, until each pentagon is in primary contact
  with the other two.  By the previous lemma, we may assume that the
  contact is not midpointer contact.  Thus, each pentagon has double
  contact or slider contact with the other pentagons.
  % checked 2016/9/5 in calcs.ml

  If the $P$-triangle does not have $3C$ contact, then obvious
  geometry forces one pentagon to have double contact and the other
  two pentagons to have slider contact (Figure~\rif{fig:172-slider}).
  The nonoverlapping of the pentagons forces one of slider contacts to
  be such that a sliding motion along the edges of contact decreases
  area and edge lengths of $T$.  Thus, the $P$-contact can be deformed
  until $3C$ contact results.

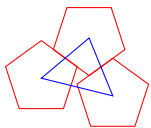
\tikzfig{172-slider}{We can slide pentagons 
(that is, translate them along their common edge segment) to decrease
lengths and the area of the Delaunay triangle}{
[scale=0.5]
\threepent{0}{0}{-90}{0.63}{-1.54}{90}{-1.27}{ -1.07}{90};
}

Now assume that the $P$-triangle has $3C$ contact.  We have classified
all $3C$ triangles.  We obtain the proof by expressing each type of
triangle in terms of explicit coordinates from the appendix
(Section~\rif{sec:appendix}) and computing bounds on the areas and
edge lengths of the triangles using interval arithmetic.  The result
follows.
\end{proof}

\begin{lemma}\libel{lemma:2C} 
  Let $T$ be a subcritical nonobtuse $P$-triangle with pentagons $A$,
  $B$, and $C$.  Then fixing $B$ and $C$, we may deform $T$ by moving
  the third pentagon $A$, without increasing the area of $T$, until
  $A$ has double contact (with $B$ and $C$).
\end{lemma}

\begin{proof}
  By Lemma~\rif{lemma:right}, the edge lengths of $T$ are at most
  $\kappa\sqrt8$.  By Lemma~\rif{lemma:primary}, we may assume that
  the pentagon $A$ has primary contact.  If the primary contact of a
  pentagon $A$ is slider contact, we may slide $A$ along the edge
  segment of contact in the direction to decrease the area of $T$
  until it has double contact.  If the contact of $A$ is midpointer
  contact, then we may rotate $A$ about the point of contact with a
  second pentagon $B$, in the direction to decrease the area of $T$
  until it has double contact.  These area-decreasing deformations
  never transform the nonobtuse subcritical triangle into a right
  triangle (Lemma~\rif{lemma:right}).
\end{proof}

\begin{lemma}\libel{lemma:a0}  
A nonobtuse $P$-triangle $T$ has area greater than $\ao$.
\end{lemma}

\begin{proof} 
  This is a computer-assisted proof.  We may assume for a contradiction
  that $T$ has area less than $\ao$.  In particular, it is
  subcritical.  By Lemma~\rif{lemma:2C}, we may assume that each
  pentagon has double contact with the other two, and that $T$ is
  $3C$.  We have classified all $3C$ triangles.  We obtain the proof
  by expressing each type of triangle in terms of explicit
  coordinates from the appendix (Section~\rif{sec:appendix}) and
  computing bounds on the areas and edge lengths of the triangles
  using interval arithmetic.  The result follows.
\end{proof}

\section{Computer Calculations}\libel{sec:calc}

The proofs of the theorems in this article rely heavily on computer
calculations.  These computer calculations are discussed further at
the end of the article in Section~\rif{appendix:cc}.  In this section,
we make use of the following lemmas, which are proved by computer.
(Although further discussion appears at the end of the article, there
is no circular reasoning involved in using those calculations here.)

\begin{lemma}[computer-assisted]\libel{NKQNXUN} \libel{calc:pseudo1} 
Let $(T_1,T_0)\in \PD$.  The edge shared between
$T_0$ and $T_1$ has length less than $1.8$. 
\end{lemma}

\begin{lemma}[computer-assisted]\libel{RWWHLQT}\libel{calc:pseudo2} 
  Let $(T_1,T_0)\in \PD$.  The longest edge of $T_0$ (that is, its
  egressive edge) has length greater than $1.8$.
\end{lemma}

\begin{lemma}\libel{calc:pseudo3}  
  Let $(T_1,T_0)\in \PD$.  The two edges of $T_0$ other than the
  longest edge have lengths less than $1.8$.
\end{lemma}

\begin{proof}
  The shared edge between $T_0$ and $T_1$ has length less than $1.8$
  by Lemma~\rif{calc:pseudo1}.  It has length at least $1.72$ by
  Lemma~\rif{lemma:172}.  If (for a contradiction) the third edge of
  $T_0$ has length at least $1.8$, then by the previous lemmas, its
  three edges have lengths at least $1.72$, $1.8$, $1.8$.  Then
\[
\area\{T_1,T_0\} > \ao + \area(1.72,1.8,1.8) 
> \ao + (\acrit + \epsN) = 2\acrit.
\]
% checked in calc.ml 2016/9
This area inequality contradicts a defining property of pseudo-dimers.
\end{proof}

\begin{lemma}[computer-assisted]\libel{BXZBPJW}\libel{calc:pseudo-area}
Let $(T_1,T_0)\in\PD$.  Then $\area\{T_0,T_1\} \ge 2\acrit - \epsM$.
\end{lemma}

\begin{lemma}[computer-assisted]\libel{JQMRXTH}\libel{calc:pseudo-area3}
Let $(T_1,T_0)\in\PD$. Assume $T_0\Ra T_-$.  
Then $\area\{T_0,T_1,T_-\} > 3\acrit + \epsM$.
\end{lemma}

\begin{corollary}\libel{lemma:egress'} 
Let $(T_1,T_0)\in\PD$.  Assume that $T_0\Ra T_-$.  Then
$\area(T_-) > \acrit+\epsM$.
\end{corollary}

\begin{proof}  
  By Lemma~\rif{calc:pseudo-area3} and the definition of pseudo-dimer,
\[
\area(T_-) = \area\{T_0,T_1,T_-\} - \area\{T_0,T_1\} 
> (3\acrit + \epsM) - 2\acrit = \acrit + \epsM.
\]
\end{proof}

\begin{corollary}\libel{lemma:m2-area}
  Suppose that $m_-(T_-)>0$.  Then $\area(T_-) > \acrit + \epsM$.
\end{corollary}

\begin{proof} 
  If $m_-(T_-)>0$, there exists $(T_1,T_0)\in\PD$ such that $T_0\Ra
  T_-$.  The result follows from the previous corollary.
\end{proof}

\begin{definition}[long isosceles]
  We say that a triangle is {\it long isosceles} if the two longest edges of
  the triangle have equal length.  We include equilateral triangles as
  a special case of long isosceles.
\end{definition}

\begin{definition}[O2C]
  We say that a triangle $T=T_0$ or $T=T_1$ in a dimer pair or a
  pseudo-dimer pair has outside double contact $(O2C)$ if the pentagon
  $A$ at the vertex of $T$ that is not shared with the other triangle in the
  pair has double contact.
\end{definition}

\begin{lemma}[computer-assisted]\libel{KUGAKIK}\libel{calc:dimer-isosc}
  Let $(T_1,T_0)\in DP$.  Then $T_1$ is not both $O2C$ and long
  isosceles.
\end{lemma}

\begin{definition}[large angle]\libel{def:large}  
  Let $T$ be a $P$-triangle.  Let $e$ be an edge of the triangle with
  pentagons $A$ and $B$ at its endpoints.  Let $\alpha = \alpha(T,e)$
  be the angle between the edges of the two pentagons $A$ and $B$.  (See
  Figure~\rif{fig:large}.)  Modulo $2\pi/5$, we can assume that
  $\alpha\in [0,2\pi/5]$.  We say that the angle is {\it large} along
  $(T,e)$ if $\pi/5<\alpha < 2\pi/5$.  
% N.B.
% The way the code deals with alpha=2pi/5 is somewhat subtle.
% The test in the code is just alpha >> pi15, and periodize_pent
% is called on the angle, which normalizes to [0,2pi/5].
% If the angle equals 2pi/5, then periodize pent splits the interval into two,
% one interval containing 0 and the other containing 2pi/5.
% In this case the test alpha >> pi15 fails.
\end{definition}

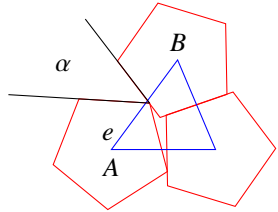
\tikzfig{large}{A Delaunay triangle $T$ 
with a large angle $\alpha$ along $(T,e)$.}
{
\begin{scope}[scale=0.8]
\threepent{0.00}{0.00}{50.58}{1.10}{1.49}{181.07}{1.73}{0.00}{217.07};
\draw (0,0) node[anchor=north] {$A$};
\draw (0,0) node[anchor=south] {$e\ $};
\draw (1.1,1.49) node[anchor=south] {$B$};
\draw[black] (50.58:1) -- ++ ($2*(50.58+72:1) - 2*(50.58:1)$);
\draw[black] (50.58:1) -- ++ ($2*(1.1,1.49)+2*(181.07:1) - 2*(50.58:1)$);
\draw (-0.8,1.4) node {$\alpha$};
\end{scope}
}

Let $T$ be a $P$-triangle and let $T'$ be the adjacent $P$-triangle
along edge $e$.  We have the invariant
$\alpha(T,e)+\alpha(T',e)=2\pi/5$.  Hence if the angle is large along
$(T,e)$ then it is not large along $(T',e)$.

\begin{lemma}[computer-assisted]\libel{FHBGHHY}\libel{calc:large}
  Let $\{T_0,T_1\}$ be given $P$-triangles (not necessarily a
  pseudo-dimer) such that $\area(T_1)\le \acrit$ and $T_1\Ra T_0$.
  Assume that there is a nonshared edge $e$ of $T_0$ of length greater
  than $1.8$ and such that the angle is not large along $(T_0,e)$.  Then
  $\area\{T_0,T_1\} > 2\acrit + \epsM$.
\end{lemma}

\begin{corollary}\libel{lemma:large}
  If $(T_1,T_0)\in \PD$ and $e$ is the longest (that is, egressive)
  edge of $T_0$.  Then the angle is large along $(T_0,e)$.
\end{corollary}

\begin{proof} 
  If the angle is not large, then we can apply the lemma to find tht
  the area of the pseudo-dimer is greater than $2\acrit + \epsM$,
  which contradicts one of the defining properties of a pseudo-dimer.
\end{proof}

\begin{lemma}[computer-assisted]\libel{HUQEJAT}\libel{calc:pent3}
  Let $T_1^i\Ra T_0$ and $\area(T_1^i)\le\acrit$  for
  distinct $P$-triangles $T_1^0$ and $T_1^1$.  Then
  $\area\{T_0,T_1^0,T_1^1\} > 3\acrit + \epsM$.
\end{lemma}

\begin{lemma}[computer-assisted]\libel{QPJDYDB}\libel{calc:pent4}
  Let $T_1^i\Ra T_0$ and $\area(T_1^i)\le\acrit$ for distinct
  $P$-triangles $T_1^0$, $T_1^1$, and $T_1^2$.  Then
  $\area\{T_0,T_1^0,T_1^1,T_1^2\} > 4\acrit$.
\end{lemma}

\section{Dimer Pairs}\libel{sec:dimer-pair}

The purpose of this section is to give a proof of the following
theorem.  This theorem is the principal optimization problem of this
article in the sense that all other optimizations deal with
configurations that are far from optimal.

\begin{theorem}\libel{thm:dimer-area}
  Let $(T_1,T_0)$ be a dimer pair.  (In particular, we assume that
  $T_1$ is subcritical, that $\area\{T_0,T_1\}\le 2\acrit$ and that
  $T_0$ and $T_1$ share a common longest edge.)  Then $(T_1,T_0)$ is
  the ice-ray dimer of area exactly $2\acrit$.
\end{theorem}

The proof will fill the entire section.  The strategy of the proof is
to give a sequence of area decreasing deformations to $(T_1,T_0)$,
until the ice-ray dimer is reached.

We fix notation that will be used throughout this section. Let
$(T_1,T_0)$ be a dimer pair.  The $P$-triangle $T_1$ has a pentagon
centered at each vertex.  We label the pentagons of $T_1$ as $A$, $B$,
$C$, with $A$ and $C$ shared with $T_0$.  We call $B$ the outer
pentagon of $T_1$.  Similarly, we label the pentagons of $T_0$ as $A$,
$C$, $D$, with outer pentagon $D$ of $T_0$.

\begin{lemma}  
  Let $(T_1,T_0)$ be a dimer pair, then every edge of $T_1$ and $T_0$
  has length less than $\kappa\sqrt8$.  In particular $T_1$ and $T_0$
  are both acute (and not just merely nonobtuse).
\end{lemma}

\begin{proof}  
  The shared edge between $T_0$ and $T_1$ is the common longest edge
  of the two triangles.  It is enough to show that this edge has
  length less than $\kappa\sqrt8$.  This is an edge of a subcritical
  triangle $T_1$. The result follows from Lemma~\rif{lemma:right}.
\end{proof}

In particular, area non-increasing deformations of a dimer pair, never
transform an acute triangle into a right or obtuse triangle.  In other
words, the nonobtuseness constraint in the definition of a dimer pair
is never a binding constraint in a deformation.

The deformation of a general dimer pair to the ice-ray dimer takes
place in several stages.  We give a summary of the stages here, before
going into details.  Here is the proof sketch:

\begin{enumerate}
\item We deform the dimer pair so that each of $T_0$ and $T_1$ is
  $O2C$ or long isosceles.
\item We deform so that each of $T_0$ and $T_1$ is $O2C$.
\item We show that both triangles are triple contact.
\item Working with triple contact triangles, we compute that the
  condition $\area\{T_0,T_1\}\le 2\acrit$ implies that $(T_1,T_0)$
  lies in a small explicit neighborhood of the ice-ray dimer.
\item We construct a curve $\Gamma$ in the configuration space of
  dimer pairs, with parameter $t$ such that $t=0$ defines the ice-ray
  dimer.
\item Working with triple contact triangles in a small explicit
  neighborhood of the ice-ray dimer, and for some small explicit
  constant $M$, each dimer pair can be connected by a path (in the
  dimer configuration space) to a dimer on the curve $\Gamma$ with
  parameter $|t|<M$.  A computation shows that the area of the dimer
  decreases along the path to $\Gamma$.  Thus, every area-minimizing
  dimer pair lies on the curve $\Gamma$.
\item The unique global minimum of the area function along $\Gamma$
  occurs at $t=0$; that is, the ice-ray dimer is the unique global
  minimizer along $\Gamma$ for $|t|<M$.
\end{enumerate}

\subsection{reduction to $O2C$ or long isosceles}

In this section, we show that each of $T=T_0$ and $T=T_1$ can be
deformed in an area decreasing way until $T$ is either $O2C$ (that is,
the outer pentagon has double contact) or long isosceles (that is, the
two longest edges of the triangle have the same length).

To show this, we assume that $T$ and its deformations are not long
isosceles; that is, it and its deformations have a unique longest edge
that is shared with the other triangle.  Fixing the two pentagons of
$T$ ($A$ and $C$) along the shared edge, we show we can deform the
outer pentagon until it has $2C$ contact.

This is easy to carry out.  By Lemmas~\rif{lemma:primary}
and~\rif{lemma:primary2}, we can move $B$ in an area decreasing way
until $B$ has primary contact.  We can continue to move the outer
pentagon by translation (meaning no rotation) that preserves contact
with either $A$ or $C$ and that is non-increasing in triangle area
until double contact is achieved.  This is $O2C$.  We do this for both
$T=T_0$ and $T=T_1$.

\subsection{reduction to $O2C$}\libel{sec:O2C}

In this subsection we show that each of $T=T_0$ and $T=T_1$ can be
deformed in an area decreasing way so that it is $O2C$.

We begin with the case $T=T_0$.  If the deformations in the previous
section made $T_0$ into a long isosceles triangle, then $(T_1,T_0)$ is
a boundary case that can also be classified as a pseudo-dimer.  By
earlier calculations, the longest edge of a pseudo-dimer is greater
than $1.8$ and has strictly greater length than the shared edge
between $T_0$ and $T_1$.  Thus, it is not long isosceles.

We now consider the case $T=T_1$.  In view of the reductions of the
previous subsection, we may assume that $T_1$ has primary contact and
that $T_1$ is long isosceles.  We may further assume that translation
of the outer pentagon $B$ while maintaining contact (with  $A$ or $C$)
in an area decreasing direction would violate the constraint that the
longest edge of $T_1$ is the shared edge.  (In other words, the
translation that decreases area would increase the edge length of the
long nonshared edge.)  For a contradiction, we may assume that the
primary contact is not $O2C$.

We claim that these conditions force $T_1$ not to be subcritical.
This is contrary to the defining conditions of a dimer pair.  Thus, we
complete this stage of the proof by proving non-subcriticality. For
the rest of the proof, we disregard $T_0$.

The proof is computer assisted.  We deform $T_1$ in an area decreasing
way into a configuration that can be easily computed.  Without loss of
generality, we assume that the outer pentagon $B$ is in contact with
pentagon $A$.  Because we are now disregarding $T_0$, we may deform the
triangle $T_1$ by moving $C$, preserving the long isosceles constraint
and decreasing area, until $C$ has primary contact.  In particular,
$C$ is in contact with $A$ or $B$.

We consider two cases, depending on whether the primary contact of $B$
has slider contact or midpointer contact with $A$.  We need two
non-anomaly lemmas, one for slider contact and one for midpointer
contact.  In both cases, the lemmas imply that the edge of contact
between $A$ and $B$ (whether slider or midpointer) is one of the long
edges of the isosceles triangle.

\begin{lemma}[slider-non-anomaly]  
  Let $T$ be a subcritical nonobtuse $P$-triangle with pentagons $A$,
  $B$, and $C$.  Assume that $B$ has slider contact with $A$.  Then
  the translation of $B$ along the slider contact in the direction to
  decrease the area also decreases the length $\dx{BC}$.
\end{lemma}

\begin{proof} 
  Assume to the contrary that the translation is increasing in
  $\dx{BC}$.  Choose coordinates so that the $x$-axis passes through
  $\c_A$ and $\c_C$, with the origin at $\c_A$, with $\c_C$ in the
  right half-plane, and with $\c_B$ in the positive half-plane
  (Figure~\rif{fig:slider}).  Our contrary assumption means the the
  line $\lambda$ through the edge of contact between $A$ and $B$ has
  positive slope.  Slider contact implies that the center $\c_B$ lies
  on the line $\lambda'$ parallel to $\lambda$ at distance $2\kappa$
  from the origin.  Every point on $\lambda'$ either has
  $y$-coordinate at least $2\kappa$ or negative $x$-coordinate.  If
  the $y$-coordinate is at least $2\kappa$ the triangle is not
  subcritical.  (The area is at least $2\kappa^2>\acrit$.) If the
  $x$-coordinate of $\c_B$ is negative, then the triangle $T$ is
  obtuse.
\end{proof}

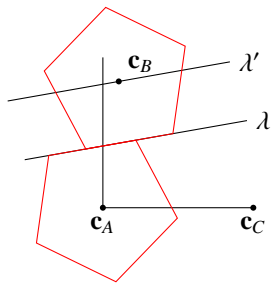
\tikzfig{slider}{The center $\c_B$ of the nonobtuse triangle lands in
  the first quadrant and gives a triangle $(\c_A,\c_B,\c_C)$ of height
  and base both at least $2\kappa$.}
{
[scale=1.0]
\draw (0,0) -- (2,0);
\draw (0,0) -- (0,2);
\coordinate (V) at ($(0,0) + (-90+2*72+10:1)   $);
\draw (V) -- ++ (10:1.5) node[anchor=west] {$\lambda$};
\draw (V) -- ++ (180+10:1.5);
\coordinate (B) at ($(V) + (10:0.5) + (90+10- 2*72 - 180:1)$);
\draw (B) -- ++ (10:1.5) node[anchor=west] {$\lambda'$};
\draw (B) -- ++ (180+10:1.5);
\penp{B}{90+10};
\pen{0}{0}{-90+10};
\smalldot{B};
\smalldot{0,0};
\smalldot{2,0};
\draw (B) node [anchor=south west] {$\c_B$};
\draw (0,0) node [anchor=north] {$\c_A$};
\draw (2,0) node [anchor=north] {$\c_C$};
}

\begin{lemma}[midpointer-non-anomaly]  
  Let $T$ be a subcritical nonobtuse $P$-triangle with pentagons $A$,
  $B$, and $C$.  Assume that $A$ has midpointer contact with $B$, with
  $A$ pointing to $B$ at $\v_{AB}$.  Then the rotation of $B$ about
  the point of contact in the direction to decrease the area also
  decreases the distance $\dx{BC}$.
\end{lemma}

\begin{proof} 
  Suppose to the contrary that the rotation is increasing in
  $\dx{BC}$.  Choose coordinates so that the $x$-axis passes through
  $\c_A$ and $\c_C$, with origin at $\c_A$, with $\c_C$ in the right
  half-plane, and with $\c_B$ in the first quadrant.  We claim that
  the distance from $\c_B$ to the the $x$-axis is at least $2\kappa$
  so that the area of the triangle is at least $2\kappa^2 > \acrit$,
  contrary to the assumption that the triangle $T$ is subcritical.  To
  prove the claim, we disregard the pentagon $C$.  We translate $B$
  directly downward, rotating $A$ as needed about $\c_A$ so that it
  maintains pointer contact with $B$.  Eventually, slider contact is
  established between $A$ and $B$, and the configuration falls into
  the setting of the previous lemma.
\end{proof}

As a corollary of these two lemmas, if $d_{AB}$ is not a longest edge,
then there exists a deformation decreasing area and $d_{BC}$.  This
allows us to reduce the long isosceles triangle to a triangle with
double contact and such that a long edge runs from $\c_A$ to $\c_B$.
The contact type between $A$ and $B$ is either slider contact or
midpointer contact.  We choose coordinates and
compute\footnote{calculations iso\_2C and iso\_2C'} with interval
arithmetic to show that no such triangle is subcritical.  This
completes this reduction.

\subsection{reduction to triple contact}\libel{sec:squeeze}

In this stage, we initially assume that both triangles $T_0$ and $T_1$
are $O2C$.  We deform so that both triangles are triple contact.

We briefly describe the argument.  Our deformations will preserve the
$O2C$ contacts.  By an argument made in the second paragraph of
Section~\rif{sec:O2C}, we may assume without loss of generality that
the triangle $T_0$ is not long isosceles.  By
Lemma~\rif{calc:dimer-isosc}, we have that $T_1$ is not long
isosceles.  Assuming that neither triangle is long isosceles, we show
that both triangles can be brought into triple contact.  Because both
triangles are already $O2C$, this amounts to decreasing the edge
between $\c_A$ and $\c_C$ until the pentagons $A$ and $C$ come into
contact.  This deformation consists of a translation of all four
pentagons in a motion we call {\it squeezing}.  It suffices to
describe the deformation separately on each triangle $T_0$ and $T_1$
and to prove that this deformation decreases area.

Let $T$ be a triangle with double contact at $B$.  We assume that
$\c_A$ and $\c_C$ lie on the $x$-axis with $\c_A$ to the left of
$\c_C$, and with $\c_B$ in the upper half-plane.  If $\c_A$ is free to
translate to the right or if $\c_C$ is free to translate to the left
without overlapping pentagons, then we do so.  We assume that we are
not in this trivial case.

The squeezing deformation is defined as a motion that translates $B$
directly upward away from the $x$-axis, while translating $A$ to the
right and $C$ to the left to maintain double contact at $B$.  Note
that $A$ and $C$ move by translation along the $x$-axis.  It is clear
that by adjusting the rates at which $B$ (on $T_1$) moves upward and
$D$ (on $T_0$) moves downward, the motions of $T_0$ and $T_1$ are
concordant and give a motion of a dimer pair.

\begin{lemma}  
  Let $T$ be a nonobtuse $P$-triangle with pentagons $A$, $B$, and
  $C$, where $B$ has double contact.  Assume that the area of $T$ is
  at most $\acrit+\epsN$.  Assume that $T$ is not in the trivial
  situation of free translation mentioned above.  Then the squeezing
  deformation decreases area.
\end{lemma}

\begin{proof}  
  We analyze the effect on the area by moving $B$ directly upward by
  $\Delta y >0$, $A$ to the right by $\Delta x_A >0$ and $C$ to the
  left by $\Delta x_C >0$.  Recall that $\dx{XY} = \dd{X}{Y}$, and let
  $\arc_X$ be the angle of the triangle at $\c_X$.  The area of $T$ is
  $\area(T) = d_{AC} d_{AB} \sin(\arc_A)/2$.  The transformed area is
  $(d_{AC}-\Delta x_A - \Delta x_C)(d_{AB}\arc_A + \Delta y)/2$.
   Let $\sigma_A = \Delta y/\Delta x_A$ and $\sigma_C = \Delta y
  /\Delta x_C$.  Passing to the limit as $\Delta y \mapsto 0$, we find
  that squeezing decreases the area exactly when
\[
d_{AC} < \left(\frac{1}{\sigma_A}
+\frac{1}{\sigma_C}\right)d_{AB}\sin(\arc_A).
\]
It is enough to prove this inequality.  We do this with a computer
calculation\footnote{calculation squeeze\_calc} using interval
arithmetic.

We defined $\sigma_A$ and $\sigma_C$ as derivatives, but in fact no
differentiation is required.  For example, consider $\sigma_A$.  The
pentagon $B$ has contact with $A$.  Thus, $B$ points into $A$ or $A$
points into $B$.  If $A$ points into $B$, then the point of contact
lies along an edge $e$ of $B$.  The squeeze transformation translates
$A$ and $B$ maintaining the contact.  Viewed from a coordinate system
that fixes $B$, the squeezing lemma translates $A$ parallel to the
line through $e$.  That is, $\sigma_A$ is simply the absolute value of
the slope of the line through $e$.  Elementary coordinate calculations
described in the coordinate section of this article give an explicit
formula for the slope of this line.  There are two cases, depending on
which pentagon points to the other.  There are no difficulties in
carrying out the computer calculations with these explicit formulas.

In a degenerate situations, the pentagons $A$ and $B$ might have
vertex to vertex contact. But even in this degenerate case, the
squeezing deformation determines an edge $e$ of $A$ or $B$ that
determines the slope $\sigma_A$.
\end{proof}

The triangle $T_1$ is subcritical.  The area of $T_0$ is given by
\[
\area(T_0) = \area\{T_0,T_1\} - \area(T_1) 
< 2\acrit - \ao = \acrit+\epsN.
\]
Thus, the assumption of the lemma holds.  We continue the squeezing
deformation until $A$ comes into contact with $C$.  This completes the
reduction to $3C$ contact.

\subsection{reduction to a small neighborhood of the ice-ray
  dimer}\libel{sec:nbd}

From this stage forward, $T_0$ and $T_1$ are both triangles with
triple contact.  We show that the condition $\area\{T_0,T_1\}\le
2\acrit$ implies that $(T_1,T_0)$ lies in a small explicit
neighborhood of the ice-ray dimer.

The shared pentagons $A$ and $C$ are in contact.  By symmetry, we may
assume that $A$ points into $C$.  We have classified all $P$-triangles
with triple contact in Section~\rif{sec:contact}.  At this stage, we
rely heavily on this classification.  The triple contact type $\Delta$
has area at least $1.5 > \acrit+\epsN$ by Lemma~\rif{lemma:delta}.
This is too large an area to be part of a dimer pair.  There are eight
combinatorial ways that the pair $(A,C)$ with $A$ pointing to $C$ can
be extended to a triple contact triangle: a pinwheel, a pin-$T$
junction, an $LJ$-junction (3 ways), and a $TJ$-junction (3 ways).
There are three ways of extending the edge along $(A,C)$ to an
$LJ$-junction depending on which of the three triangle edges of the
$LJ$-junction is placed along $(A,C)$.  A similar remark applies to
$TJ$-junctions.  A pinwheel has cyclic symmetry, so it gives rise to a
single case.

We need an argument to show that only a single combinatorial type of
pin-$T$ triangle needs to be considered.  In this paragraph only, we
shift notation and refer to labels on pentagons in the pin-$T$
configuration shown in Figure~\rif{fig:cord-pint}.  We claim that if
the area the triangle has area less than $\acrit+\epsN$, then the edge
length $\dx{BC}$ in that figure is the unique longest edge of the
triangle.  This claim is established by computer calculation.  (See
Section~\rif{sec:one-pintx}.)  This means that the shared edge of the
pin-$T$ triangle is determined.
%calculation dimer.ml:one\_pintx.

Because $T_0$ and $T_1$ both have triple contact, the dimer pair
$(T_1,T_0)$ lies in a four-dimensional configuration space.  The
strategy of the proof is to check by computer that there does not
exist a dimer pair outside a small explicit neighborhood of the
ice-ray dimer.  In other words, the area constrains $\area(T_1)\le
\acrit$ and $\area\{T_0,T_1\}\le2\acrit$ are impossible to satisfy
when the longest edge on both triangles is the shared edge.  We run
over $64 = 8\times 8$ cases depending on the combinatorial types of
$T_0$ and $T_1$.  In each case, $T_1$ runs over a three-dimensional
configuration space.  Most of the $64$ cases do not contain the
ice-ray dimer.  In these cases, it is not necessary to specify a small
explicit neighborhood.  In the cases that do contain the ice-ray
dimer, the neighborhood is desribed in local coordinates.
% calculation in dimer.ml

We say a word about the coordinate system used to carry out these
calculations.  The triangles $T_0$ and $T_1$ separately have good
coordinate systems described in Section~\rif{sec:coords}.  Three
variables each running over a bounded closed interval parameterize the
configuration space for each configuration type.  These coordinates
are always numerically stable.  The quantities $x_{AC}$ and
$\alpha_{AC}$ (that parameterize two pentagons in contact)
can be computed from $T_0$ or $T_1$ alone.  A natural
way to try to parameterize the dimer pairs of a given combinatorial
type is to use the three coordinates $x_1,x_2,x_3$ from
Section~\rif{sec:coords} for $T_1$, then to choose an appropriate
quantity $x_4$ on $T_0$ such that $T_0$ is determined by $x_{AC}$,
$\alpha_{AC}$ (viewed as functions of $x_1,x_2,x_3$) and $x_4$.  Then
$x_1,\ldots,x_4$ give coordinates for the dimer pair that can be used
to do the computer calculations.

Usually, this strategy works, but in a few situations, there is no
obvious way to pick the fourth coordinate $x_4$ in a numerically
stable way.  Fortunately, we can give a characterization of all
situations where it is difficult to pick a numerically stable
coordinate $x_4$.  These are expressed as conditions on the
combinatorial type of $T_0$ and as bounds on $x_{AC}$ and
$\alpha_{AC}$.  In each situation, we use good coordinates provided by
Section~\rif{sec:coords} to show that these conditions of numerical
instability force $T_0$ to have area greater than $\acrit+\epsN$,
which is incompatible with the conditions defining a dimer pair.
These too are computer calculations.  (See Sections~\rif{sec:one-ljx}
and~\rif{sec:one-tjx}.)  Thus, we are justified in excluding a few
situations, where coordinates become unstable.  (The underlying source
of numerical instability is our use of the law of sines
\[
\frac{a}{\sin\alpha} = \frac{b}{\sin\beta}
\]
to compute the length of one triangle edge $a$ in terms of another
$b$, which encounters instability for $\beta$ near $0$.)
% calculations one_ljx one_tjx in dimer.ml

In terms of the local coordinates of Section~\rif{sec:coords}, if
$x_1=x_2=x_3=x_4=0$ defines the ice-ray dimer, then the explicit small
neighborhood we exclude is given by $|x_i|\le 0.01$, for $i=1,2,3,4$.
For example, in the pinwheel type on $T_1$, we have
$(x_1,x_2,x_3)=(\alpha,\beta,x_\gamma)$ as given in
Section~\rif{sec:pint}, and the fourth variable $x_4$ is determined by
the type of $T_0$.  These computer calculations are used to complete this
stage of the optimization.

\subsection{defining a curve}\libel{sec:gamma}

At this stage, we define a curve $\Gamma$ in the configuration space
of triple contact dimer pairs such that the ice-ray dimer is defined
by parameter value $t=0$.  We represent the ice-ray dimer by a pair of
triple contact triangles with shared pentagons $A$ and $C$, where $A$
points to $C$.  The   curve is described by the  shear motion that
slides along all four edges of contact, illustrated in Figures
\ref{fig:Gamma} and \ref{fig:Gamma'}.
The parameter $t$ is the signed distance between the pointer vertex
$A\to C$ and the midpoint of the receptor edge on $C \leftarrow A$.
See Figure~\rif{fig:Gamma'}.

\tikzfig{Gamma'}{Thre configurations along the curve $\Gamma$.}
{
\begin{scope}[scale=0.5,xshift=-6cm]
\threepentnoD{0.00}{0.00}{79.87}{0.93}{1.36}{259.87}{1.83}{0.00}{223.87};
\threepentnoD{0.00}{0.00}{-64.13}{1.03}{-1.50}{-244.13}
{1.83}{0.00}{-208.13};
\draw (0,0) node {$A$};
\draw (1.83,0) node {$C$};
\draw (0.5,-3.5) node {$t= 0.25$};
\smalldot{79.87-72:1};
\smalldot{$(1.83,0)+(223.87-3*72+180:0.809)$};
\end{scope}
\begin{scope}[scale=0.5]
\threepentnoD{0.00}{0.00}{72.00}{0.96}{1.43}{252.00}{1.81}{0.00}{216.00};
\threepentnoD{0.00}{0.00}{-72.00}{0.96}{-1.43}{-252.00}{1.81}{0.00}{-216.00};
\draw (0,0) node {$A$};
\draw (1.81,0) node {$C$};
\smalldot{1,0};
\draw (0.5,-3.5) node {$t=0$};
\end{scope}
\begin{scope}[scale=0.5,xshift=6cm]
\threepentnoD{0.00}{0.00}{-79.87}{0.93}{-1.36}{-259.87}{1.83}{0.00}{-223.87};
\threepentnoD{0.00}{0.00}{64.13}{1.03}{1.50}{244.13}{1.83}{0.00}{208.13};
\draw (0,0) node {$A$};
\draw (1.83,0) node {$C$};
\draw (0.5,-3.5) node {$t= -0.25$};
\smalldot{-79.87+72:1};
\smalldot{$(1.82,0)+(-223.87+3*72+180:0.809)$};
\end{scope}
}
%# format_pinwheelAB zero zero (sigma - m 0.25);;
%- : string = "{0.00}{0.00}{79.87}{0.93}{1.36}{259.87}{1.83}{0.00}{223.87}"
%# format_pinwheelAB zero zero (sigma + m 0.25);;
%- : string = "{0.00}{0.00}{64.13}{1.03}{1.50}{244.13}{1.83}{0.00}{208.13}"

In terms of the coordinates $(x_\alpha,\alpha)$ of
Section~\rif{sec:coords} for two pentagons in contact, the relative
position of $A$ and $C$ is described by the curve $\alpha=\pi/5$ and
$x_\alpha = \sigma+t$.

We call a $\Gamma$-dimer to be a dimer on the curve $\Gamma$.  We
define a $\Gamma$-triangle to be a $P$-triangle that occurs as one of
the two triangles in a $\Gamma$-dimer.

\subsection{reduction to the curve}\libel{sec:curve}

At this stage, we show that we can reduce to points on the curve in
the following sense.  Working with triple contact triangles in a small
explicit neighborhood $U$ of the ice-ray dimer ($|x_i|<M$, where
$M=0.01$, for $i=1,2,3,4$ in appropriate coordinates), each dimer pair
$(T_1,T_0)$ can be connected by a path (in the dimer configuration
space) to a $\Gamma$-dimer with parameter $|t|<0.01$.  The area
function is decreasing along this path.  We have two tasks.  First, we
construct the path $P$, and then we show that the area function
decreases along the path.  We will use $s$ as a local parameter on the
path ($s\mapsto P(s)$), which we need to keep separate from the
parameter $t$ for $\Gamma$.  We construct a path such that $s=0$
determines the initial dimer pair $(T_1,T_0)$ and such that the path
$P$ is defined for $s\in [0,s_0]$, for some $s_0>0$.

Section~\rif{sec:shared} accomplishes these tasks.  The path $P$ is
constructed and it terminates on $\Gamma$.  If the initial point of
the path lies in the neighborhood $U$, then the path stays in $U$
and terminates at a point on $\Gamma$ with parameter $|t|<M$.

Finally, we need to check that the area function is decreasing. This,
we prove by a computer calculation of the derivative of the area
function along the path at $s=0$.  For this, we use automatic
differentiation algorithms, as described in
Section~\rif{sec:auto-diff}.  This completes the reduction to points
on the curve $\Gamma(t)$.
% calculations in autodiff.ml

\subsection{global minimization along the curve}

The final stage of the proof of Theorem~\rif{thm:dimer-area} is the
minimization of the area of a $\Gamma$-dimer, for parameters
$|t|<0.01$.  We have now reduced to an optimization problem in a
single variable that is relatively easy to solve.  The area function
is obviously analytic.  By automatic differentiation, we take the
second derivative of the area function as a function of $t$ and
calculate that it is always positive.  Thus, the area function has a
unique global minimum on $|t|\le 0.01$.  By symmetry in the underlying
geometry, it is clear that the area function has derivative zero along
$\Gamma$ at $t=0$.  (See Figure~\rif{fig:Gamma}.) The global
minimum is therefore given at $t=0$, which is the ice-ray dimer.

\section{Pseudo-Dimers}

In this section, we determine the structure of pseudo-dimers and
specialize the function $b$ to this context.

For a nonobtuse triangle, the area is a monotonic function of its edge
lengths.  This makes it easy to give lower bounds on triangle areas.
Here are some simple area calculations that will be used.
\begin{align*}
\area(1.8,1.8,1.8) &> \acrit + 0.112\\
\area(1.8,1.8,1.72) &> \acrit + 0.069\\
\area(1.8,1.72,1.72) &> \acrit + 3\epsM\\
\area(1.8,1.8,2\kappa) &> \acrit + \epsM\\
\area(\kappa\sqrt8,2\kappa,1.72) &> \acrit +\epsN.
\end{align*}
% calculations in calc.ml. Checked 9/2016.

\begin{lemma}\libel{lemma:pd-m2}  
If $(T_1,T_0)\in\PD$, then $m_-(T_0)=0$.
\end{lemma}

\begin{proof} 
  Assume for a contradiction that $m_-(T_0)>0$.  Then there exists
  $(T'_1,T_0')\in\PD$ and $(T'_0,T_0)\in\M$.  The shared edge between
  $T_0$ and $T'_0$ has length greater than $1.8$.  This is the
  egressive edge $e$ of $T_0$.  This is impossible by
  Lemma~\rif{calc:large}, which states that the condition of having a
  large angle is not symmetrical for two pseudo-dimers $(T_1',T_0')$
  and $(T_1,T_0)$.
\end{proof}

\begin{corollary}\libel{lemma:m1m2}  
  For every triangle $T$, either $m_+(T)=0$ or $m_-(T)=0$.
\end{corollary}

\begin{proof}  
  If $m_+(T)>0$, then there exists some some $(T_1,T)\in\PD$.  The
  lemma gives $m_-(T)=0$.
\end{proof}

\begin{lemma}[pseudo-dimer disjointness]  
  Assume $(T_1,T_0)\in \PD$ and $(T'_1,T'_0)\in\PD$ and
  $\{T_0,T_1\}\cap\{T'_0,T'_1\} \ne \emptyset$.  Then $T_0 = T'_0$.
\end{lemma}

\begin{proof}  
  If we dismiss the other three cases $T_1=T'_1$ and $T_1=T'_0$ and
  $T_0=T'_1$ of a nonempty intersection, then the conclusion $T_0 =
  T'_0$ will stand.

  Assume first that $T_1 = T'_1$.  Then $T_1\Ra T_0$ and $T_1\Ra
  T'_0$, which gives $T_0=T'_0$.

  Next assume that $T_1=T'_0$.  We have $T'_1\Ra T'_0 = T_1 \Ra T_0$.
  By the calculations above (Lemma~\rif{calc:pseudo1} and
  Lemma~\rif{calc:pseudo2}), this puts incompatible constraints on the
  length of the shared edge between $T_0$ and $T_1$.  It must have
  length less than $1.8$ and greater than $1.8$.

  The case $T_0=T'_1$ follows from the previous case by symmetry.
\end{proof}

\begin{lemma}[pseudo-dimer-obtuse]\libel{lemma:pd-obtuse}
  Assume $(T_1,T_0)\in\PD$.  Then each edge of $T_0$ and $T_1$ has
  length less than $\kappa\sqrt8$.  In particular, if $T$ is obtuse
  then $T\nRa T_0$ and $T \nRa T_1$.
\end{lemma}

\begin{proof} 
  The shared edge between $T_0$ and $T_1$ has length less than $1.8 <
  \kappa\sqrt8$ (Lemma~\rif{calc:pseudo1}).  The shared edge is the
  longest edge of $T_1$, so that each edge of $T_1$ has length less
  than $1.8$.  The only possibility is the egressive edge of $T_0$.
  But if the egressive edge of $T_0$ has length at least
  $\kappa\sqrt8$, then
\[
\area\{T_1,T_0\} > \ao + \area(\kappa\sqrt8,2\kappa,1.72) 
> \ao + (\acrit + \epsN) = 2\acrit.
\]
This contradicts the area condition in the definition of pseudo-dimer.
This completes the first claim of the lemma.

If $T$ is obtuse, then its longest edge has length at least
$\kappa\sqrt8$.  If $T\Ra T_i$, then the shared edge has length at
least $\kappa\sqrt8$, contrary to the first claim of the lemma.
\end{proof}

\begin{corollary}\libel{lemma:pd-n2}
If $(T_1,T_0)\in\PD$, then $n_-(T_0)=0$.
\end{corollary}

\begin{proof} 
  Suppose for a contradiction that $n_-(T_0)>0$.  Then $(T_+,T_0)\in
  \N$ for some $T_+$.  Because $T_0$ is nonobtuse,
  Lemma~\rif{lemma:obtuse-b-arrow} implies that there exists some
  obtuse triangle $T$ such that $T\Ra T_0$.  This is contrary to the
  previous lemma.
\end{proof}

\begin{lemma}\libel{lemma:pd-1} 
  Let $(T_1,T_0)\in\PD$.  Then
  $n_+(T_1)=n_-(T_1)=m_+(T_1)=m_-(T_1)=0$.  In particular, $b(T_1) =
  \area(T_1)$.
\end{lemma}

\begin{proof}  
  We claim that $m_+(T_1)=m_-(T_1)=0$.  Otherwise $T_1$ shares an
  egressive edge of length greater than $1.8$ with some pseudo-dimer.
  However, the longest edge of $T_1$ is its shared edge with $T_0$,
  which has length less than $1.8$.  This gives the claim.

  Next we claim that $n_+(T_1)=0$.  Otherwise, $(T_1,T_0)\in \N$.
  Because $T_0$ is nonobtuse, Lemma~\rif{lemma:obtuse-b-arrow} 
  implies that there exists an obtuse
  triangle $T$ such that $T\Ra T_0$.  This contradicts
  Lemma~\rif{lemma:pd-obtuse}.

  Finally, we claim that $n_-(T_1)=0$.  Otherwise, $(T_+,T_1)\in \N$ for
  some $T_+$.  Because $T_1$ is nonobtuse, this implies that there exists
  an obtuse triangle $T$ such that $T\Ra T_1$.  This contradicts
  Lemma~\rif{lemma:pd-obtuse}.
\end{proof}

\begin{lemma}\libel{lemma:pd-b}  
  Let $(T_1,T_0)\in\PD$.  Then $b(T_0) > \acrit$.  That is, $T_0$ is
  not $b$-subcritical.
\end{lemma}

\begin{proof}
  We recall from Corollary~\rif{lemma:pd-n2} that $n_-(T_0)=0$.  From
  Lemma~\rif{lemma:pd-m2} we have $m_-(T_0)=0$.  Thus, $b(T_0) =
  \area(T_0) + \epsN n_+(T_0) + \epsM m_+(T_0)$.  We have
  (by Lemma~\rif{calc:pseudo-area})
\[
\area(T_0) = \area\{T_0,T_1\} - \area(T_1) 
> (2\acrit-\epsM) -\acrit = \acrit - \epsM.
\]

We first treat the case that $n_+(T_0)>0$ or $m_+(T_0)>0$.  Then we
have $n_+(T_0)+m_+(T_0)\ge 1$.  Thus,
\[
b(T_0) \ge \area(T_0) + \epsM (n_+(T_0)+m_+(T_0)) > (\acrit - \epsM)
+ \epsN > \acrit.
\]
This completes this case.

For the remainder of the proof, we assume that $n_+(T_0)=m_+(T_0)=0$.  In
particular, we have $b(T_0) = \area(T_0)$.

Let $T_0\Ra T_-$.  We have $(T_0,T_-)\in\M$ by the definition of $\M$,
unless the uniqueness property of $\M$ condition 4 fails.  That is,
there exists $T'_1\ne T_1$ such that $(T'_1,T_0)\in\PD$.  The edges of
$T_0$ shared with $T_1$ and $T'_1$ have length at least $1.72$ and
less than $1.8$.  The egressive edge of $T_0$ has length at least
$1.8$.  Thus,
\[
b(T_0) = \area (T_0) \ge \area(1.72,1.72,1.8) > \acrit.
\]
% calc is in table above.
This completes the proof.
\end{proof}

\section{Main inequality for dimers}

Recall that the previous section shows that there exists a unique
dimer pair (up to congruence): the ice-ray dimer.  In particular, if
$(T_1,T_0)$ is a dimer, then $\area(T_0)=\area(T_1)=\acrit$, and
$T_0\Ra T_1$, and $T_1\Ra T_0$, and $(T_1,T_0)$ is a dimer too.  Thus,
the relationship between $T_0$ and $T_1$ in a dimer pair is
symmetrical.  This section proves the following theorem.

\begin{theorem}\libel{thm:dimer} 
  Let $(T_1,T_0)$ be a dimer pair.  Then for $T\in\{T_0,T_1\}$, we have
  $m_+(T)=m_-(T) = n_+(T)=n_-(T) = 0$; and $b(T) = \area(T)$.
  Moreover, $\{T_0,T_1\}$ is a cluster, and the main inequality holds
  for $\{T_0,T_1\}$.
\end{theorem}

The proof will occupy the entire section.  Before treating dimers, we
treat the easy case of singletons.

\begin{lemma}\libel{lemma:singleton}  
  Assume that all the edges of a $P$-triangle $T$ have length less
  than $1.72$.  Then the cluster of $T$ is the singleton $\{T\}$ and
  $b(T) =\area(T)$.  Generally, if $\{T\}$ is a singleton
  cluster, then $b(T) > \acrit$ and the main inequality holds.
\end{lemma}

\begin{proof} 
  Assume that all the edges of $T$ have length less than $1.72$.  The
  constant $1.72$ is built into the definition of the constants
  $n_\pm,m_\pm$. This gives $n_+(T)=n_-(T) = m_+(T)=m_-(T)=0$.  Thus,
  $b(T)=\area(T)$.  We have $\area(T) > \acrit$ by
  Lemma~\rif{lemma:172}.  Because $T$ is not $b$-subcritical, there are
  no arrows $T\rab -$.

  We claim that there cannot exist an arrow $T_+\rab T$.  Otherwise,
  the longest edge of $T_+$ has length less than $1.72$ and again
  $T_+$ is not $b$-subcritical.  The claim follows and the cluster is
  a singleton.

  In general, if $\{T\}$ is any singleton cluster, there is no arrow
  $T\rab -$.  This implies that $T$ is not $b$-subcritical, so that
  $b(T) > \acrit$.  The result follows.
\end{proof}

\subsection{interaction of dimers and pseudo-dimers}

Next, we show the disjointness of dimers from pseudo-dimers.

\begin{lemma}\libel{lemma:pd-nra-dimer}
  Let $(T_1,T_0)\in DP$ and let $(T'_1,T'_0)\in\PD$.  Then for
  $T\in\{T_0,T_1\}$, we have $T'_0\nRa T$.
\end{lemma}

\begin{proof} 
  Assume $T'_0\Ra T$.  By Lemma~\rif{lemma:egress'}, we have
  $\area(T)>\acrit$.  But if $T\in\{T_0,T_1\}$, we have
  $\area(T)=\acrit$.  This gives the result.
\end{proof}

\begin{corollary}\libel{lemma:m2-dimer}  
  Let $(T_1,T_0)\in DP$ and let $(T'_1,T'_0)\in \PD$.  Then
  $m_-(T_0)=m_-(T_1)=0$.
\end{corollary}

\begin{proof}  
  If $m_-(T_-)>0$, then by the definition of the set $\M$, there
  exists $(T_1,T_0)\in \PD$, such that $T_0\Ra T_-$.  However,
  $T'_0\nRa T_-$, for $T_-\in\{T_0,T_1\}$ by
  Lemma~\rif{lemma:pd-nra-dimer}.  The result follows.
\end{proof}

\begin{lemma}\libel{lemma:m1-dimer}  
  Let $(T_1,T_0)\in DP$ and let $(T'_1,T'_0)\in\PD$.  Then
  $\{T_0,T_1\}\cap \{T'_0,T'_1\} = \emptyset$.  Moreover,
  $m_+(T_0) = m_+(T_1)=0$.
\end{lemma}

\begin{proof}
  The relationship between $T_0$ and $T_1$ in a dimer pair is
  symmetrical.  It is enough to show $T_1\not\in\{T_0',T_1'\}$.  We
  claim that $T_1\ne T'_0$.  Otherwise, $T'_0\Ra T_0$, which is
  contrary to Lemma~\rif{lemma:pd-nra-dimer}.

  We claim that $T'_1\ne T_1$.  Otherwise, if $T_1=T'_1$, then
  $T'_1=T_1\Ra T_0$ and $T'_1\Ra T'_0$, so that $T_0 = T'_0$, which
  have shown impossible.  This proves the disjointness result.

  If $m_+(T) >0$, then there exists $T_1''$ such that $(T_1'',T)\in
  \PD$.  This is impossible for $T\in\{T_0,T_1\}$ by the disjointness
  result.
\end{proof}

\subsection{interaction with obtuse triangles}

\begin{lemma}\libel{lemma:t2b-nonobtuse}
  If $T_1$ is obtuse, $T_0$ is nonobtuse, and if $T_1 \Ra T_0$, then
  $T_0$ is not $b$-subcritical.
\end{lemma}

\begin{proof} 
  If $T_1$ is obtuse, then its longest edge, which is shared with
  $T_0$, has length at least $\kappa\sqrt8$.

  Assume for a contradiction that $T_0$ is nonobtuse and
  $b$-subcritical.  By Lemma~\rif{lemma:right}, $\area(T_0) > \acrit$.
  Thus, we must have $n_-(T_0)>0$ or $m_-(T_0)>0$.  We consider two
  cases, according to which of these inequalities occurs.

  Assume in the first case that $n_-(T_0)>0$; that is, $(T',T_0)\in
  \N$.  Recall that this implies $m_-(T_0)=0$ by
  Lemma~\rif{lemma:n2m2}.  We have $T_1\Ra T_0$ and $(T_1,T_0)\not\in
  \N$, because $T_1$ is obtuse.  It follows that $n_-(T_0)\le 2$.  In
  fact, $n_-(T_0)$ is equal to the number of nonobtuse $T$ with
  longest edge of length at least $1.72$ such that $T\Ra T_0$.  If
  $n_-(T_0)=1$, we have
  \[
b(T_0) \ge \area(T_0) - \epsN 
> \area(\kappa\sqrt8,1.72,2\kappa)  - \epsN > \acrit.
\]
If $n_-(T_0)=2$, we have
  \[
b(T_0) \ge \area(T_0) - 2\epsN 
> \area(\kappa\sqrt8,1.72,1.72)  - 2\epsN > \acrit.
\]
  This completes the first case.
% calculation checked in calcs.ml 2016/9.

  Finally, we consider the case that $m_-(T_0)>0$ (and $n_-(T_0)=0$).
  We have $m_-(T_0)\le 2$, because $T_1\Ra T_0$, and $T_1$ is obtuse,
  and cannot be part of a pseudo-dimer.  The following area estimate
  gives the result.
\[
  b(T_0) \ge \area(T_0) - 2\epsM\ 
  > \area(\kappa\sqrt8,1.8,2\kappa) -
  2\epsM > \acrit.
\]
% calculation checked in calcs.ml 2016/9.
The constant $1.8$ comes from the egressive edge of a pseudo-dimer in
the definition of $\M$ and $m_-$.
\end{proof}

\begin{corollary}\libel{lemma:n2b} 
  If $n_-(T_-)>0$ with $T_-$ nonobtuse, then $b(T_-) > \acrit$.
\end{corollary}  

\begin{proof}  
  By the definition of $\N$ (nonobtuse target), There exists an obtuse
  triangle $T$ such that $T\Ra T_-$.  The result follows from the
  lemma.
\end{proof}

\begin{lemma}\libel{lemma:m2b}  
  Let $T$ be $b$-subcritical and nonobtuse.  Then $n_-(T)=0$.
  Moreover, assume that there exists $T'$ that is $b$-subcritical such
  that $T'\rab T$ or $T\rab T'$.  Then $\area(T)\le \acrit$.
\end{lemma}

\begin{proof}  
  By the contrapositive of the corollary, it follows that $n_-(T)=0$.

  Assume for a contradiction that $\area(T) > \acrit$.  We have
\begin{equation}\libel{eqn:2-ac}
\area(T) > \acrit \ge b(T) \ge \area(T) - \epsM m_-(T).
\end{equation}
Thus, $m_-(T) >0$.  By the definition of $\M$, there exists
$(T'_1,T'_0)\in\PD$ with $(T'_0,T)\in\M$ and $T'_0\Ra T$.  By
Corollary~\rif{lemma:egress'}, we have $\area(T) >
\acrit+\epsM$. Combined with Inequality ~\rif{eqn:2-ac}, this gives
$m_-(T)\ge 2$.  Repeating the argument, we have $(T''_1,T''_0)\in \PD$
with $(T''_0,T)\in \M$ and $T''_0\Ra T$.  The triangles $T'$, $T'_0$,
and $T''_0$ are distinct, because for example $b(T'_0) > \acrit \ge
b(T')$ (Lemma~\rif{lemma:pd-b}).  Because of the arrow $T'\rab T$ or $T\rab
T'$, the triangles $T$ and $T'$ belong to the same cluster.  The
longest edge of $T'$ has length at least $1.72$.  Then
\[
b(T) \ge \area(T) - m_-(T)\epsM > \area(1.8,1.8,1.72) - 3\epsM > \acrit.
\]
This  contradicts the assumption that $T$ is $b$-subcritical.
\end{proof}

\begin{lemma} \libel{lemma:rab-sequence} Assume that $T_1$ and $T_0$
  are both nonobtuse.  Then there does not exist a sequence $T_1\rab
  T_0 \rab T$, with $T_1\ne T$.
\end{lemma}

\begin{proof}  
  Assume for a contradiction that the sequence exists.  By the
  Lemma~\rif{lemma:m2b}, $n_-(T_0)=0$ and $\area(T_0)\le \acrit$.

  We claim that $m_+(T_0)=0$.  Otherwise, there exists
  $(T'_1,T_0)\in\PD$, and by Lemma~\rif{lemma:pd-b}, $T_0$ is not
  $b$-subcritical, contradicting the assumptions of the lemma.

  We claim that $m_-(T_0)=0$.  This follows by
  Lemma~\rif{lemma:m2-area} and the claim $\area(T_0)\le \acrit$.

We have that $n_+(T_0)=0$.  Otherwise, we reach the contradiction,
\[
\acrit \ge b(T_0) =\area(T_0) + \epsN n_+(T_0) > \ao + \epsN = \acrit.
\]
This shows that $\area(T_0) = b(T_0)$.

By Lemma~\rif{lemma:m2b}, we have $\area(T_1)\le \acrit$.
It follows that $(T_1,T_0)\in\PD$, and by Lemma~\rif{lemma:pd-b},
we reach a contradiction to the assumption that $T_0$ is $b$-subcritical. 
\end{proof}
  
We are finally in a position to prove Theorem~\rif{thm:dimer}.

\begin{proof}[Proof of dimer theorem~\rif{thm:dimer}.]
  Let $(T_1,T_0)$ be a dimer pair, and let $T\in \{T_0,T_1\}$.  We
  have proved that $m_+(T)=m_-(T)=0$ in Lemma~\rif{lemma:m1-dimer} and
  Corollary~\rif{lemma:m2-dimer}.

  We claim that $n_-(T_0)=n_-(T_1)=0$.  Otherwise, $T_1$ or $T_0$ has
  an edge of length at least $\kappa\sqrt8$, coming from an edge
  shared with an obtuse triangle, and the triangle $T_0$ or $T_1$ is
  not subcritical.

  We claim that $n_+(T_0)=n_+(T_1)=0$.  Otherwise, say
  $(T_1,T_0)\in\N$, and we have the contradiction $n_-(T_0)>0$.

  This shows that $b(T_i)=\area(T_i)$, for $i=0,1$.  Because $T_0$ and
  $T_1$ are both subcritical, they are also $b$-subcritical, and fall
  into the same cluster.  This is the full cluster, for otherwise, we
  would have say $T\rab T_0\rab T_1$, which is impossible by
  Lemma~\rif{lemma:rab-sequence}, for $T'$ nonobtuse.  And if $T'$ is
  obtuse, this contradicts Lemma~\rif{lemma:t2b-nonobtuse}).
  This gives the proof.
\end{proof}

\subsection{cluster structure}

% Fig 1. 1-obtuse.pdf page 13b.

We continue with our analysis of the clusters in a fixed saturated
packing of regular pentagons.  

\begin{lemma}\libel{lemma:notlong}  
  Let $T'$ and $T$ be $P$-triangles, such that $T'\Ra T$, where
  $T'$ is nonobtuse subcritical and $T$ is obtuse.  Then the edge
  of attachment is not the longest edge of $T$.
\end{lemma}

\begin{proof} We have seen that each edge of a nonobtuse subcritical
  triangle has length less than $\kappa\sqrt8$ and that the longest
  edge of an obtuse Delaunay triangle has length at least
  $\kappa\sqrt{8}$.  These are incompatible conditions on a shared
  edge.
\end{proof}

\begin{corollary} \libel{lemma:nm2}
  Let $(T_+,T_-)\in\N$, where $T_-$ is obtuse.  Then the edge shared
  between the triangles is not the longest edge of $T_-$.  In
  particular, $n_-(T_-)\le 2$.
\end{corollary}

\begin{proof} If $(T_+,T_-)\in \N$ (obtuse target), then $T_+\Ra T_-$.
  Also, $T_+$ and $T_-$ satisfy the assumptions of
  Lemma~\rif{lemma:notlong}.
\end{proof}

\begin{lemma}\libel{lemma:3m2} 
If $m_-(T) = 3$, then $b(T) > \acrit + \epsN$.
\end{lemma}

\begin{proof} By Lemma~\rif{lemma:n2m2}, $n_-(T)=0$.  There is a
  longest (egressive) edge of a pseudo-dimer along each edge of $T$,
  each of length at least $1.8$.  This gives
\[
b(T) \ge \area(T) - 3\epsM 
\ge \area(1.8,1.8,1.8) - 3\epsM > \acrit + \epsN.
\]
\end{proof}
%calculation checked in calcs.ml 2016/9

\begin{lemma}\libel{lemma:Nb}  
  If $(T_+,T_-)\in \N$, then $b(T_+)>\acrit$.
\end{lemma}

\begin{proof}  Otherwise, $n_+(T_+)\ge 1$, and
\[
\acrit \ge b(T_+) 
\ge \area(T_+) + \epsN (1 - n_-(T_+)) - \epsM m_-(T_+)
> \acrit - \epsN n_-(T_+) - \epsM m_-(T_+).
\]
So $n_-(T_+) > 0$ or $m_-(T_+)>0$. This gives two cases.  

Suppose that $n_-(T_+) >0$.  Recall that $T_+$ is nonobtuse.  Then
$(T,T_+)\in\N$ for some nonobtuse $T$ whose shared edge with $T_+$ has
length at least $1.72$.  Thus, by the definition of $\N$ (nonobtuse
target), there exists $T'$ obtuse such that $T'\Ra T_+$.  The shared
edge has length at least $\kappa\sqrt8$.  Also, $n_-(T_+)\le 2$
(because $(T',T_+)\not\in \N$).  By Lemma~\rif{lemma:n2m2}, we have
$m_-(T_+)=0$.  Then
\[
b(T_+) \ge \area(T_+) + \epsN (1 - 2) 
\ge \area (2\kappa,1.72,\kappa\sqrt8) - \epsN > \acrit.
\]
This completes this case.
%calculation checked in calcs.ml 2016/9

Finally, suppose that $m_-(T_+)>0$ and $n_-(T_+)=0$.  There exists a
pseudo-dimer $(T'_1,T'_0)$ such that $T'_0\Ra T_+$.  We have
by Lemma~\rif{lemma:m2-area},
\[
b(T_+) \ge \area(T_+)+\epsN - \epsM m_-(T_+) 
> (\acrit + \epsM) + \epsN -\epsM  3 \ge \acrit.
\]
\end{proof}

\begin{lemma}\libel{lemma:no-ao} 
  There is no arrow $T_1 \rab T_0$ with $T_1$ nonobtuse and $T_0$
  obtuse.
\end{lemma}

\begin{proof}  
  Assume for a contradiction that such a pair $(T_1,T_0)$ exists.  By
  Lemma~\rif{lemma:Nb}, $(T_1,T_0)\not\in\N$.  By the definition of
  $\N$ (obtuse target), the longest edge of $T_1$ has length less than
  $1.72$.  By Lemma~\rif{lemma:singleton}, $T_1$ forms a singleton
  cluster.  This contradicts $T_1\rab T_0$.
\end{proof}

\section{Obtuse Clusters}\libel{sec:obtuse}

In this section we prove the strict main inequality for clusters that
contain an obtuse triangle.  This will involve several cases, but in
every case the strict main inequality will be found to hold by a large
margin.  This allows us to use rather crude approximations of
area in this section.  In particular, we are able to disregard most
constraints on the shapes of Delaunay triangles imposed by the
pentagons.  Instead, we use generic features of the triangles such as
the fact that the circumradius is at most two and the edge lengths of
the triangle are at least $2\kappa$.

\begin{remark}\libel{rem:delaunay}
  Recall that the Delaunay property implies that two adjacent Delaunay
  triangles $T_1$ and $T_2$ have the property that $\alpha_1 +
  \alpha_2\le \pi$, where $\alpha_i$ is the angle of $T_i$ that is not
  at the shared edge of $T_1$ and $T_2$.  In particular, {\it two
    obtuse Delaunay triangles cannot be joined along an edge that is
    the longest on both triangles.}  The extreme case
  $\alpha_1+\alpha_2=\pi$ corresponds to the degenerate situation
  where $T_1$ and $T_2$ form a cocircular quadrilateral. When
  cocircular, either diagonal of the quadrilateral gives an acceptable
  Delaunay triangulation.
\end{remark}

% page 7 pdf obtuse.

We write $\areta(d_1,d_2,h)$ for the area of a triangle with two edges
$d_1,d_2$ and circumradius $h$.  In general, two noncongruent
triangles have data $d_1,d_2,h$.  We choose $\areta(d_1,d_2,h)$ to
give the area of that triangle such that its third edge $d_3$ is as
long as possible.

The following lemma shows that under quite general
conditions, we are justified in our decision to choose $d_3$ as long
as possible in the definition of the function $\areta(d_1,d_2,h)$.  It
is justified in the sense that the other choice does not usually give
a Delaunay triangle of a pentagon packing, according to the following
simple test.

\begin{lemma}\libel{lemma:areta}  
  Let $d_1$, $d_2$ and $\eta$ be positive real numbers.  Assume that
  $T$ and $T'$ are triangles with edge lengths $d_1,d_2,d_3$ and
  $d_1,d_2,d_3'$, and with the same circumradius $\eta$. Assume
  $2\kappa\le d_1\le d_2$.  Set $\theta =
  \arc(\eta,\eta,d_1)+\arc(\eta,\eta,2\kappa)$.  If $2\kappa \le d_3'
  < d_3$, then $\theta < \pi$ and $2\eta\sin(\theta/2) \le d_2$.
\end{lemma}

As a corollary, 
in contraposition, if $\theta\ge\pi$ or if $d_2 < 2\eta\sin(\theta/2)$,
then the triangle $T'\ne T$, with $d_3' < d_3$,
cannot satisfy the constraint $2\kappa\le d_3'$ of a  Delaunay triangle.

\begin{proof} 
  See Figure~\rif{fig:areta}.  Let $\p$, $\q$, and $\r$ (resp. $\p$,
  $\q$, and $\r'$) be the vertices of $T$ (resp. $T'$) on a common
  circle, with
\[
\norm{\p}{\q} =d_1, 
\quad\norm{\p}{\r}=\norm{\p}{\r'}=d_2, \quad\text{and } 
\norm{\q}{\r'} = d_3' \le d_3=\norm{\q}{\r}.
\]
The angle on the circumcircle from $\p$ to $\q$ is
$\arc(\eta,\eta,d_1)$, and $\theta$ is the angle on the circumcircle
from $\p$ to the first point $\s$ beyond $\q$ such that
$\norm{\s}{\q}=2\kappa$.  If $\theta\ge\pi$, a point $\r'\ne \r$
satisfying the constraints does not exist.  Assume $\theta < \pi$.  As
the figure indicates, $\norm{\p}{\r'}$ is minimized (as a function of
$d_2$) when $\r'=\s$, and $d_3' = 2\kappa$, the lower constraint.
Then $d_2 = \norm{\p}{\r'}\ge \norm{\p}{\s}= 2\eta\sin(\theta/2)$.
\end{proof}

\tikzfig{areta}{
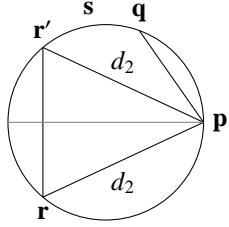
There can be two positions $\r,\r'$ 
on the circumcircle for
the third vertex of the triangle.  
Here, $d_1 = \norm{\p}{\q}$.}{
[scale=1.3]
\draw (0,0) circle (1cm);
\draw (1,0) node[anchor=west] {$\p$} --  (70:1cm) node[anchor=south] {$\q$};
\draw (1,0) -- node[above=1pt] {$d_2$} (130:1cm) node[anchor=south] {$\r'$};
\draw (1,0) -- node[below=1pt] {$d_2$} (-130:1cm) node[anchor=north] {$\r$};
\draw (130:1cm) -- (-130:1cm);
\draw[thin,gray] (1,0) -- (-1,0);
\draw (100:1cm) node[above=1pt] {$\s$};
}

\begin{lemma}\libel{lemma:t2b}
  If $T_1$ is obtuse, and $T_1 \Ra T_0$, then $T_0$ is not
  $b$-subcritical.
\end{lemma}

\begin{proof}  
  If $T_0$ is nonobtuse, then this is Lemma~\rif{lemma:t2b-nonobtuse}.

  Assume that $T_0$ is obtuse.  By basic properties of Delaunay
  triangles (Remark~\rif{rem:delaunay}), Delaunay triangles never join
  along an edge that is the longest on both triangles.  Thus, $T_1$
  attaches to $T_0$ along an edge adjacent to the obtuse angle of
  $T_0$.  The shared edge has length at least $\kappa\sqrt8$.  To
  bound the area of $T_0$, we deform $T_0$ decreasing its area and
  increasing its longest edge and its circumradius, until we obtain a
  triangle of circumradius $\eta=2$, and shortest edges $2\kappa$ and
  $\kappa\sqrt{8}$.  Then a numerical calculation (using
  Corollary~\rif{lemma:nm2} and Lemma~\rif{lemma:obtuse-b}) gives
\[
b(T_0) \ge \area(T_0) - \epsN n_-(T_0) \ge 
\areta(2\kappa,\kappa\sqrt{8},2) - 2\epsN > \acrit.
\] 
% checked 2016/2/18 in Mathematica. It holds by a margin 0.1856....
% 2016/9 checked in calcs.ml.
The use of the function $\areta$ is justified by
Lemma~\rif{lemma:areta} and the numerical estimate
\[
d_2 = \kappa\sqrt8 
<  4 \sin(\arc(2,2,2\kappa)) = 2\eta\sin(\theta/2).
\] 
% checked 2016/2/18, 2016/9 in Mathematica.
% part of areta function in calcs.ml, lemma_t2b.
\end{proof}

In future uses of the function $\areta$, we always check that the
conditions of Lemma~\rif{lemma:areta} justify the use of the function.
We do not show these calculations.

\begin{lemma}\libel{lemma:sequence}  There does not exist a three
  term sequence $- \rab -\rab -$ where the three triangles are
  distinct.
\end{lemma}

\begin{proof}  
  Assume for a contradiction, that such a sequence exists.  By
  Lemma~\rif{lemma:t2b}, there does not exist a sequence $- \rab -
  \rab -$, where the first triangle is obtuse.  Thus, we may assume
  that the first triangle is nonobtuse.  By Lemma~\rif{lemma:no-ao},
  there does not exist $T_1\rab T_0$, where $T_1$ is nonobtuse and
  $T_0$ is obtuse.  Thus, we may assume that every triangle in the
  sequence is nonobtuse.  This is impossible by
  Lemma~\rif{lemma:rab-sequence}.
\end{proof}

If $\C$ is a cluster that is not a singleton, then there is
some arrow $T_1\rab T_0$.  We have the following structure
theorem for clusters.

\begin{theorem}\libel{lemma:c-weak}
  Let $\C$ be a cluster, and let $T_1\rab T_0$ be an
  arrow between triangles in $\C$.
  Then 
  \begin{equation}\libel{eqn:c-weak}
  \C = \{T_0\}\cup \{T \mid T\rab T_0\}.
  \end{equation}
\end{theorem}

\begin{proof}  
  We use Lemma~\rif{lemma:sequence}.  Assume that $T\rab T_0$.  There
  is no arrow $T'\rab T$, with $T'\ne T_0$, because that would also
  produce a sequence $T' \rab T \rab T_0$ of three distinct triangles.
  There is a unique arrow out of $T$.  Thus, we have accounted for all
  of the arrows in and out of $T$.

  There is no arrow $T_0\rab T''$, with $T''\ne T_1$, because that
  would produce a sequence $T_1\rab T_0\rab T''$ of  three 
  distinct triangles.  We have accounted for all the arrows in and
  out of $T_0$.  Thus, the full cluster has been identified.
\end{proof}

\begin{corollary}\libel{lemma:card4}
  Every cluster is finite of cardinality at most $4$.
\end{corollary}

\begin{proof} At most three triangles attach to $T_0$.
\end{proof}

The following theorem is the main result of this
section.

\begin{theorem}\libel{lemma:obtuse}  
  Let $\C$ be any cluster that contains an obtuse triangle.  Then 
   the strict main inequality (\rif{eqn:strict-main})
  holds for $\C$.
\end{theorem}

We prepare for the proof of the theorem with some lemmas.

\begin{lemma} \libel{lemma:obtuse-source} 
  Let $T_1\rab T_0$ be an
  arrow between two triangles in a cluster that contains an obtuse
  triangle.  Then $T_1$ is obtuse.  Moreover, for every obtuse
  triangle $T$ in the cluster, $b(T) = \area(T) -\epsN n_-(T)$.
\end{lemma}

\begin{proof} 
  Assume for a contradiction that $T_1$ is nonobtuse.  By
  Lemma~\rif{lemma:no-ao} applied to the arrow $T_1\rab T_0$, the
  triangle $T_0$ is nonobtuse.  By assumption and the structure
  theorem for clusters, there exists $T'\rab T_0$, where $T'$ is
  obtuse.  The singleton lemma (Lemma~\rif{lemma:singleton}) implies
  that the longest edge of $T_1$, which is shared with $T_0$ has
  length at least $1.72$.  By the definition of $\N$ (nonobtuse
  target), we have $(T_1,T_0)\in\N$.  By Lemma~\rif{lemma:Nb}, we have
  $b(T_1)>\acrit$, and $T_1$ is not $b$-subcritical.  Thus, we obtain
  a contradiction to $T_1\rab T_0$.

 Thus, $T_1$ is obtuse.
  Moreover, if $T$ is obtuse, then $b(T) = \area(T) -\epsN n_-(T)$ by
  Lemma~\rif{lemma:obtuse-b}.
\end{proof}

Let $\C$ be a cluster.  By {\it external edge} of the cluster, we mean an
edge of a triangle in the cluster that is not shared with another
triangle in the cluster.  Let $\tn(\C)$ be the total number of
external edges of length at least $1.72$ of the cluster $\C$.  We have
\[
\tn(\C) =\sum_{T\in \C} \tn(T),
\]
where $\tn(T)$ is the number of edges of $T$ of length at least $1.72$
that are external edges of its cluster.
Define $\tb(T) := \area(T)-\epsN\tn(T)$.  We use $\tb(T)$ to give an
easily computed lower bound given in the following lemma.

\begin{lemma} \libel{lemma:tb}
  Let $\C$ be a cluster of cardinality at least $2$ 
  that contains an obtuse triangle.
Then 
\[
\sum_{T\in\C} b(T)\ge \sum_{T\in\C} \tb(T).
\]
\end{lemma}

\begin{proof}
  Let $\N_{ext}\subseteq \N$ be the subset consisting of pairs
  $(T_+,T_-)$ such that at least one of $T_+$ and $T_-$ is not in
  $\C$.  Define $\M_{ext}\subseteq \M$ similarly.  If $(T_+,T_-)\in
  \N\setminus \N_{ext}$, then $T_+,T_-\in\C$ and the pair $(T_+,T_-)$
  contributes $+\epsN$ to the value of $b(T_+)$ and $-\epsN$ to the
  value of $b(T_-)$.  These contributions cancel.  Similar comments
  apply to $(T_+,T_-)\in\M\setminus \M_{ext}$.  Thus,
\[
\sum_{T\in \C} b(T) = \sum_{T\in \C} b_{ext}(T),
\]
where $b_{ext}(T)$ is defined as $b(T)$, but using $\N_{ext}$ and
$\M_{ext}$ instead of $\N$ and $\M$. It is enough to show that
\[
b_{ext}(T) \ge \tb(T),
\]
for all $T\in\C$.

Let $T\in\C$.  To prove the inequality for $T$, it is enough to show
that $\tn(T)\ge n_{ext,-}(T)+m_{ext,-}$.  In fact, this inequality gives
\begin{align*}
b_{ext}(T) &\ge \area(T) - \epsN n_{ext,-}(T) - \epsM m_{ext,-}(T) \\
&\ge \area(T) - \epsN (n_{ext,-}(T)+m_{ext,-}(T)) \\
&\ge \area(T) - \epsN \tn(T)\\
&= \tb(T).
\end{align*}
Note that $n_{ext,-}(T)+m_{ext,-}(T)$ counts pairs
$(T_+,T)\in\N_{ext}\sqcup M_{ext}$ and every such shared edge is
external and has length at least $1.72$.  Thus every pair counted in
$n_{ext,-}(T)+m_{ext,-}(T)$ is also counted in $\tn(T)$.  This gives
the inequality and completes the proof of the lemma.
\end{proof}

\begin{proof}[proof of Theorem~\rif{lemma:obtuse}.]
  The proof involves several relatively simple cases.  
  We recall that each Delaunay triangle has edge lengths at
  least $2\kappa$ and circumradius at most $2$.

  If the cluster is a singleton $\{T\}$, where $T$ is obtuse, then the
  singleton lemma (Lemma~\rif{lemma:singleton}) gives the result.  We
  now assume that $\C$ is not a singleton.  By
  Theorem~\rif{lemma:c-weak}, the cluster $\C$ has the form of
  Equation~\rif{eqn:c-weak} for some triangle $T_0$.  Each $T$ such
  that $T\rab T_0$ is obtuse by Lemma~\rif{lemma:obtuse}.

  We break the proof into six cases depending on whether $T_0$ is
  nonobtuse, and depending on $\card(\C)\in \{2,3,4\}$.  In each case
  we prove inequality
\begin{equation}\libel{eqn:n}
 \area(\C) > \acrit \card(\C) + \epsN \tn(\C).
\end{equation}
By Lemma~\rif{lemma:tb},    this implies that
\[
\sum_{T\in\C} b(T) \ge \sum_{T\in\C} \tb(T) 
\ge \area(\C) -  \epsN\tn(\C) 
> \acrit \card(\C),
\]
which is the strong main inequality.

{\it Case 1. The triangle $T_0$ is a nonobtuse triangle, and
  $\C=\{T_0,T_1\}$.}  The triangle $T_0$ has a vertex $\v$ that is
not shared with $T_1$.  By the Delaunay property, $\v$ lies outside
the circumcircle of $T_1$.  The triangles $T_0$ and $T_1$ form a
quadrilateral $Q$ whose diagonal is the shared edge of $T_0$ and
$T_1$.  We deform the quadrilateral $Q$ to decrease its area while
maintaining the following constraints:
\begin{enumerate}
\item The vertex $\v$ lies on or outside the circumcircle of
  $T_1$. The circumradius of $T_1$ is at most $2$.
\item The edge length of the $i$th edge of $Q$ is at least
  $d_i\in\{2\kappa,1.72\}$,
where  the number of $d_i$ that equal $1.72$ is $\tn(\C)$; and
\item $T_1$ is not acute; $T_0$ is not obtuse.
\end{enumerate}
We drop all other constraints as we deform. (In particular, we do not
enforce the nonoverlapping of pentagons in the $P$-triangles.)  We
continue to deform $Q$ until one of the following two subcases hold:
\begin{enumerate}
\item $Q$ is cocircular; or
\item For all $i=1,2,3,4$, the $i$th edge of $Q$ has reached its lower
  bound $d_i$.
\end{enumerate}

In the first subcase (cocircularity), we drop the third constraint
(acute/obtuse) and continue area decreasing deformations for $Q$
under the constraint of a fixed circumcircle.  We note that the area
of a cocircular quadrilateral $Q$ depends only on the lengths of the
edges and not on their cyclic order on $Q$.  We may thus rearrange the
edge order as we deform.  For a given circumcircle, the area is
minimized when three of the edges attain their lower bound $d_i$.  By
suitable reordering of the edges, we may assume that $Q$ is an
isosceles trapezoid and that the fourth (free) edge is parallel to and
at least as long as its opposite edge on $Q$.  For such $Q$, the area as a
function of the circumradius is concave, so that the minimum occurs
when the circumradius is as small (that is, all edges attain the
minimum $d_i$) or as large (that is, $\eta(Q)=2$) as possible.  When
$\eta(Q)=2$, we relax the edge lengths constraints further to allow
three edges to have length $2\kappa$.  Explicit numerical calculations
in these two extremal configurations show that the inequality
(\rif{eqn:n}) is satisfied for each $\tn\in \{1,2,3,4\}$.

In the second subcase (every edge attains its minimal length $d_i$),
the four edge lengths are fixed.  We drop the constraint that $T_0$ is
not obtuse. The area of $Q$ is a concave function of the length of the
diagonal.  We thus minimize the area of $Q$ when the diagonal is as
small as possible (that is, $T_1$ is a right triangle -- when this
satisfies the constraint that $\v$ is outside the circumcircle of
$T_1$) or as large as possible (that is, $Q$ is cocircular).  The
cocircular case has already been considered.  Explicit numerical
calculations of $Q$ when $T_1$ is right gives the inequality
(\rif{eqn:n}) in each case.
% checked 2016/2/18 in Mathematica.

{\it Case 2. The triangle $T_0$ is a nonobtuse triangle, and
  $\C=\{T_0,T_1,T_1'\}$.}  
The long edges of the obtuse triangles
$T_1$ and $T_1'$ have length at least $\kappa\sqrt{8}$.

We consider a subcase where $\eta(T_1)\le 1.7$ and $\eta(T_1') \le
1.7$.  Then calculations based on the monotonicity of the area
functions give
\[
\area(T_0) \ge \area(d,\kappa\sqrt{8},\kappa\sqrt{8}) > 
   \begin{cases}1.73,&\text{if } d= 2\kappa\\ 
     1.73+\epsN, & \text{if } d = 1.72 \end{cases}.
\]
The areas of $T=T_1,T_1'$ are at least
\begin{equation}\libel{eqn:173}
\area(T) \ge \areta(2\kappa,d,1.7) >
   \begin{cases}1.08,&\text{if } d = 2\kappa\\
     1.08+2\epsN, & \text{if } d = 1.72 \end{cases}.
\end{equation}
These bounds give inequality (\rif{eqn:n}):
\[
\area(T_0) +\area\{T_1,T_1'\}  > 1.73 + 2(1.08) + \epsN \tn > 
3 \acrit +  \epsN \tn.
\]
% checked 2016/2/18 in Mathematica.

By symmetry, we may now assume that $\eta(T_1) \ge 1.7$.  
The areas of $T_1$ and $T_1'$ are at least
\begin{equation}\libel{eqn:968}
\areta(2\kappa,2\kappa,2) > 0.968.
\end{equation}
By the Delaunay condition, because $T_1\Ra T_0$, $T_1$ is obtuse, and
$T_0$ is nonobtuse, this forces $\eta(T_0)\ge 1.7$.  We minimize the
area of $T_0$ subject to the constraints that its circumradius is at
least $1.7$, that it is nonobtuse, and its edge lengths are at least
$\kappa\sqrt{8}$, $\kappa\sqrt{8}$, and $2\kappa$.  If two edges are
$2\kappa$, $\kappa\sqrt{8}$ (or even if two edges are $\kappa\sqrt8$,
$\kappa\sqrt8$), then $T_0$ is obtuse by the circumradius constraint.
The binding constraints for the minimizaton become $\eta(T_0)=1.7$,
$2\kappa$ edge length, and a right triangle.  Such a triangle has area
at least
\[
2\kappa\sqrt{\eta^2 - \kappa^2} \ge 2.41.
\]  
There are five external edges, and we have $\tn(\C)\le 5$.
This completes this case:
\[
\area(T_1) + \area(T_1') + \area(T_0) 
>
2(0.968) + 2.41 > 3\acrit + 5 \epsN \ge 3\acrit + \epsN \tn.
\] 
%checked 2016/2/18 in Mathematica.

{\it Case 3. The triangle $T_0$ is a nonobtuse triangle, and
  $C'=\{T_0,T_1,T_1',T_1''\}$.}

This case is almost identical to case 2.  We use the same bounds
(Equations (\rif{eqn:173}) and (\rif{eqn:968})) on $\area(T)$ as
before, for $T = T_1, T_1', T_1''$.  We can improve the bound on the
area of $T_0$:
\[
\area(T_0) 
\ge \area(\kappa\sqrt{8},\kappa\sqrt{8},\kappa\sqrt{8}) > 2.2668.
\]
Moreover, in the subcase where $\eta(T_0)\ge 1.7$, we have
(even after dropping the nonobtuseness constraint):
\[
\area(T_0) \ge \areta(\kappa\sqrt{8},\kappa\sqrt{8},1.7) > 2.6.
\]
In this case, $\tn\le 6$.  Proceeding as before, we get
\[
\area\{T_0,T_1,T_1',T_1''\} > 
\begin{cases}
2.2668 + 3 (1.08) \\
2.6 + 3(0.968)
\end{cases}
> 4\acrit +  \epsN \tn.
\] 
%checked 2016/2/18 in Mathematica.

This completes the proof for cases involving a nonobtuse triangle
$T_0$.  In the remaining cases, we assume that $T_0$ is obtuse.  
In the remaining cases, every triangle in $\C$ is obtuse.

{\it Case 4. The triangle $T_0$ is an obtuse triangle, and
  $\C=\{T_0,T_1\}$.}

In this case, $\tn\le 4$.  It will not be necessary to create subcases
according to whether short edges are at least $2\kappa$ or $1.72$.  We
will show that we can relax the lower bound on the short edges to
$2\kappa$ and still obtain the bound (\rif{eqn:n}).

We minimize area by flattening $T_0$ by stretching its long edge until
$\eta(T_0)=2$.  We further decrease area, keeping the circumradius
fixed, by contracting the shorter edge not shared with $T_1$, until
that edge has length $2\kappa$.

Next continue to minimize area by contracting an edge of $T_1$,
keeping its circumradius fixed, until an edge has length $2\kappa$.
Then, allowing the circumradius of $T_1$ to increase, we continue until
both shorter edges have length $2\kappa$ or until the circumradius
reaches $2$.

First assume that both shorter edges of $T_1$ have length $2\kappa$.
We have reduced to a one-parameter family of quadrilaterals.  We can
choose the parameter to be the length $x$ of the diagonal, the common
edge of $T_1$ and $T_0$.  The parameter $x$ ranges between
$\kappa\sqrt{8}$ and $x_{\max}\approx 2.9594$, determined by the
condition $\eta(2\kappa,2\kappa,x_{\max}) = 2$.  We check numerically
that
\[
\area(T_1) + \area(T_0) \ge \area(2\kappa,2\kappa,x) +
\areta(2\kappa,x,2) > 2\acrit + 4 \epsN \ge 2\acrit + \epsN \tn.
\] 
% checked 2016/2/18 in Mathematica.
% 2016/9 in calcs.ml.

Next, assume the circumradius of $\eta(T_1)$ reaches $2$, then we
have a cocircular quadrilateral that can be treated as in Case 1.  In
particular, the minimizing cocircular quadrilateral has three edges of
length $2\kappa$ and circumradius $2$, which has area
\[
\area\{T_0,T_1\}
>  2\acrit+4\epsN \ge 2\acrit+\epsN\tn.
\]
This completes the argument in this case.

{\it Case 5. The triangle $T_0$ is an obtuse triangle, and
  $\C=\{T_0,T_1,T_1'\}$.}

By Remark~\rif{rem:delaunay}, there is no arrow $T \rab T_0$ in $\C$
such that the shared edge is the long edge of $T_0$.  In particular,
there cannot exist (Case 6) with $\C=\{T_0,T_1,T_1',T_1''\}$ with
every triangle obtuse.  Thus, Case 5 is the last case to be
considered.

We have $\tn\le 5$.
The area of $T_0$ is at least
$\areta(\kappa\sqrt{8},\kappa\sqrt{8},2) > 2.45$.  
Using our earlier estimates (\rif{eqn:968})
for $\area(T)$, for $T=T_1,T_1'$, we
have
\[
\area\{T_1,T_1',T_0\} > 2 (0.968) + 2.45 > 3\acrit + \epsN\tn.
\] 
% checked 2016/2/18 in Mathematica.
% checked 2016/9 in calcs.ml.
(As mentioned earlier, each use of the function $\areta$ is justified
by a calculation based on Lemma~\rif{lemma:areta}.)
This completes the proof of the theorem.
\end{proof}

\section{Nonobtuse Clusters}\libel{sec:nonobtuse}

In this section we prove the main inequality for clusters in which
every triangle is nonobtuse.  By Corollary~\rif{lemma:n2b} and
Lemma~\rif{lemma:Nb}, if $T$ is $b$-subcritical and nonobtuse, then
$n_+(T)=n_-(T)=0$.

\begin{lemma}\libel{lemma:m2} 
  Let $T\rab T_0$. Assume that $T$ and $T_0$ are nonobtuse.  Then
  $m_-(T)=0$.
\end{lemma}

\begin{proof}  
  Assume for a contradiction that $m_-(T)>0$.  We have just observed
  that $n_+(T)=n_-(T)=0$.  By Lemma~\rif{lemma:m1m2}, we have
  $m_+(T)=0$.  By Lemma~\rif{lemma:3m2}, we have $m_-(T)\le 2$.

  We have $m_-(T)=2$.  Otherwise, if $m_-(T)=1$, we have a
  contradiction (by Corollary~\rif{lemma:m2-area}):
\[
\acrit \ge b(T) \ge \area(T) - \epsM 
> (\acrit + \epsM) - \epsM = \acrit.
\]

Because $m_-(T)=2$, there exist two pseudo-dimers $(T'_1,T'_0)$ and
$(T''_1,T''_0)$ such that $T'_0\Ra T$ and $T''_0\Ra T$.  The shared
edges $e'$ and $e''$ have length at least $1.8$.  Moreover, the angles
are large along $(T_0',e')$ and $(T_0'',e'')$, but not large along
$(T,e')$ and $(T,e'')$.  (See Definition~\rif{def:large} and
Lemma~\rif{lemma:large}.)  Lemma~\rif{lemma:2small} (below) and the
estimate
\[
\acrit \ge b(T) \ge \area(T) - 2\epsM 
> (\acrit + 2\epsM) - 2\epsM = \acrit
\]
complete the proof.
\end{proof}

\begin{lemma} \libel{lemma:2small}
  Let $T$ be a nonobtuse $P$-triangle.  Suppose that two of its edges
  $e'$ and $e''$ have length at least $1.8$ and that the angles are
  not large along $(T,e')$ and $(T,e'')$.  Then $\area(T) > \acrit +
  2\epsM$.
\end{lemma}

\begin{proof}  
  If the third edge has length at least $1.63$, then the result easily
  follows:
\[
\area(T) \ge \area(1.8,1.8,1.63) > \acrit + 2\epsM.
\]
We remark that $1.63$ is close to the minimum edge length
$2\kappa\approx 1.618$.  This leaves hardly any flexibility in the
relative position of the two pentagons along this edge.

We may assume without loss of generality that the third edge has
length in the range $[2\kappa,1.63]$.  Let the pentagons at the
vertices of $T$ be $A$, $B$, and $C$, with $d_{AB}\le 1.63$.  For the
moment, we disregard the pentagon $C$ and parallel translate $B$,
decreasing $d_{AB}$ until $A$ and $B$ come into contact.  We assume
without loss of generality that $B$ points into $A$ at $\v_B$.  Draw
the configuration as in Figure~\rif{fig:not-large} with a vertical
receptor edge $e$ on $A$.  There are two cases, depending on whether
$\v_B$ lies above or below the midpoint of the edge $e$.  (The pointer
cannot be at the midpoint of $e$ by Lemma~\rif{lemma:mid-172}, and
$1.72 > 1.63 \ge d_{AB}$.)  Let $\gamma\in[0,2\pi/5)$ be the angle
%(]
formed by edges of $A$ and $B$ at the pointer $\v_B$ as in
Figure~\rif{fig:not-large}.

We have the constraint
\[
\cos(\pi/5 - \gamma) =\sin(\gamma+3\pi/10) \le 1.63 - \kappa.
\]
This constraint expresses the fact the distance from $\c_B$ to the
edge $e$ of pentagon $A$ can be at most $1.63-\kappa$.  The constraint
implies that
\[
\gamma \ge \pi/5 +
\arccos (1.63-\kappa) > 2\pi/5 - 0.021\quad \text{ or }\quad
\gamma\le \pi/5 - \arccos(1.63-\kappa) < 0.021,
\]
according to whether $\v_B$ is below or above the midpoint of $e$.

Now we return to the original $P$-triangle with pentagons $A$, $B$,
$C$ in their original position.  Because $\gamma$ was obtained after a
parallel translation of $B$, it equals the incidence angle of lines
through edges of the original $A$ and $B$.  Changing the choice of
edges on $A$ and $B$, we find an incidence angle  $\gamma' \in
[4\pi/5-0.021,4\pi/5+0.021]$.  See Figure~\rif{fig:not-large}.

We form a triangle with angles $\alpha$, $\beta$, and $\gamma'$ by
extending edges of $A$, $B$, and $C$.  The assumption that angles are
not large along $(T,e')$ and $(T,e'')$ gives $\alpha\ge \pi/5$ and
$\beta\ge \pi/5$.  The angle sum of the triangle gives a contradiction
\[
\pi=\alpha+\beta+\gamma' 
\ge  \pi/5 + \pi/5 + (4\pi/5 - 0.021) > \pi.
\]
\end{proof}

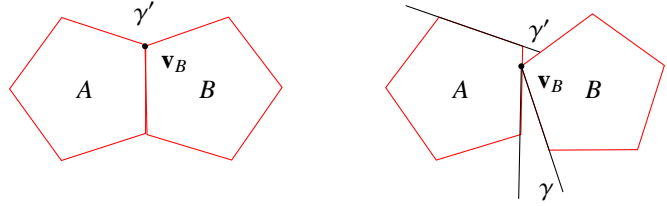
\tikzfig{not-large}{The angle between the nearly vertical edges of
  $A$ and $B$ is $\gamma$.  On the left, $\gamma= 0.02$ and $
  2\kappa<d_{AB}\approx 1.6296 < 1.63$.  There is almost no play in
  the configuration.  The  figure on the right does not satisfy the
  constraints of the proof of Lemma \ref{lemma:2small}, but
  illustrates the notation.}
{
\begin{scope}[scale=1]
\pen{0}{0}{35.87};
\pen{1.623}{0}{-34.72+180};
\draw (0,0) node {$A$};
\draw (1.623,0) node {$B$};
\draw (0.8,1) node {$\gamma'$};
\draw ($(1.623,0)+(-34.72+180:1)$) node[anchor=north west] {\ $\v_B$};
\smalldot{$(1.623,0)+(-34.72+180:1)$};
\end{scope}
\begin{scope}[scale=1,xshift=5cm]
\pen{0}{0}{34.86};
\pen{1.756}{0}{-17.68+180};
\draw (0,0) node {$A$};
\draw (1.756,0) node {$B$};
\draw ($(1.756,0) + (-17.68+180:1)$) --  ++ 
($1.5*(-17.68+180+72:1) - 1.5*(-17.68+180:1)$) node[anchor=east] {$\gamma$};
\draw ($(1.756,0) + (-17.68+180:1)$) --  ++ 
($1.5*(34.86-72:1) - 1.5*(34.86:1)$) ;
\coordinate (A) at (34.86:1);
\coordinate (Aa) at (34.86+72:1);
\draw ($1.21*(A) - 0.21*(Aa)$) node[anchor=south] {$\gamma'$} 
-- ($-0.4*(A) + 1.4*(Aa)$);
\draw ($(1.756,0)+(-17.68+180:1)$) node[anchor=north west] {\ $\v_B$};
\smalldot{$(1.756,0)+(-17.68+180:1)$};
\end{scope}
}
% deg23 (thetax (m 0.02) (m 0.02));;
% deg23 (thetax (m 0.3) (m 0.3));;

Next, we turn our attention to the target $T_0$ of an arrow $T\rab T_0$,
where both $T$ and $T_0$ are nonobtuse.  By
Lemma~\rif{lemma:singleton}, the shared edge has length at least
$1.72$.

\begin{lemma}\libel{lemma:n2-0} 
  Let $T_0$ be a triangle in a cluster $\C$ containing only nonobtuse
  triangles.  Then $n_-(T_0)=0$.
\end{lemma}

\begin{proof} 
  If $n_-(T_0)>0$, then by Corollary~\rif{lemma:obtuse-b-arrow} there
  exists an obtuse triangle $T'$ such that $T'\rab T_0$ and the
  cluster $\C$ contains an obtuse triangle.
\end{proof}

\begin{lemma}\libel{lemma:T1-all} 
  Let $T\rab T_0$. Assume that $T$ and $T_0$ belong to a cluster $\C$
 containing only  nonobtuse triangles.  Then
  $m_+(T)=0$.  Moreover, $n_+(T)=n_-(T)=m_+(T)=m_-(T)=0$ and
  $b(T)=\area(T)\le \acrit$.
\end{lemma}

\begin{proof}
  If $m_+(T)>0$, then there exists a pseudo-dimer $(T'_1,T'_0)$ such
  that $T=T'_0$.  Then Lemma~\rif{lemma:pd-b} implies the
  contradiction that there is no arrow $T\rab T_0$.

  The final statement is a summary of the preceding series of lemmas.
   If $n_+(T)>0$, then $(T,T_0)\in\N$ and $n_-(T_0)>0$, which is contrary
  to Lemma~\rif{lemma:n2-0}.    The equalities $n_-(T)=m_-(T)=0$ are
  Lemmas~\rif{lemma:n2b} and~\rif{lemma:m2}.
\end{proof}

\begin{lemma}\libel{lemma:4}
  Let $\C$ be a cluster consisting of nonobtuse triangles.  Assume
  that the cardinality of $\C$ is four.  Then the strict main
  inequality holds for $\C$.
\end{lemma}

\begin{proof}  
  Let $\C = \{T_0\}\cup \{T_1^i \mid i=1,2,3\}$.  We have $b(T_1^i) =
  \area(T_1^i)\le \acrit$ by Lemma~\rif{lemma:T1-all}.  By
  Lemma~\rif{lemma:n2-0}, we have
  $n_-(T_0)=0$.

  We claim $m_-(T_0)=0$.  Otherwise, if $m_-(T_0) >0$, then
  $(T,T_1^i)\in\PD$ for some $T$ and some $i$.  The arrow $T_1^i\rab
  T_0$ is inconsistent with Lemma~\rif{lemma:pd-b}

  We claim $n_+(T_0)=0$.  Otherwise, if $(T_0,T_1^i)\in\N$, then we
  get $n_-(T_1^i)>0$, which is contrary to Lemma~\rif{lemma:T1-all}.

  Hence all the negative coefficients $n_-,m_-$ are zero on the cluster:
  $b(T_1^i)=\area(T_1^i)$ and $b(T_0) \ge \area(T_0)$.  The result now
  follows from Lemma~\rif{calc:pent4}.
\end{proof}

\begin{lemma}\libel{lemma:3}
  Let $\C$ be a cluster consisting of nonobtuse triangles.  Assume
  that the cardinality of $\C$ is three.  Then the strict main
  inequality holds for $\C$.
\end{lemma}

\begin{proof}  
  Let $\C = \{T_0\}\cup \{T_1^i \mid i=1,2\}$.  We have $b(T_1^i) =
  \area(T_1^i)\le \acrit$ by Lemma~\rif{lemma:T1-all}.  By
  Lemma~\rif{lemma:n2-0}, we have
  $n_-(T_0)=0$.

  We claim that $m_-(T_0)\le 1$.  Otherwise, by the definition of
  $\M$, there exists $(T,T_1^i)\in\PD$ for some $T$ and some $i$.  The
  arrow $T_1^i\rab T_0$ is inconsistent with Lemma~\rif{lemma:pd-b}.

This gives
\[
b(T_0) \ge \area(T_0) - \epsM
\]
By Lemma~\rif{calc:pent3}, we have
\[
\sum_{T\in\C} b(T) \ge (\area(T_0) - \epsM) +  \area(\C\setminus \{T_0\}) 
= -\epsM+ \area(\C) > 3\acrit.
\]
This is the strict main inequality for $\C$.
\end{proof}

\begin{lemma}\libel{lemma:2} 
  Let $\C$ be a cluster consisting of nonobtuse triangles.  Assume
  that the cardinality of $\C$ is two.  If the cluster is not a dimer
  pair, then the strict main inequality holds for $\C$.
\end{lemma}

\begin{proof}  
  Let $\C = \{T_1,T_0\}$, with $T_1\rab T_0$. By
  Lemma~\rif{lemma:T1-all}, $b(T_1)=\area(T_1)\le \acrit$.

  We assume that $\C$ is not a dimer pair $(T_1,T_0)$.

  We consider the case of a pseudo-dimer.  If $(T_1,T_0)\in\PD$, then
  $m_-(T_0)=n_-(T_0)=0$ (by Lemmas~\rif{lemma:pd-m2}
  and~\rif{lemma:pd-n2}).  Thus, by Lemma~\rif{calc:pseudo-area} and
  Lemma~\rif{lemma:T1-all},
\[
b(T_1)+b(T_0) 
\ge \area\{T_1,T_0\} + \epsM (n_+(T_0) + m_+(T_0)) 
> 2\acrit - \epsM  + \epsM (n_+(T_0)+m_+(T_0)).
\]
The main inequality follows if we show that $n_+(T_0)>0$ or
$m_+(T_0)>0$.  Assume for a contradiction that $n_+(T_0)=m_+(T_0)=0$.
Pick $T$ such that $T_0\Ra T$.  The condition $n_+(T_0)=0$ implies
that $(T_0,T)\not\in\N$.  The longest edge of $T_0$ is at least $1.72$
by Lemma~\rif{calc:pseudo2}.  According to the definition of $\M$, we
have $m_+(T_0)>0$ unless uniqueness fails: there exists $T'_1\ne T_1$
such $(T'_1,T_0)\in\PD$.  This is impossible by
Lemma~\rif{lemma:pd-1}, because the cardinality of $\C$ is only two.
This completes the case of a pseudo-dimer.

By Lemma~\rif{lemma:n2-0}, we have $n_-(T_0)=0$.  Thus,
by Lemma~\rif{lemma:T1-all},
\begin{equation}\libel{eqn:2}
b(T_0) \ge \area(T_0)  - \epsM m_-(T_0)
\quad b(T_1) = \area(T_1) \le \acrit.
\end{equation}

We may assume that $\area\{T_0,T_1\}> 2\acrit$.  Otherwise,
$(T_1,T_0)$ is a dimer pair or a pseudo-dimer, and these cases have
already been handled.  We assume for a contradiction that that the
strict main inequality is false:
\[
2\acrit \ge b(T_1) + b(T_0).
\]
Combined with the Inequalities~\rif{eqn:2},
this gives
\begin{align*}
2\acrit &\ge b(T_1) + b(T_0) \\
&\ge \area(T_1) + \area(T_0)  - \epsM m_-(T_0)\\
&> 2\acrit  - \epsM m_-(T_0).
\end{align*}
This implies that  $m_-(T_0) > 0$.

We consider the case $m_-(T_0)= 3$.  By Lemma~\rif{lemma:3m2}, we have
$b(T_1) + b(T_0) > \ao + (\acrit + \epsN) = 2\acrit$, which completes
this case.

We consider the case $m_-(T_0)=2$.  In this case, there are two
pseudo-dimers that share a longest edge with $T_0$.  These edges have
length at least $1.8$. The third edge is shared with $T_1$ and has
length at least $1.72$.  Then
\[
b(T_1)+b(T_0) 
\ge \ao + (\area (1.8,1.8,1.72) - 2\epsM) > 2\acrit.
\]
% checked 2016/9 in calcs.ml.

Finally, we consider the case $m_-(T_0)=1$.  There exists a
pseudo-dimer whose long edge $e$ is shared with $T_0$.  That edge has
length at least $1.8$, and the angle is not large along $(T_0,e)$. We
are in the context covered by Lemma~\rif{calc:large}.  That lemma
implies
\[
b(T_1) + b(T_0) \ge \area\{T_1,T_0\} - \epsM 
> (2\acrit  +\epsM) - \epsM = 2\acrit.
\]
This completes the proof.
\end{proof}

We are ready to give a proof of the pentagonal ice-ray conjecture. We
repeat the statement of the theorem (Theorem~\rif{thm:main}) from the
introduction of the article.

\begin{theorem}   
  No packing of congruent regular pentagons in the Euclidean plane has
  density greater than that of the pentagonal ice-ray.  The pentagonal
  ice-ray is the unique periodic packing of congruent regular
  pentagons that attains optimal density.
\end{theorem}

We combine the proof with a proof of Theorem~\rif{thm2:main}.

\begin{proof}
  By the main inequality in Lemma~\rif{lemma:main} applied to
  $a=\acrit$, together with Remark~\rif{rem:equal} it is enough to
  give a proof of Theorem~\rif{thm2:main}.  Specifically, we show that
  every cluster in every saturated packing is finite of cardinality at
  most $4$.  This is Corollary~\rif{lemma:card4}.  If $\C$ is a dimer
  pair, then $\C$ is the ice-ray dimer in the pentagonal ice-ray and
  the (weak) main inequality holds for $\C$, with equality exactly for
  the ice-ray dimer.  This is Theorem~\rif{thm:dimer} and
  Theorem~\rif{thm:dimer-area}.  If $\C$ is not a dimer pair, then we
  show that $\C$ satisfies the strict main inequality.  If $\C$ has an
  obtuse triangle, then this is found in Lemma~\rif{lemma:obtuse}.  If
  every triangle in $\C$ is nonobtuse, and if $\C$ has cardinality
  $4$, $3$, $2$, or $1$, then the result is found in
  Lemmas~\rif{lemma:4}, \rif{lemma:3}, \rif{lemma:2},
  \rif{lemma:singleton}.

  This completes the proof of the main theorem.
\end{proof}

\section{Appendix on Explicit Coordinates}\libel{sec:coords}
\libel{sec:appendix}

\subsection{two pentagons in contact}\libel{sec:two}

Let $A$ and $B$ be pentagons in contact, with $B$ the pointer at
vertex $\v_B$ to the receptor pentagon $A$.  Label vertices
$(\u_A,\w_A,\u_B,\v_B,\w_B)$ of $A$ and $B$ as in
Figure~\rif{fig:theta}.  Let $x=x_\alpha=\norm{\v_B}{\w_A}$ and $\alpha
= \angle(\v_B,\u_B,\u_A)$.  We have $0\le x_\alpha\le 2\sigma$ and
$0\le \alpha\le 2\pi/5$.

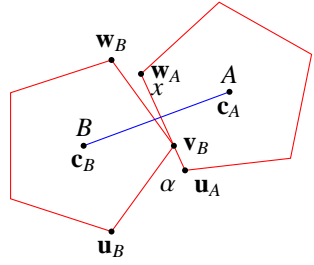
\tikzfig{theta}{coordinates for a pair of pentagons in contact}{
[scale=1.2]
\pen{0}{0}{0};
\pen{1.616}{0.598}{168.54};
\draw[blue] (0,0) node[anchor=south,black] {$B$}
 -- (1.616,0.598) node[anchor=south,black] {$A$};
\draw (0,0) node[anchor=north,black] {$\c_B$};
\smalldot{0,0};
\draw (1.616,0.598) node[anchor=north,black] {$\c_A$};
\smalldot{1.616,0.598};
\draw (72:1) node[anchor=south] {$\w_B$};
\draw (0:1) node[anchor=west] {$\v_B$};
\draw (-72:1) node[anchor=north] {$\u_B$};
\smalldot{72:1};
\smalldot{0:1};
\smalldot{-72:1};
\smalldot{0.636,0.797};
\draw (0.636,0.797) node[anchor=west] {$\w_A$}
node[anchor=north west] {$x$};
\smalldot{1.124,-0.2727};
\draw (1.124,-0.2727) node[anchor=north west] {$\u_A$}
node[anchor=north east] {$\alpha$};
}

Let $\bl = \bl(x_\alpha,\alpha) = \norm{\c_A}{\c_B}$, viewed as a
function of $x_\alpha$ and $\alpha$.  We omit the explicit formula for
$\bl$, but it is obtained by simple trigonometry.  Under the symmetry
$\u_A\leftrightarrow \w_A$, $\u_B\leftrightarrow\w_B$, we have the
transformation
\begin{equation}\libel{eqn:sym}
x\leftrightarrow 2\sigma-x,
\quad \alpha\leftrightarrow 2\pi/5-\alpha
\end{equation}
and
\[
\bl(x,\alpha) = \bl(2\sigma-x,2\pi/5 - \alpha).
\]
We use $(x_\alpha,\alpha)$ as the standard coordinates on the
configuration space of two pentagons in contact.  These coordinates
are particularly convenient, because they present the configuration
space as a rectangle $[0,2\sigma]\times[0,2\pi/5]$.

\subsection{angles of two pentagons in contact}\libel{sec:xalpha}

Let $A$ and $B$ be two pentagons in contact, given in coordinates by
$(x_\alpha,\alpha)$ as in the previous subsection.  Referring to
Figure~\rif{fig:theta}, $\v_B$ is the pointer vertex of $B$ into $A$,
and we have an angle $\theta' = \angle(\c_B,\c_A,\v_B)$ that specifies
the location of $\v_B$ relative to the segment $(\c_B,\c_A)$.  The
oriented angle $\angle(\c_B,\c_A,\v)\in[0,2\pi/5]$ to another vertex
$\v=\u_B,\w_B,\ldots$ of $B$ is $\theta' + 2\pi k/5$ for some integer
$k$.  We thus consider $\theta'$ as an angle defined module
$2\pi\ring{Z}/5$.  By subtracting a multiple of $2\pi/5$, we choose
the angle to lie in the range $[-\pi/5,\pi/5]$.  With these
conventions, in the figure, $\theta'$ is positive.

Referring to Figure~\rif{fig:theta}, $\u_A$ is a vertex of $A$, and we
have an angle $\theta = \angle(\c_A,\c_B,\u_A)$ that gives the
location of $\u_A$ relative to the segment $(\c_A,\c_B)$.  The angle
$\angle(\c_A,\c_B,\u)$ to another vertex of $A$ is $\theta + 2\pi k/5$
for some integer $k$.  Adjusting by a multiple of $2\pi/5$, we choose
$\theta$ to lie in the range $[0,2\pi/5]$.

We stress our convention that $\theta'$ refers to the angle on the
pointer pentagon and that $\theta$ refers to the angle on the receptor
pentagon.  It can be easily checked that 
\begin{equation}
\theta+\theta' \equiv \alpha \mod (2\pi\ring{Z}/5).
\end{equation}
(We remark in passing that we can define angles $\theta$, $\theta'$,
and $\alpha$ even when $A$ and $B$ are not in contact in such a way
that this relation still holds.)  We may consider $\theta$ and
$\theta'$ functions of the standard variables $(x_\alpha,\alpha)$.  With
our conventions $\theta'\in [-\pi/5,\pi/5]$ and $\theta\in [0,2\pi/5]$
on the range of these angles, $\theta'$ and $\theta$ are both
determined as continuous functions of $(x_\alpha,\alpha)$.  To obtain
continuity, it is necessary for $\theta$ to take both values $0$ and
$2\pi/5$, even though these values are equal modulo $2\pi\ring{Z}/5$.

In general, when two pentagons $A$ and $B$ come into contact,
sometimes $A$ is the pointer and sometimes $B$ is the pointer.  We
describe an extended coordinate system $(x_\alpha,\alpha)$, with a
domain $\alpha\in[0,4\pi/5]$ and $x_\alpha\in[0,2\sigma]$ of twice the
size that unifies both pointer directions.  See
Figure~\rif{fig:extended}.  When $\alpha\le2\pi/5$, the coordinates are
precisely as before.  Note that configurations with $\alpha=2\pi/5$
are ambiguous, with $B$ pointing into $A$ with coordinates
$(x_\alpha,\alpha)$, or with $A$ pointing into $B$ with coordinates
$(x'_\alpha,\alpha')=(2\sigma - x_\alpha,0)$.  When $\alpha > 2\pi/5$, we
let ($x_\alpha,\alpha)$ represent the configuration with $A$ pointing
into $B$ and coordinates $(x'_\alpha,\alpha') =
(2\sigma-x_\alpha,\alpha-2\pi/5)$.  The configuration of two pentagons
in contact depends continuously on the coordinates $(x_\alpha,\alpha)$.
The dependence is analytic except along $\alpha=2\pi/5$.

Using the continuous dependence of the configuration on the
coordinates, we may uniquely extend the functions $\theta$ and
$\theta'$ to continuous functions on the extended domain.  The
functions still represent the inclination angle of a vertex of $A$
(resp. $B$) with respect to the edge $(\c_A,\c_B)$.  However, when
$\alpha\ge 2\pi/5$, the range of $\theta'$ is $[0,2\pi/5]$ and the
range of $\theta$ is $[\pi/5,3\pi/5]$.  Also, $\theta$ becomes
the coordinate related to the pointer vertex (of $A$ into $B$).
Here is an explicit formula
for the extension on the domain $\alpha\ge 2\pi/5$:
\begin{align*}
\dx{AB}(x_\alpha,\alpha) &=\dx{AB}(2\sigma-x_\alpha,\alpha-2\pi/5),\\
\theta'(x_\alpha,\alpha) 
&= \theta(2\sigma-x_\alpha,\alpha-2\pi/5),\\
\theta(x_\alpha,\alpha) 
&= 2\pi/5+ \theta'(2\sigma-x_\alpha,\alpha-2\pi/5).
\end{align*}
By unifying both directions of pointing, extended coordinates lead to
a significant reduction in the number of cases to be considered.  In
fact, a single calculation can involve multi-triangle configurations
with several ($k$) edges in contact, and without extended coordinates
this leads to $2^k$ times the number of cases.

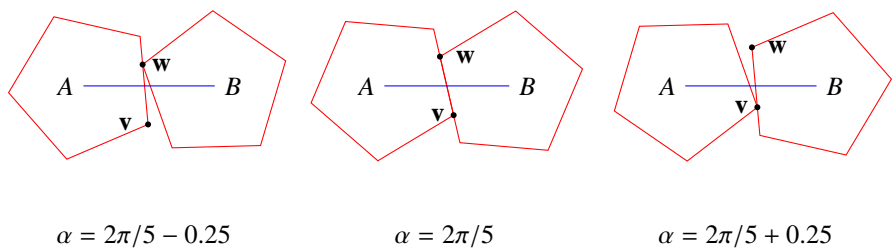
\tikzfig{extended}{Extended coordinates give a continuous transition
  from pairs $(A,B)$ with pointer $B\to A$ to pointer $A\to B$.  The
  functions $\theta,\theta'$ are continuous in $\alpha,x_\alpha$.  The
  function $\theta$ is the angle $\angle(\c_A,\c_B,\v)$, and $\theta'
  = \angle(\c_B,\c_A,\w)$. The vertices $\v$ and $\w$ transition into
  and away from the pointer vertex.  In these figures, $\alpha$ varies
  and $x_\alpha=0.8=\norm{\v}{\w}$ is fixed.}
{
\begin{scope}[scale=1]
\pen{0}{0}{41.16};
\pen{1.745}{0}{180-16.51};
\draw[blue] (0,0) node[black,anchor=east] {$A$} 
-- (1.745,0) node[black,anchor=west] {$B$};
\draw (0.8,-2) node {$\alpha = 2\pi/5 - 0.25$};
\smalldot{41.16-72:1};
\draw (41.16-72:1) node[anchor=east] {$\v~$~};
\smalldot{$(1.745,0)+(180-16.51:1)$};
\draw ($(1.745,0)+(180-16.51:1)$)  node[anchor=west] {$\w$};
\end{scope}
\begin{scope}[scale=1,xshift=4cm]
\pen{0}{0}{49.067};
\pen{1.66}{0}{180-22.93};
\draw (0.8,-2) node {$\alpha = 2\pi/5$};
\draw[blue] (0,0) node[black,anchor=east] {$A$} 
-- (1.66,0) node[black,anchor=west] {$B$};
\smalldot{49.067-72:1};
\draw (49.067-72:1) node[anchor= east] {$\v$~};
\smalldot{$(1.66,0)+(180-22.93:1)$};
\draw ($(1.66,0)+(180-22.93:1)$)  node[anchor=west] {$~\w$};
\end{scope}
\begin{scope}[scale=1,xshift=8cm]
\pen{0}{0}{55.489};
\pen{1.745}{0}{180-30.83};
\draw (0.8,-2) node {$\alpha = 2\pi/5+0.25$};
\draw[blue] (0,0) node[black,anchor=east] {$A$} 
-- (1.745,0) node[black,anchor=west] {$B$};
\smalldot{55.489-72:1};
\draw (55.489-72:1) node[anchor=east] {$\v$};
\smalldot{$(1.745,0)+(180-30.83:1)$};
\draw ($(1.745,0)+(180-30.83:1)$)  
node[anchor=west] {$~\w$};
\end{scope}
}
%(* fig:extended *)
%deg23 (ethetax (m 0.8) (pi25 - m 0.25));;
%deg23 (ethetax (m 0.8) (pi25 - m 0.0));;
%deg23 (ethetax (m 0.8) (pi25 + m 0.25));;

\subsection{two pentagons in contact, alternative
  coordinates}\libel{sec:banana}

Inversely, $d_{AB} = \norm{\c_A}{\c_B}$, $\theta$, and $\theta'$
determine $(x_\alpha,\alpha)$.  In fact, any two of $d_{AB}$,
$\theta$, $\theta'$ determine $(x_\alpha,\alpha)$ up to finite
ambiguity.  In general, there can be two configurations of pentagons
for a given $\dx{AB}$ and $\theta'$. See Figure~\rif{fig:banana2}.
They can be distinguished by a boolean variable giving the sign of $h
:= x_\alpha-\sigma$.

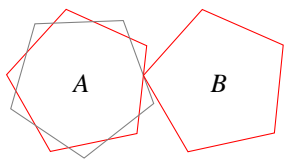
\tikzfig{banana2}{With given fixed center $\c_A$ and fixed pointer
  pentagon $B$, there might be two pentagons $A$ in contact with
  $B$. They are distinguished by a boolean variable indicating whether
  the pointer $B\to A$ lies above or below the midpoint of the edge on
  $A$.}
{
\begin{scope}[scale=1,xshift=5cm]
\pen{0}{0}{29.68};
\pentpink{0}{0}{-15.887}{1}{black!50};
\pen{1.82}{0}{180-5.73};
\draw (0,0) node[black] {$A$} ;
\draw (1.82,0) node[black] {$B$};
\end{scope}
}
%(* fig:banana2 *)
%let Some a = (theta_banana (m 1.82) (m 0.1)) in 
%degofrad a.high;;
%let Some a = theta_banana (m 1.82) (m (-. 0.1)) in
%degofrad a.high;;
%degofrad 0.1;;

The computer code implements a function that generates a configuration
of two pentagons $(A,B)$ in contact as a function of $\dx{AB}$ and
$\theta'$, for contact type $B\to A$ and $h\ge0$.  (This function does
not use extended coordinates.)  Our proof of the pentagonal ice-ray
conjecture depends on having fast, numerically-stable algorithms for
computing configurations in terms of these variables on intervals.  It
is a matter of simple trigonometry to express $\theta$ in terms of
$\dx{AB}$ and $\theta'$.  However, it is somewhat more work to give
good interval bounds on $\theta$ as a function of intervals $\dx{AB}$
and $\theta'$.  The implementation of this function is based on
detailed information about the image of $(d_{AB},\theta')$ as
functions on the set of pairs $(A,B)$ with $B\to A$ and $h\ge 0$.  The
image is a convex region in the plane with with a piecewise analytic
boundary.

\subsection{pentagon existence test}\libel{sec:pet}

The values $(\dx{AB},\theta,\theta')$ can be defined for a pair of
nonoverlapping pentagons, even when $A$ and $B$ do not touch.  The
computer code implements a test to determine whether a given triple
$(d,\theta,\theta')$ is equal to a triple
$(\dx{AB},\theta,\theta')$ associated with some pair $A,B$ of
pentagons in contact.  More generally, when $(d,\theta,\theta')$
are interval-valued variables, the computer code implements a test to
determine if there exist pentagons $A$ and $B$ that do not overlap and
whose values lie in the given intervals.  The interested reader can
consult the computer code for the details of the test.

\subsection{coordinates for zero, single, and double contact
  triangles}

The configuration space of $P$-triangles in which no pair of pentagons
is in contact is six dimensional.  Let $A$, $B$, $C$ be the pentagons
centered at the vertices $\c_A$, $\c_B$, $\c_C$ of the triangle.  For
example, we can use the six coordinates
$\dx{AB},\dx{BC},\dx{AC},\theta_{ABC},\theta_{BCA},\theta_{CAB}$.  The
three edge lengths $\dx{XY}$ of the triangle determine the triangle up
to congruence, and the three angles $\theta_{XYZ}$ determine the
inclination of each pentagon $X$ relative to the edge $(\c_X,\c_Y)$ of
the triangle.  See Figure~\rif{fig:p-triangle}.  Other quantities can
be easily computed from these six coordinates.  For example, to
compute $\theta_{ACB}$, the angle of pentagon $A$ relative to the edge
$(\c_A,\c_C)$, we use the relation
\begin{equation}
\arc_A + \theta_{ABC} + \theta_{ACB} \cong 0\mod 2\pi/5.
\end{equation}
where $\arc_A$ is the angle of the triangle at vertex $\c_A$.  
Similar equations hold for the angles at $\c_B$ and $\c_C$.

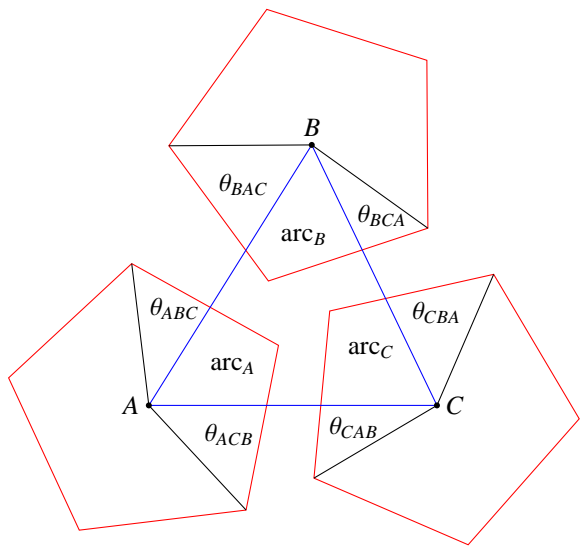
\tikzfig{p-triangle}{Angle conventions. 
  The edge lengths of the Delaunay triangle are
  $d_{AC}$, $d_{AB}$, and $d_{BC}$.  The angles at $\c_A$ are
  $\theta_{ABC}$, $\arc_A$, and $\theta_{ACB}$.  The angles at $\c_B$
  are $\theta_{BAC}$, $\arc_B$, and $\theta_{BCA}$. The angles at
  $\c_C$ are $\theta_{CBA}$, $\arc_C$, and $\theta_{CAB}$.}
{
\begin{scope}[scale=1.9,xshift=5cm]
\def\x{1.2}
\threepent
{0.00}{0.00}{96.88}
{\x*0.95}{\x*1.52}{252.34}
{\x*1.68}{0.00}{210.61};
\draw(0,0) node[anchor=east] {$A$};
\draw(\x*0.95,\x*1.52) node[anchor=south] {$B$};
\draw(\x*1.68,0) node[anchor=west] {$C$};
\smalldot{0,0};
\smalldot{\x*0.95,\x*1.52};
\smalldot{\x*1.68,0};
\draw(0,0) -- ++ (96.88:1);
\draw(0,0) -- ++ (96.88- 2*72:1);
\draw(\x*0.95,\x*1.52) -- ++ (252.34+72:1);
\draw(\x*0.95,\x*1.52) -- ++ (252.34-72:1);
\draw(\x*1.68,0) -- ++ (210.61 - 2*72:1);
\draw(\x*1.68,0) -- ++ (210.61:1);
\draw(25:0.65) node {$\arc_A$};
\draw(75:0.7) node {$\theta_{ABC}$};
\draw(-20:0.6) node {$\theta_{ACB}$};
\draw($(\x*0.95,\x*1.52)+(-45:0.7)$) node {$\theta_{BCA}$};
\draw($(\x*0.95,\x*1.52)+(-95:0.65)$) node {$\arc_B$};
\draw($(\x*0.95,\x*1.52)+(210:0.55)$) node {$\theta_{BAC}$};
\draw($(\x*1.68,0)+(90:0.65)$) node {$\theta_{CBA}$};
\draw($(\x*1.68,0)+(140:0.60)$) node {$\arc_C$};
\draw($(\x*1.68,0)+(195:0.6)$) node {$\theta_{CAB}$};
\end{scope}
}
% (* fig:p-triangle *)
% format_pinwheelAC (m 0.1) (m 0.1) (m 0.3);;

Each $P$-triangle determines a six-tuple $(\dx{AB},\dx{BC},\ldots)$.
We can algorithmically test whether a six-tuple of real numbers (or of
intervals) comes from $P$-triangles by using the triangle inequality
and the test of Section~\rif{sec:pet}.

The configuration space of $P$-triangles in which a single pair
$(A,B)$ of pentagons comes into contact is five-dimensional.  Let $A$,
$B$, $C$ be the pentagons centered at the vertices $\c_A$, $\c_B$,
$\c_C$ of the triangle.  Assume that $A$ is in contact with $C$, with
pointer $A\to C$.  For example, we can use the five coordinates
$\dx{AB},\dx{BC},\dx{AC},\theta_{ACB},\theta_{BAC}$.  The three edge
lengths $\dx{XY}$ of the triangle determine the triangle up to
congruence, and the two angles $\theta_{XYZ}$ determine the
inclination of each pentagon relative to the triangle.  It is clear
that the $\theta_{ACB}$ fixes the inclination of $A$ and that
$\theta_{BAC}$ fixes the inclination of $B$.  The variable
$\theta_{CAB}$ is computed from $\dx{AC}$ and $\theta_{ACB}$ by the
procedure in Section~\rif{sec:banana}.  Other quantities can be easily
computed from these quatitites. Again, we can algorithmically
test whether a given five-tuple $(\dx{AB},\dx{BC},\ldots)$ 
comes from a $P$-triangle.

The configuration space of $P$-triangles in which two separate pairs
$(A,B)$ and $(B,C)$ of pentagons comes into contact is
four-dimensional.  Let $A$, $B$, $C$ be the pentagons centered at the
vertices $\c_A$, $\c_B$, $\c_C$ of the triangle.  We can use the
extended coordinates $(x_\alpha,\alpha)$ of Section~\rif{sec:xalpha}
to give the relative position of $A$ and $B$ and similar coordinates
$(x_\gamma,\gamma)$ to give the relative position of $B$ and $C$.
These four coordinates uniquely determine the $P$-triangle up to
finite ambiguity.  The ambiguity is resolved by specifying which edge
of $B$ is in contact with $A$ and which edge is in contact with $C$.
These four coordinates determine other quantities such as lengths
$\dx{AB}$ and $\dx{BC}$ and angles $\theta_{BAC}$, $\theta_{ABC}$,
$\theta_{BCA}$, $\theta_{CBA}$.  Once again, we can algorithmically
test whether a given four-tuple $(x_\alpha,\alpha,x_\gamma,\gamma)$
comes from a $P$-triangle.

\subsection{triple contact}

If we have coordinates on a $3C$-triangle that determine the variables
$(x_\alpha,\alpha)$ for each of the pairs $\{A,B\}$, $\{B,C\}$, and
$\{A,C\}$ of pentagons, then we may use the function $\bl$ of
Section~\rif{sec:two} to calculate the edge lengths and area of the
$3C$-triangle.

The configuration space of $3C$-triangles is three dimensional,
obtained by imposing three contact constraints between pairs of
pentagons on the six-dimensional configuration space of all
$P$-triangles.

\subsection{$3C$ triangles with a shared edge}\label{sec:shared}

Some calculations deal with a dimer pair of $3C$ triangles sharing a
triangle edge and two pentagons $\bar{A}$ and $\bar{C}$, say with
$\bar{A}$ pointing into $\bar{C}$.  (We place bar accents on symbols
in this subsection for compatibility with the sections that follow.)
Let $\bar{B}$ and $\bar{D}$ be the two outer pentagons of the dimer
pair.  The configuration space of dimer in which both triangles have
triple contact is four dimensional.  In this situation, it is best to
develop coordinate systems that make efficient use of the shared
information.  Associated with the pentagons $(\bar{A},\bar{C})$ in
contact are two variables $(\bar x_\beta,\bar\beta)$ (that we rename
from $(x_\alpha,\alpha)$ in Section \ref{sec:xalpha}).  It is generally
advantageous to make $(\bar x_\beta,\bar\beta)$ two of the four
coordinates on the dimer pair.  We supplement this with one further
angle $\bar\alpha$ to determine the position of $\bar B$ and a further
angle to determine $\bar D$.

We focus on the $P$-triangle $(\bar A,\bar B,\bar C)$, with
coordinates $(\bar\alpha,\bar\beta,\bar x_\beta)$.  Similar
considerations apply to the other $P$-triangle $(\bar A,\bar D,\bar
C)$.  We assume hat $\bar A$ points into $\bar C$.

For each triple contact, the following sections give further details
about its coordinate systems.  There will be a {\it shared edge}
coordinate system for each contact type and each pair of pentagons
$(\bar A,\bar C)$ such that $\bar A$ points into $\bar C$.

We call a $\Gamma$-triangle to be a $P$-triangle that appears as one
of the two triangles (with given shared edge) in the curve $\Gamma$ of
Section~\rif{sec:gamma}.  We set up coordinates in a uniform manner so
that regardless of the contact type of the $3C$ triangle, the curve
$\Gamma$ is parameterized by variable $t$ and is given as
\[
t = \bar x_\beta  -\sigma, \quad
\bar\alpha = 0,\quad
\bar\beta = \pi/5.
\]
(This is the curve restricted to the $P$-triangle $(\bar A,\bar B,\bar
C)$, with a similar description on the other triangle $(\bar A,\bar
D,\bar C)$ in the dimer.)

When the the configuration space of a give $3C$-type contains
$\Gamma$-triangles and the ice-ray triangle, we also describe a path
$P$ with parameter $s$ from an arbitrary triangle in that
configuration space to a $\Gamma$-triangle.  Again, we set up
coordinates uniformly so that the formulas are independent of the
contact type of the $P$-triangle.  Coordinates will be defined in such
a way that for all points of the domain, the relation $\bar\alpha\ge
0$ holds.  The path $P$ from an arbitrary point in the domain
$(\bar\alpha_0,\bar\beta_0,\bar x_{\beta,0})$ to the curve $\Gamma$ is
defined by functions $s\mapsto (\bar\alpha(s),\bar\beta(s),\bar
x_\beta(s))$ (with parameter $s\ge 0$), where
\begin{align*}
\bar\alpha(s) &= \begin{cases}
\bar\alpha_0 - s, & \bar\alpha_0 > 0;\\
0, &\bar\alpha_0=0;
\end{cases}\\
\bar\beta(s) &= \begin{cases}
\bar\beta_0 - s, & \bar\beta_0 > \pi/5;\\
\bar\beta_0 + s, & \bar\beta_0 < \pi/5;\\
\pi/5, &\bar\beta_0=\pi/5;
\end{cases}\\
\bar x_\beta(s) &= \bar x_{\beta,0}.
\end{align*}
Note that $\bar x_\beta$ remains constant along the path $P$.  This
path is to be understood piecewise.  That is, the path continues until
a boundary is hit (say $s$ such that $\bar\beta_0-s = \pi/5$ or
$\bar\alpha_0-s = 0$).  Then a new initial value
$(\bar\alpha_1,\bar\beta_1)$ is set at the boundary, and the path
continues.  The path is continuous and piecewise linear.  The path
terminates when $\bar\alpha = 0$ and $\bar\beta = \pi/5$.  These
termination conditions are the defining conditions of the image of
$\Gamma$.  Thus, the path leads to $\Gamma$ in all cases.
If the initial configuration satisfies
\[
|\bar\alpha_0|\le M,\quad 
|\bar\beta_0-\pi/5|\le M,\quad |\bar x_{\beta,0}-\sigma|\le M,
\]
then these inequalities for $(\bar\alpha(s),\bar\beta(s),\bar
x_\beta(s))$ hold along the path $P$.

We write $\bar\alpha(s,\bar\alpha_0)$, $\bar\beta(s,\bar\beta_0)$, and
$\bar x_\beta(s,\bar x_{\beta,0})$ to show the dependence of the path
$P$ on the initial point.  The dimer pair has a second triangle $(\bar
A,\bar D,\bar C)$ whose initial configuration has coordinates
$(\bar\alpha_1,\bar\beta_1,\bar x_{\beta,1})$, where by Equation
(\rif{eqn:sym}) we have
\[
\bar\beta_0+\bar\beta_1 = \pi/5,\quad 
\bar x_{\beta,0}+\bar x_{\beta,1} = 2\sigma.
\]
The equations for the path on the two triangles then satisfy
\[
\bar\beta(s,\bar\beta_0)+\bar\beta(s,\bar\beta_1)=\pi/5,\quad
\bar x_{\beta}(s,\bar x_{\beta,0}) + \bar x_{\beta}(s,\bar x_{\beta,1}) = 2\sigma.
\]
This means that the paths for the two triangles are coherent along the
shared pentagons $(\bar A,\bar C)$ and determine a path of the full
dimer $(\bar A,\bar B,\bar C,\bar D)$ that terminates at the curve
$\Gamma$.

\subsection{triple contact at a $\Delta$-junction}

We describe a coordinate system on $3C$-triangles of $\Delta$-junction
type.  As indicated in Figure~\rif{fig:cord-delta}, we use coordinates
$(x_\alpha,\alpha,\beta)$, where $x_\alpha$ is a length and $\alpha$
and $\beta$ are each angles between lines through edges of pentagons
in contact.  We assume that $B$ points into $A$ and into $C$ and that
$A$ points into $C$. The length $x_\alpha$ is the (small) distance
between the nearly coincident vertices of pentagons $B$ and $C$.  The
coordinates satisfy the conditions $0\le\beta\le\alpha\le\pi/5$,
$\alpha+\beta\le \pi/5$, and $x_\alpha\in[0, 2\sigma -
\sigma/\kappa]$.  Starting from these coordinates, we define $\gamma$
by $\alpha+\beta+\gamma=\pi/5$, and angles $\alpha',\beta',\gamma'$ of
the inner triangle $\Delta$ by
\begin{align}\libel{eqn:abc}
\alpha+\alpha' &= 2\pi/5,\\
\beta+\beta' &= 2\pi/5,\nonumber\\
\gamma+\gamma' &= 2\pi/5.\nonumber
\end{align}
The edges $y_\alpha$, $y_\beta = 2\sigma$, and $y_\gamma$ of the
triangle $\Delta$ opposite the angles $\alpha'$, $\beta'$, $\gamma'$,
respectively are easily computed by the law of sines.  Define
$x_\beta$ by $x_\alpha+y_\gamma+x_\beta=2\sigma$, and $x_\gamma$ by
$y_\alpha+x_\gamma=2\sigma$.  The value $x_\beta$ is the distance
between the nearly coincident vertices of pentagons $A$ and $C$, and
$x_\gamma$ is the distance between the nearly coincident vertices of
pentagons $A$ and $B$.  The edges of the $3C$ Delaunay triangle have
lengths
\[
\bl(x_\alpha,\alpha'),\quad \bl(x_\beta,\beta'),\quad \bl(x_\gamma,\gamma').
\]
By Lemma~\rif{lemma:delta}, triangles of type $\Delta$ have area too
large to be relevant for calculations of dimers.  There is no need to
describe the shared coordinates $(\bar \alpha,\bar\beta,\bar x_\beta)$
in this case.

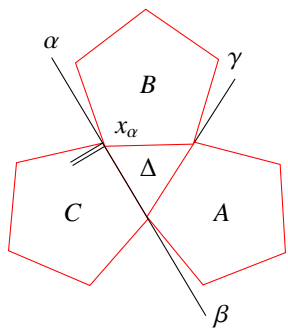
\tikzfig{cord-delta}{Coordinates for $\Delta$-types}
{
[scale=1.0]
\threepentnoD
{0.00}{0.00}{-5.16}%C
{0.99}{1.70}{235.38}%B
{1.98}{0.00}{183.43}; %A
\draw (0,0) node {$C$};
\draw (0.99,1.70) node {$B$};
\draw (1.98,0.0) node {$A$};
\draw (1.0,0.6) node {$\Delta$};
%\draw[blue] (0,0) -- (1.98,0);
\draw (-5.16:1) -- ++ (126 - 5.16 :2.5) node[anchor=south] {$\alpha$};
\draw (-5.16:1) -- ++ (- 180 + 126 - 5.16 :1.5) node[anchor=west] {$\beta$};
\draw ++ (1.98,0) ++ (183.43 - 72:1) 
-- ++ (54 + 3.43:1) node[anchor=south] {$\gamma$};
\draw (72-5.16:1) -- ++ (126 + 90 - 5.16:0.5);
\draw (0.99,1.70) ++ (235.38:1) node[anchor = south west] {$x_\alpha$} 
-- ++ (126+90 - 5.16:0.5);
}

\subsection{triple contact at a pinwheel type}

We describe a coordinate system on $3C$-triangles of pinwheel type.
We assume that $C$ points into $B$, that $B$ points into $A$, and that
$A$ points into $C$.  As indicated in Figure~\rif{fig:cord-pinwheel},
we use coordinates $(\alpha,\beta,x_\gamma)$.  The angles $\alpha$ and
$\beta$ are angles between pentagon edges on touching pentagons.  The
value $x_\gamma$ is the distance between the pointer vertex of
pentagon $A$ and the pointer vertex of pentagon $B$.  The coordinates
satisfy constraints: $0\le\alpha$, $0\le\beta$, $\alpha+\beta\le
\pi/5$, and $0\le x_\gamma\le 2\sigma$.  Define $\gamma$ by
$\alpha+\beta+\gamma=\pi/5$.  The angles $\alpha'$, $\beta'$, and
$\gamma'$ of the inner background triangle $P$ of the pinwheel are
given by Equation~\rif{eqn:abc}.  The edge lengths $x_\alpha$,
$x_\beta$, $x_\gamma$ of the inner triangle $P$ are easily computed
from $(\alpha,\beta,x_\gamma)$ by the law of sines.  The edges of the
$3C$-triangle have lengths
\[
\bl(x_\alpha,\alpha),\quad 
\bl(x_\beta,\beta),\quad \bl(x_\gamma,\gamma).
\]

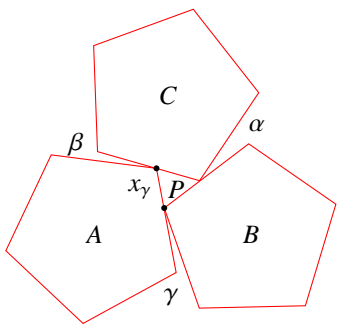
\tikzfig{cord-pinwheel}{Coordinates for pinwheel type}{
\begin{scope}[xshift=4cm,scale=1.2]
\threepentnoD{0.00}{0.00}{46.69}%A
{0.82}{1.53}{218.09}%C
{1.73}{0.00}{163.28};%B
\draw(0,0) node {$A$};
\draw(0.82,1.53) node {$C$};
\draw(1.73,0) node {$B$};
\draw(1.8,1.2) node {$\alpha$};
\draw(-0.2,1.0) node {$\beta$};
\draw(0.85,-0.7) node {$\gamma$};
\draw(0.9,0.5) node {$P$};
\draw(0.5,0.5) node {$x_\gamma$};
\smalldot{0.772,0.288}; %B pointer.
\smalldot{0.686,0.728}; %A pointer.
\end{scope}
}

\subsubsection{pinwheel type with a shared edge}

Now we specialize this discussion of Section \rif{sec:shared} to
pinwheels.  Pinwheels have a rotational symmetry, so we may assume
without loss of generality that the shared edge is $A\to C$ (meaning,
$A$ and $C$ are the shared pentagons and $A$ points to $C$).  This
choice of shared edge gives $(\bar A,\bar C) = (A,C)$.  The shared
variables are $(\bar x_\beta,\bar\beta)=(x_\beta,\beta)$.  The
nonshared variable is $\bar\alpha=\alpha$.

\subsection{triple contact at a $LJ$-junction type}

We describe a coordinate system on $3C$-triangles of $LJ$-junction
type.  As indicated in Figure~\rif{fig:cord-L}, we use coordinates
$(\alpha,\beta,x_\alpha)$.  The angles $\alpha$ and $\beta$ are each
formed by edges of two pentagons in contact.  Let $x_\alpha$ be the
distance between the pointer vertex of $C$ to $A$ and the pointer
vertex of $B$ to $C$.  The coordinates satisfy relations:
$\alpha,\beta\in [0,2\pi/5]$, $\pi/5\le\alpha+\beta\le 3\pi/5$, and
$0\le x_\alpha\le 2\sigma$.  Define $\gamma$ by
$\alpha+\beta+\gamma=3\pi/5$.  The three acute angles $\alpha'$,
$\beta'$, and $\gamma'$ of the inner $L$-shaped quadrilateral are
given by Equation~\rif{eqn:abc}.  The edge lengths of the $L$-shaped
quadrilateral are easily computed by triangulating the quadrilateral
into two triangles and applying the law of sines.  (Triangulate $L$ by
extending the line through the edge of $A$ containing the pointer
vertex of $C$ into $A$.)  This gives $x_\beta$, the distance between
the pointer vertex of $C$ to $A$ and the inner vertex of $A$.  This
gives $x_\gamma$, the distance between the pointer vertex of $B$ and
the inner vertex of pentagon $A$.  As before, the edges of the
$3C$-triangle have lengths
\[
\bl(x_\alpha,\alpha),\quad \bl(x_\beta,\beta),\quad \bl(x_\gamma,\gamma).
\]

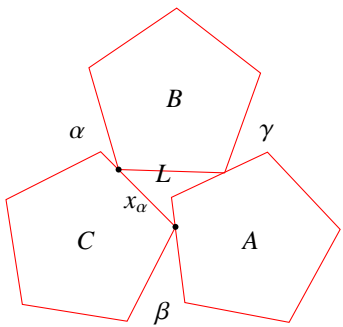
\tikzfig{cord-L}{Coordinates for $LJ$-junction type}{
\begin{scope}[xshift=8cm,scale=1.2]
\threepentnoD{0.00}{0.00}{80.86}
{0.97}{1.58}{232.21}
{1.82}{0.00}{223.24};
\draw (0,0) node {$C$};
\draw (0.97,1.58) node {$B$};
\draw (1.82,0) node {$A$};
\draw (2.0,1.2) node {$\gamma$};
\draw (-0.1,1.2) node {$\alpha$};
\draw (0.85,-0.8) node {$\beta$};
\draw (0.85,0.75) node {$L$};
\draw (0.55,0.4) node {$x_\alpha$};
\smalldot{$(0,0)+(80.86-72:1)$};
\smalldot{$(0.97,1.58)+(232.24:1)$};
%\smalldot{0.94,0.48};
%\smalldot{1.532,0.753};
\end{scope}
}

\subsubsection{$LJ_1$-junction type}

Now we specialize to $LJ$-junctions with a shared edge $C\to A$.
Then $(\bar A,\bar B,\bar C)=(C,B,A)$.

The shared variables are $(\bar \beta,\bar x_\beta)=(\beta,x_\beta)$.
The nonshared variable is $\bar\alpha=\gamma' = 2\pi/5-\gamma\ge0$.
These coordinates are numerically stable.  We compute other angles and
edges by triangulating the $L$ by extending the receptor edge of $A$
and using the law of sines.

\subsubsection{$LJ_2$-junction type}

We specialize to $LJ$-junctions with a shared edge $B\to C$.
In this case, $(\bar A,\bar B,\bar C)=(B,A,C)$.

The shared variables are $(\bar x_\beta,\bar\beta)=(x_\alpha,\alpha)$.
The nonshared variable is $\bar\alpha=\beta$.  These coordinates are
numerically stable.  They are exactly the standard variables given
above for a general $LJ$-junction.

\subsubsection{$LJ_3$-junction type}\libel{sec:one-ljx}

We specialize to $LJ$-junctions with a shared edge $B\to A$.
In this case, $(\bar A,\bar B,\bar C)=(B,C,A)$.

The shared variables are $(\bar x_\beta,\bar\beta)=(x_\gamma,\gamma)$.
The nonshared variable is $\beta$.  If $\beta > 0.9$ a
calculation\footnote{one\_ljx} shows that the triangle is not
subcritical.  We may therefore assume that $\beta\le 0.9$.  Under this
additional assumption, these coordinates are numerically stable.  We
compute other lengths and angles by triangulating the $L$-region by
extending the receptor edge of $A$.

There does not exist an ice-ray triangle with longest edge along the
edge $(A,B)$.

\subsection{triple contact at a $TJ$-junction type}

We describe a coordinate system on $3C$-triangles of $TJ$-junction
type.  As indicated in Figure~\rif{fig:cord-T}, we use coordinates
$(\alpha,\beta,x_\gamma)$.  The angles $\alpha$ and $\beta$ are each
formed by edges of two pentagons in contact.  The length $x_\gamma$ is
the distance between the pointer vertex of $B$ into $A$ and the inner
vertex of $A$.  The coordinates satisfy:
$\alpha,\beta\in[\pi/5,2\pi/5]$, $3\pi/5\le \alpha+\beta\le 4\pi/5$,
$0\le x_\gamma\le 2\sigma$.  Define $\gamma$ by
$\alpha+\beta+\gamma=\pi$.  Three of the angles $\alpha'$, $\beta'$,
and $\gamma'$ of the inner irregular $TJ$-shaped pentagon $P$ are
given by Equation~\rif{eqn:abc}.  The edge lengths of the $TJ$-shaped
pentagon are easily computed by triangulating $P$ into three triangles
and applying the law of sines.  (Triangulate by extending the edge of
$P$ shared with $A$ that ends at the pointer vertex of $A$ into $C$
and by extending the edge of $P$ shared with $C$ that contains pointer
vertex of $B$ into $C$.)  This gives $x_\alpha$, the distance between
the pointer vertex of $B$ to $C$ and the inner vertex of $C$.  This
gives $x_\beta$, the distance between the pointer vertex of $A$ to $C$
and the inner vertex of $C$.  As before, the edges of the
$3C$-triangle have lengths
\[
\bl(x_\alpha,\alpha),\quad \bl(x_\beta,\beta),\quad \bl(x_\gamma,\gamma).
\]

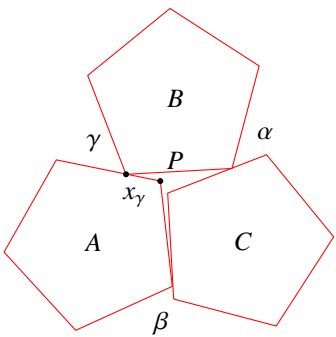
\tikzfig{cord-T}{Coordinates for $TJ$-junction types}{
\begin{scope}[xshift=12cm,scale=1.2]
\threepentnoD{0.00}{0.00}{114.48} %A
{0.90}{1.59}{237.18} %B
{1.66}{0.00}{219.24}; %C
\draw (0,0) node {$A$};
\draw (0.9,1.59) node {$B$};
\draw (1.66,0) node {$C$};
\draw (0.9,0.9) node {$P$};
\draw (0,1.1) node {$\gamma$};
\draw (0.75,-0.9) node {$\beta$};
\draw (1.9,1.2) node {$\alpha$};
\draw (0.45,0.5) node {$x_\gamma$};
\smalldot{0.737,0.675};
\smalldot{0.358,0.75};
\end{scope}
}

\subsubsection{$TJ_1$-junction type}\libel{sec:one-tjx}

We specialize to $TJ$-junctions with a shared edge $A\to C$.  In this
case, $(\bar A,\bar B,\bar C)=(A,B,C)$.

The shared variables are $(\bar x_\beta,\bar\beta)=(x_\beta,\beta)$.
The nonshared variable is $\bar\alpha=\alpha$.  If $\beta < 1.0$ a
calculation\footnote{one\_tjx} shows that the triangle is not
subcritical.  We may therefore assume that $\beta \ge 1.0$.  Under
this additional assumption, these coordinates are numerically stable.
We compute other lengths and angles by triangulating the $TJ$-region
by extending the receptor edge of $A$ and extending the receptor edge
of $C$.

There does not exist an ice-ray triangle with longest edge along the
edge $(A,C)$.

\subsubsection{$TJ_2$-junction type}

We specialize to $TJ$-junctions with a shared edge $B\to A$. In this
case, $(\bar A,\bar B,\bar C)=(B,C,A)$.

The shared variables are $(\bar x_\beta,\bar\beta)=(x_\gamma,\gamma)$.
The nonshared variable is $\bar\alpha =\beta$.  From
$\alpha+\beta+\gamma=\pi$, we obtain $(\alpha,\beta,x_\gamma)$, which
are the standard coordinates described above for $TJ$-junction
triangles.  These coordinates are numerically stable.

There does not exist an ice-ray triangle with longest edge along the
edge $(A,B)$.

\subsubsection{$TJ_3$-junction type}

We specialize to $TJ$-junctions with a shared edge $B\to C$.
In this case, $(\bar A,\bar B,\bar C)=(B,A,C)$.

The shared variables are $(\bar x_\beta,\bar\beta)=(x_\alpha,\alpha)$.
The nonshared variable is $\bar\alpha=\beta'=2\pi/5-\beta$.  These
coordinates are numerically stable.  We compute other lengths and
angles by triangulating the $TJ$-region by extending the receptor edge
of $A$ (of $B\to A$) and extending the receptor edge of $C$ (of $A\to
C$).

\subsection{triple contact at a pin-$T$ junction type}\libel{sec:pint}
We describe a coordinate system on $3C$-triangles of pin-$T$ junction
type.  As indicated in Figure~\rif{fig:cord-pint}, we use coordinates
$\alpha$, $\beta$, and $x_\alpha$.  The angles $\alpha$ and $\beta$
are each formed by edges of two pentagons in contact.  The length
$x_\alpha$ is the distance between the nearly coincident vertices of
$B$ and $C$.  The coordinates satisfy $\pi/5 \le \alpha \le 2\pi/5$,
$\pi/5\le \beta \le 2\pi/5$, and $3\pi/5 \le \alpha+\beta$.
Lemma~\rif{lemma:0605} shows that $0\le x_\alpha\le 0.0605$.  Define
$\gamma$ by $\alpha+\beta+\gamma = \pi$.  Three of the angles
$\alpha'$, $\beta'$, and $\gamma'$ of the inner irregular $T$-shaped
pentagon are given by Equation~\rif{eqn:abc}.  The edge lengths of the
$T$-shaped pentagon $P$ are easily computed by triangulating $P$ into
three triangles and applying the law of sines.  (Triangulate by
extending the two edges of $P$ that meet at the pointer vertex of $C$
into $A$.)  This gives $x_\beta$, the distance between the pointer
vertex of $C$ to $A$ and the inner vertex of $A$.  This gives
$x_\gamma$, the distance between the pointer vertex of $A$ to $B$ and
the pointer vertex of $B$ into $C$.  As before, the edges of the
$3C$-triangle have lengths
\[
\bl(x_\alpha,\alpha),\quad \bl(x_\beta,\beta),\quad \bl(x_\gamma,\gamma).
\]

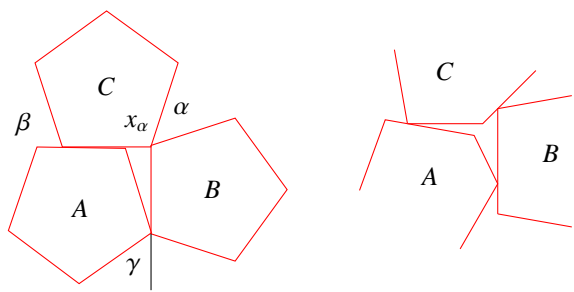
\tikzfig{cord-pint}{Coordinates for pin-$T$ junction types. Although
  it is difficult to tell from the figure, $A$ points into $B$, $B$
  into $C$, and $C$ into $A$.  The parameter $\beta'\approx 0$
  measures the incidence angle between the nearly horizontal edges of
  $A$ and $C$.  The region bounded by the three pentagons is a
  $T$-shaped pentagon, with stem between the nearly parallel edges of
  $A$ and $C$ and two arms along $B$.  The arm between $B$ and $C$ is
  almost imperceptible.  The distance between the neighboring vertices
  of $B$ and $C$ is $x_\alpha$.  The figure on the right distorts the
  pentagons to make the incidence relations more apparent.}
{
\begin{scope}[xshift=4.5cm,yshift=1.5cm]
\threepentnoD{0.00}{0.00}{90} %C
{1.40}{-1.377}{0} % B
{-0.35}{-1.628}{-18.89}; % A
\draw (0,0) node {$C$};
\draw (1.4,-1.37) node {$B$};
\draw (-0.35,-1.62) node {$A$};
\draw (0.36,-2.4) node {$\gamma$};
\draw (-1.1,-0.5) node {$\beta$};
\draw (1.0,-0.3) node {$\alpha$};
\draw (0.4,-0.49) node {$x_\alpha$};
\draw (1.4 - 0.809,-1.377 - 0.5878) -- +(0,-0.75);
\end{scope}
\begin{scope}[xshift=9.5cm,yshift=1cm]
\draw[red] (0.2,-1.2) -- ++ (90:1.4) -- ++ (10:1); % B
\draw[red] (0.2,-1.2) -- ++ (-10:1);
\draw[red] (-1,0) -- ++(1,0) -- ++ (45:1); %C
\draw[red] (-1,0) -- ++(100:1);
\draw[red] (-1,0) -- + (-10:0.9) -- (0.2,-0.8) -- ++ (-120:1); %A
\draw[red] (-1,0) -- ++(180-10:0.3) -- ++(-110:1);
\draw (0.9,-0.4) node {$B$};
\draw (-0.5,0.7) node {$C$};
\draw (-0.7,-0.7) node {$A$};
\end{scope}
}

\begin{lemma}\libel{lemma:0605}
  Let $T$ be a $3C$ triangle of type pin-$T$.  The coordinates
  $\alpha$, $\beta$, and $x_\alpha$ satisfy the relation
\[
x_\alpha \sin(2\pi/5) \le 2\sigma (\sin(\alpha+\pi/5) - \sin(\beta+\pi/5)).
\]
In particular, $x_\alpha \le 0.0605$.
\end{lemma}

\begin{proof} Let $\v_{AB}$ be the pointer vertex of $A$ to $B$, and
  let $\v_{BC}$ be the pointer vertex of $B$ to $C$.  Let $\v$ and
  $\v_{BC}$ be the endpoints of the edge of $B$ containing $\v_{AB}$.
  We represent $T$ as in Figure~\rif{fig:cord-pint}, with the lower edge of
  $C$ along the $x$-axis.  Because $\v_{AB}$ lies on the segment between
  $\v$ and $\v_{BC}$, the $y$-coordinate $y(\v_{AB})$ of $\v_{AB}$ is
  nonpositive and lies between the $y$-coordinates $y(\v)$ and
  $y(\v_{BC})$.  We have
\begin{align*}
y(\v) &= x_\alpha \sin(2\pi/5) - 2\sigma\sin(\alpha+\pi/5)\\
y(\v_{AB}) & = -x_\beta \sin(\beta') - 2\sigma\sin(\beta+\pi/5).
\end{align*}
Using $x_\beta \sin(\beta')\ge 0$ and $y(\v) \le y(\v_{AB})$, we
obtain the claimed inequality.

Recall that $\pi/5\le \beta \le 2\pi/5$.  In particular, we have
$\sin(\alpha+\pi/5) \le 1$ and $\sin(\beta+\pi/5)\ge \sin(2\pi/5)$.
This gives
\[
x_\alpha \le 2\sigma(1/\sin(2\pi/5) - 1) < 0.0605.
\]
% checked in calcs.ml 2016/9
\end{proof}

\subsubsection{pin-$T$-junction type}\libel{sec:one-pintx}

A computer calculation\footnote{one\_pintx} shows that the longest
edge in a pint-$T$-junction is always $B\to C$.  We specialize to
$TJ$-junctions with a shared edge $B\to C$.  We use the standard
coordinates, with $(\bar x_\beta,\bar\beta)=(x_\alpha,\alpha)$ shared.
An ice-ray triangle does not have type pin-$T$.

This completes our discussion of coordinates used for computations.

\section{Appendix on Computer Calculations}\libel{appendix:cc}

The code for the computer-assisted proofs is written in Objective
Caml.  There are about five thousand lines of code, available for
download from github \cite{Git}.  The computer calculations for this
article take about 60 hours in total to run on an Intel quad 2.6 GHz
processor with 3.7GB memory, running the Ubuntu operating system.  The
hashtables occupy between  1 and 2 GB of memory.

\subsection{interval arithmetic}

To control rounding errors on the computer, we use an interval
arithmetic package for OCaml by Alliot and Gotteland, which runs on
the Linux operating system and Intel processors \cite{All}.  Intervals
are represented as pairs $(a,b)$ of floating point numbers, giving the
lower $a$ and upper $b$ endpoints of the interval.

\subsection{automatic differentiation}\libel{sec:auto-diff}

Recall that there are several ways to compute derivatives by computer,
such as numerical approximation by a difference quotient
$(f(y)-f(x))/(y-x)$, symbolic differentiation (as in computer algebra
systems), and automatic differentiation.  In this project, we
differentiate functions of a single variable using automatic
differentiation.  The value of a function and its derivative are
represented as a pair $(f,f')$ of intervals (the $1$-jet of the
function at an interval-valued point $x_0$), where $f$ is an interval
bound on the function at $x_0$, and $f'$ is an interval bound on the
derivative of the function at $x_0$.  More complex expressions can be
built from simpler expressions by extending arithmetic operations
$(+)$, $(-)$, $(*)$, $(/)$ to $1$-jets.  For example,
\begin{align*}
(f,f') + (g,g') &= (f+g,f'+g'),\\
(f,f') * (g,g') &= (f g,f g' + f' g),\\
(f,f') / (g,g') &= (f / g, (f' g - f g')/g^2).
\end{align*}
where the component-wise arithmetic operations on the right are
computed by interval arithmetic.  Automatic differentiation extends
standard functions $F$ to functions $F^D$ on $1$-jets.
For example,
\[
\op{sqrt}^D(f,f') = (\sqrt{f}, f'/(2\sqrt{f})),\quad 
\sin^D (f,f') = (\sin(f), \cos(f) f').
\]

Sections~\rif{sec:dimer-pair} and Section~\rif{sec:coords} describe
the proof of the local minimality of the ice-ray dimer.  We review
that argument here with an emphasis on automatic differentiation.
Automatic differentiation allows us to show that the ice-ray dimer is
the unique minimizer of area in an explicit neighborhood of the
ice-ray dimer.  In Section~\rif{sec:dimer-pair}, we give a curve
$\Gamma$ (in the configuration space of dimers) with parameter
$t\in\ring{R}$ that passes through the ice-ray dimer point at $t=0$.

For any point $x_0$ in an explicit neighborhood of the ice-ray dimer
in the dimer configuration space, that section describes a path from
$x_0$ to a point on the curve $\Gamma$.  Using automatic
differentiation algorithms, we show by computer that area decreases as
we move along $P$ from $x_0$ towards $\Gamma$.  Thus, the area
minimizer, lies on $\Gamma$ for some parameter $|t|\le M$.

By symmetry in the underlying geometry, it is clear that the area
function has derivative zero along $\Gamma$ at $t=0$. By taking a
second derivative with automatic differentiation, we find that the
second derivative of the area function along this curve is positive
when $|t|\le M$.  Thus, the ice-ray dimer is the unique area minimizer
on this curve within this explicit neighborhood.

We remark that many of the functions that are used in the proof of the
pentagonal ice-ray conjecture are not differentiable.  Our use of
automatic differentiation is restricted to a small neighborhood of the
ice-ray dimer, where all the relevant functions are analytic.

\subsection{meet-in-the-middle}\libel{sec:mitm}

A common algorithmic technique for reducing the time complexity of an
algorithms through greater space complexity is called {\it
  meet-in-the-middle} (MITM).  This is very closely related to the
{\it linear assembly algorithms} used in the solution to the Kepler
problem~\cite{hales2003some}. (Both methods break the problem into
subproblems of smaller complexity that are later recombined.  MITM
stores the data for recombination in a hashtable.  Linear assembly
encodes the data for recombination as linear programs.  We did not
try to solve the pentagonal ice-ray conjecture with linear programming
techniques. Such an approach might also work.)

Some introductory examples of meet-in-the-middle
algorithms can be found at the blog post \cite{mitm}.
% Cosmin Negruseri, 10 aug 2012. Coding contest trick: Meet in the middle.
% www.infoarena.ro/blog/meet-in-the-middle.
A simple example from there is to find if there are four numbers in a
given finite set $S$ of integers that sum to zero, where repetitions
of integers are allowed.  If we calculate all possible sums of four
integers, testing if each is zero, then there are $n^4$ sums, where
$n$ is the cardinality of $S$.  The MITM solution to the problem
computes and stores (in a hashset) all sums $a+b$ of unordered pairs
of elements from $S$.  We then search the hashset for a collision,
meaning a sum $a+b$ that is the negative of another sum $c+d$: $a+b =
-(c+d)$.  Any such collision gives $a+b+c+d=0$.  The MITM solution
involves the computation of $n^2$ sums $a+b$, rather than $n^4$, for a
substantial reduction in complexity.  MITM techniques have numerous
applications to cryptography, and it is there that we first
encountered the technique.  See for example,
\cite{dinur2014dissection}
% http://www.ma.huji.ac.il/~nkeller/CACM2013.pdf
which applies MITM to a general class of dissection problems,
including the Rubik's cube.

We obtain computational bounds on the area of clusters of Delaunay
triangles (or more accurately, $P$-triangles).  Each of our clusters
will be assume to consist of one triangle (called the {\it central
triangle}), flanked by one, two, or three additional triangles along
its edges.  We call the flanking triangles {\it peripheral triangles}.
The aim of each computation is to give a lower bound on the sum of the
areas of the triangles, subject to a collection of constraints.  Two
types of constraints are allowed: (1) constraints that can be
expressed in terms of a single triangle, and (2) assembly constraints.
An assembly constraint states that the central $P$-triangle fits
together with a flanking triangle.  In more detail, the central
triangle $T_0$ shares an edge and two pentagons $A$, $B$ with a
flanking triangle $T_1$.  Associated with $T_0$ are parameters
$d^0_{AB}$, $\theta^0_{ABC}$, $\theta^0_{BAC}$ giving the edge length
of the common edge with $T_1$, and the inclination angles of the
pentagons $A$ and $B$ with respect to that common edge.  Similarly,
associated with $T_1$ are parameters $d^1_{AB}$, $\theta^1_{ABC}$,
$\theta^1_{BAC}$.  The assembly constraint along the common edge is
\begin{equation}\libel{eqn:assembly}
d^0_{AB} = d^1_{AB},\quad 
\theta^0_{ABC} = -\theta^1_{ABC},\quad \theta^0_{BAC} = -\theta^1_{BAC}.
\end{equation}
The negative sign comes from the opposite orientations of the common
edge with respect to $T_0$ and $T_1$.

A cluster consisting of one central $P$-triangle and $k$ flanking
triangles is a point in a configuration space of dimension $6 + 3k$.
If we cover the configuration space with cubes of edge-length
$\epsilon$, then there are order $(1/\epsilon)^{6+3k}$ cubes.  This is
generally beyond our computational reach when $k>0$.

We can use MITM techniques to reduce to order $(1/\epsilon)^6$ cubes,
and this puts all our computations (barely) within the reach of a
laptop computer.  Specifically, we fix a edge size $\epsilon$ and
cover the configuration space of peripheral triangles by cubes of size
$\epsilon$, calculating area, edge lengths, inclination angles, and
other relevant quantities (using interval arithmetic) over each cube.

The idea of MITM is to place the peripheral triangle data into a hash
table, keyed by the variables $d^1_{AB}$, $\theta^1_{ABC}$,
$\theta^1_{BAC}$ that are shared with the central triangle through
Equation~\rif{eqn:assembly}.  We view Equation~\rif{eqn:assembly} as
the analogue of the collision condition $(a+b) = -(c+d)$ in the simple
example of MITM given above.  Of course, the variables $d^1_{AB}$
etc. are represented as intervals with floating point endpoints and
cannot be used directly as keys to a hashtable.  Instead we discretize
the keys in such a way that a collision of real numbers implies a
collision of the keys.

Once the hash is created, we divide the configuration space of central
triangles into cubes and compute the relevant quantities (edge
lengths, triangle area, and inclination angles) over each cube using
interval arithmetic.  The areas of the peripheral triangles are
recovered from the hash table and combined with the area of the
central triangle to get a lower bound on the sum of the areas of the
triangles in the cluster.

The entire process is iterated for smaller and smaller $\epsilon$
until the desired bound on the areas of the triangles in the cluster
is obtained.  Each time $\epsilon$ is made smaller, only the
peripheral cubes that were involved in a key collision with a central
cube are carried into the next iteration for subdivision into smaller
cubes.  Only the cubes with suitably small triangle area bounds are
carried into the next iteration.  In practice, to achieve our bounds
in Section~\rif{sec:calc}, the smallest $\epsilon$ that was required
was approximately $0.00024$.

\subsection{preparation of the inequalities}

Section~\rif{sec:calc} gives a sequence of inequalities that have been
proved by MITM algorithms.  We organize these calculations to permit a
more or less uniform proof of all of them.  The triangle $T_0$ is  the
central triangle.  The other triangles $T_1,\ldots$ in the cluster are the
peripheral triangles.

We prove each inequality by contradiction. Specifically, we negate the
conclusion and add it to the set of assumptions.  Then we show in each
case that the domain defined by the set of assumptions is empty. In
the computer code, we implement {\it out-of-domain} functions that
return true when the interval input lies entirely outside the given
domain.

We deform the cluster of triangles to make computations easier.  Note
that in every case except for the triangle $T_-$ in
Lemma~\rif{calc:pseudo-area3}, the peripheral triangles are all
subcritical.  We prove the lemmas of Section~\rif{sec:calc} in the
sequential order given in that section.  In particular, we may
assuming the previous lemmas to simplify what is to be proved in those
that follow.

\begin{lemma}  
  We can assume without loss of generality that the triangle $T_-$ in
  Lemma~\rif{calc:pseudo-area3} is $O2C$.
\end{lemma}

\begin{proof} 
  Assume for a contradiction, that we cannot deform into a $O2C$.  Let
  the three pentagons of $T_-$ be $A,B,C$, where $A,C$ are shared with
  the central triangle $T_0$.  By Lemmas~\rif{lemma:primary}
  and~\rif{lemma:primary2}, we can assume that $B$ has primary
  contact.  By translating $B$, we can continue to deform $T_-$,
  decreasing its area, until it is a right triangle (because $O2C$ is
  assumed not to occur).  By Lemma~\rif{calc:pseudo2}, the shared edge
  with $T_0$ has length at least $1.8$, so a right angle gives the
  proof of Lemma~\rif{calc:pseudo-area3} in this case:
\[
\area\{T_0,T_1,T_-\} > 2\ao + \area(T_-) 
\ge 2\ao + 1.8\kappa > 3\acrit + \epsM.
\]
% checked 2016/9 lemma_pseudo_area3_prep in calcs.ml.
\end{proof}

We note that an earlier lemma (Lemma~\rif{lemma:2C}) shows that
subcritical peripheral triangles can be assumed to be $O2C$. Thus, we
will assume in the computations that all peripheral triangles are
$O2C$.

In the rest of this section, we describe how each calculation has been
prepared, to reduce the dimension of the configuration space, prior to
computation.

\subsubsection{calculation~of Lemma~\rif{calc:pseudo1}} 

We repeat Lemma~\rif{calc:pseudo1}.  Let $(T_1,T_0)\in \PD$.  The edge
shared between $T_0$ and $T_1$ has length less than $1.8$.  (The
reference code [NKQNXUN] links this statement to the relevant body of
computer code.)

Negating the conclusion, we may assume that the shared edge between
triangles has length at least $1.8$.  Let $A,B,C$ be the pentagons of
$T_0$, with $A,C$ shared with $T_1$.  We drop the constraint $T_0\nRa
T_1$.  (That is, we no longer assume that the longest edge of $T_0$ is
shared with $T_1$.  Instead, we merely assume that
$\max(\dx{AB},\dx{BC})\ge 1.8$.  We may assume that $B$ has primary
contact. By symmetry, we may assume contact between $A$ and $B$.  As
long as $B$ is not $O2C$, we may continue to move $B$ (decreasing area
as always) until $\max(\dx{AB},\dx{BC})=1.8$.  The calculation reduces
to three subcases:
\begin{enumerate}
\item $T_0$ has $O2C$ contact.
\item $B$ has midpointer contact along $AB$, with $\dx{AB}=1.8$.  (We
  eliminate the case $\dx{BC}=1.8$ because this would give using
  Lemma~\rif{lemma:mid-172},
\[
\area\{T_0,T_1\} > \ao + \area(T_0) \ge \ao + \area(1.8,1.8,1.72) > 2\acrit.
\]
% checked calcs.ml lemma_pseudo_area1_prep 2016/9
\item $B$ has slider contact with $A$ and $\max(\dx{AB},\dx{BC})=1.8$.
\end{enumerate}

\subsubsection{calculation~of Lemma~\rif{calc:pseudo2}} 

We repeat the statement of Lemma~\rif{calc:pseudo2}.
Let $(T_1,T_0)\in \PD$.  The longest edge of $T_0$
has length greater than $1.8$ (reference code [RWWHLQT]).

Negating the conclusion, we may assume that all edges of $T_0$ have
length at most $1.8$.  As in the previous calculation, we let $A,C$ be
the shared pentagons, and we reduce to three cases:
\begin{enumerate}
\item $T_0$ has $O2C$ contact.
\item $B$ has midpointer contact along $AB$.  We continue to deform by
  translating $B$ until $T_0$ is long isosceles.
\item $B$ has slider contact with $A$, and $T_0$ is long isosceles.
\end{enumerate}

\subsubsection{calculation~of Lemma~\rif{calc:pseudo-area}} 

We repeat the statement of Lemma~\rif{calc:pseudo-area}.
Let $(T_1,T_0)\in\PD$.  Then $\area\{T_0,T_1\} \ge 2\acrit - \epsM$
(reference code [BXZBPJW]).

Negating the conclusion, we may assume that $\area\{T_0,T_1\}\le
2\acrit - \epsM$.  Let $A,B,C$ be the pentagons of $T_0$, with $A,C$
shared with $T_1$.  We deform $B$ until primary contact.  If it is not
$O2C$, then we continue to translate $B$.  We never encounter the long
isosceles constraint while translating $B$ because if the longest edge
and shared length have equal lengths, we have $\area\{T_0,T_1\} >
2\acrit$ by Lemmas~\rif{calc:pseudo1} and~\rif{calc:pseudo2}.  Thus,
we always reduce to $O2C$ on $T_0$.  We may continue with a squeeze
transformation (Section~\rif{sec:squeeze}) until $T_1$ is long
isosceles or $3C$.  We consider two long isosceles subcases, depending
on which of the two nonshared edges of $T_1$ has the same length as the
shared edge.

The squeeze transformation may result in $T_0$ and $T_1$ becoming
triple contact.  The calculations for dimer pairs in triple contact
were carried out with assumptions that were sufficiently relaxed to
include this case of pseudo-dimer triple contact.
% see dimer.ml run_group4.

\subsubsection{calculation~of Lemma~\rif{calc:pseudo-area3}}

We repeat the statement of Lemma~\rif{calc:pseudo-area3}.  Let
$(T_1,T_0)\in\PD$. Assume $T_0\Ra T_-$.  Then $\area\{T_0,T_1,T_-\} >
3\acrit + \epsM$ (reference code [JQMRXTH]).

As noted above, we can assume that both peripheral triangles, $T_1$
and $T_-$ are $O2C$.  Negating the conclusion, we assume that
$\area\{T_0,T_1,T_-\}\le 3\acrit+\epsM$.  This implies that
\[
\area(T_-) = \area\{T_0,T_1,T_-\} - \area\{T_0,T_1\} 
\le (3\acrit+\epsM) - (2\acrit - \epsM) = \acrit + 2\epsM.
\]
We let $A,B,C$ be the pentagons in the central triangle $T_0$, where
$A$ is shared with $T_0,T_1,T_-$ and $B$ is shared between $T_0$ and
$T_-$.  While $B$ is not in contact with another pentagon, we may
translate $B$ in a squeeze transformation, moving it along the segment
joining $\c_A$ and $\c_B$.  This decreases the areas of $T_0$ and
$T_-$.  We continue until $B$ contacts $A$ or $C$.

Renaming pentagons of the central triangle, we assume that $\bar A$
and $\bar C$ are in contact, with $\bar A$ pointing to $\bar C$.  We
consider six cases: each permutation on three letters determines a
choice for the edge of $T_0$ shared with $T_1$, and a choice of the
edge of $T_0$ shared with $T_-$.

\subsubsection{calculation~of Lemma~\rif{calc:dimer-isosc}}

We repeat the statement of Lemma~\rif{calc:dimer-isosc}.  Let
$(T_1,T_0)\in DP$.  Then $T_1$ is not both $O2C$ and long isosceles
(reference code [KUGAKIK]).

Negating the conclusion, we may assume that $T_1$ is $O2C$ and long
isosceles and that $T_0$ is $O2C$.  We perform the calculation without
further preparation.

\subsubsection{calculation~of Lemma~\rif{calc:large}} %FHB.

We repeat the statement of Lemma~\rif{calc:large}.  Let $\{T_0,T_1\}$
be given $P$-triangles (not necessarily a pseudo-dimer) such that
$\area(T_1)\le \acrit$ and $T_1\Ra T_0$.  Assume that there is a
nonshared edge $e$ of $T_0$ of length greater than $1.8$ and that the
angle is not large along $(T_0,e)$.  Then $\area\{T_0,T_1\} > 2\acrit
+ \epsM$ (reference code [FHBGHHY]).

Negating the conclusion, we may assume that $\area\{T_0,T_1\}\le
2\acrit + \epsM$.  We assume that $T_0$ has a nonshared edge of length
at least $1.8$ and that the angle is not large along that edge.  We
form five cases:
\begin{enumerate}  
\item $T_0$ and $T_1$ are both $O2C$.
\item In the last four cases, we assume that in triangle $T_0$, some
  pentagon $A$ contacts some other pentagon $C$ with $A$ pointing to
  $C$.  The shared edge is not $(A,C)$. The longest nonshared edge of
  $T_0$ is exactly $1.8$.  The angle is not large along that edge.
  The four cases come by a binary choice of the shared edge $AB$ or
  $BC$ and a binary choice of the longest edge as one of the remaining
  two edges.
\end{enumerate}

We claim that we can always deform to one of these five cases.  To see
this, assume to the contrary that none of these cases hold.  If the
$1.8$ constraint binds, we translate the outer pentagon $\bar B$ of
$T_0$ while maintaining the $1.8$ constraint until it comes into
contact with one of the shared pentagons $\bar A$ or $\bar C$. This
falls into one of the last four cases.  If the $1.8$ constraint does
not bind, then we may translate $\bar B$ until $T_0$ is $O2C$.  Here
$\{\bar A,\bar B,\bar C\}=\{A,B,C\}$ is a relabeling of the pentagons.

\subsubsection{calculation~of Lemma~\rif{calc:pent3}} %HUQ.

We repeat the statement of Lemma~\rif{calc:pent3}.  Let $T_1^i\Ra T_0$
and $\area(T_1^i)\le\acrit$ for $i=0,1$ for distinct $P$-triangles
$T_1^0$ and $T_1^1$.  Then $\area\{T_0,T_1^0,T_1^1\} > 3\acrit +
\epsM$ (reference code [HUQEJAT]).

We assume that $\area\{T_0,T_1^0,T_1^1\}\le 3\acrit + \epsM$.  We
consider four cases:
\begin{enumerate}  
\item There exists a pair $(A,C)$ of pentagons of $T_0$ in contact.
  We assume that $A$ points to $C$.  This becomes three cases
  according to the edge of $T_0$ that is not shared.
\item $T_0$ has no pentagons in contact. Let $B$ be
  the pentagon of $T_0$ that is shared with $T_1^0$ and $T_1^1$.  We
  squeeze $A$ along $(\c_A,\c_B)$ and squeeze $C$ along $(\c_B,\c_C)$.
  This allows us to assume that both $T_1^0$ and $T_1^1$ are long
  isosceles.
\end{enumerate}  

\subsubsection{calculation~of Lemma~\rif{calc:pent4}} 

We repeat the statement of Lemma~\rif{calc:pent4}.  Let $T_1^i\Ra T_0$
and $\area(T_1^i)\le\acrit$ for $i=0,1,2$ for distinct $P$-triangles
$T_1^0$, $T_1^1$, and $T_1^2$.  Then $\area\{T_0,T_1^0,T_1^1,T_1^2\} >
4\acrit$ (reference code [QPJDYDB]).

Each $T_1^i$ is $O2C$.  We carry out the calculation as a single case,
using MITM as usual as described above.  There is an $S_3$-symmetry to
the situation that we can exploit to reduce the search space.

%% END OF FILE
      
      \bibliography{pentagon} 

      \bibliographystyle{plainnat}
%\bibliography{announce_refs}

%% shell:>makeindex index/Index
%% shell:>makeindex index/Notation
%% edit index/Notation.ind to put an \indexspace before Greek entries.

%\printindex{index/Index}{General index}
%\printindex{index/Notation}{Notation index}

%\input{frontmatter}

\end{document}